\documentclass{article}
\usepackage[utf8]{inputenc}

\usepackage{amsmath}
\usepackage{amssymb}
\usepackage{amsfonts}
\usepackage{amsrefs}
\usepackage{amsthm}
\usepackage{mathtools}
\usepackage{multirow}
\usepackage{enumitem}
\usepackage{tikz-cd}
\usepackage{soul}
\usepackage{mathrsfs}
\usepackage{bbm}
\usepackage{float}
\usepackage{authblk}
\usepackage[labelfont=bf, labelsep=none]{caption}
\usepackage{abstract}
\usepackage{colortbl,dcolumn}
\usepackage{rotating}
\usepackage{afterpage}
\usepackage{subfig}
\usepackage{pifont}
\usepackage{relsize}
\usepackage{graphicx}
\usepackage{xcolor}
\usepackage{adjustbox}
\usepackage{tabularray}

\newcommand{\g}{\mathfrak{g}}
\newcommand{\gaff}{\mathfrak{\hat{g}}}
\newcommand{\aff}{\mathrm{aff}}
\newcommand{\Uq}{U_{q}(\mathfrak{g})}
\newcommand{\Uaff}{U_{q}(\mathfrak{\hat{g}})}

\newcommand{\Udash}{U'_{q}(\mathfrak{\hat{g}})}

\newcommand{\Imin}{I_{\mathrm{min}}}

\newcommand{\Pcal}{\mathcal{P}}
\newcommand{\Ycal}{\mathcal{Y}}
\newcommand{\Zcal}{\mathcal{Z}}

\newcommand{\U}{\mathcal{U}}

\newcommand{\Fcal}{\mathcal{F}}
\newcommand{\Ccal}{\mathcal{C}}

\newcommand{\pb}{\mathbf{p}}
\newcommand{\sbold}{\mathbf{s}}

\newcommand{\Vaff}{V_{\mathrm{aff}}}
\newcommand{\Laff}{L_{\mathrm{aff}}}
\newcommand{\Baff}{B_{\mathrm{aff}}}
\newcommand{\Haff}{H_{\mathrm{aff}}}
\newcommand{\Caff}{C_{\mathrm{aff}}}
\newcommand{\Zbb}{\mathbb{Z}}
\newcommand{\Nbb}{\mathbb{N}}
\newcommand{\Cbb}{\mathbb{C}}

\newcommand{\Qbb}{\mathbb{Q}}
\newcommand{\Pbar}{\overline{P}}
\newcommand{\et}{\Tilde{e}}
\newcommand{\ft}{\Tilde{f}}
\newcommand{\wt}{\mathrm{wt}}

\newcommand{\wpr}{\bigwedge^{r}\Vaff}
\newcommand{\vac}[1]{\langle {#1}\vert}
\newcommand{\Eaff}[1]{E_{#1}^{(1)}}
\newcommand{\xa}{x_{\alpha}}
\newcommand{\xb}{x_{\beta}}
\newcommand{\xt}{x_{\theta}}

\newcommand{\xmt}{x_{-\theta}}
\newcommand{\sign}{\mathrm{sign}}
\newcommand{\presign}{\mathrm{pre{\hbox{-}}sign}}

\numberwithin{equation}{section}

\usepackage{hyperref}
\hypersetup{
    pdftitle={Duncan Laurie - Young wall realizations of level 1 irreducible highest weight and Fock space crystals of quantum affine algebras in type E},
    pdfauthor={Duncan Laurie},
    hidelinks,
    pdfhighlight={/N},
    colorlinks=true,
    linkcolor=blue,
    citecolor=blue,
    urlcolor=blue,
    }
\newcommand\blfootnote[1]{%
  \begingroup
  \renewcommand\thefootnote{}\footnote{#1}%
  \addtocounter{footnote}{-1}%
  \endgroup
}

\colorlet{lightgreen}{white!55!green}
\colorlet{verylightgreen}{white!75!green}
\colorlet{lightorange}{white!65!orange}
\colorlet{verylightorange}{white!65!orange}
\colorlet{darkgreen}{black!10!green}
\colorlet{lightcyan}{white!75!cyan}

\def\cell#1{%
\ifdim#1pt=0pt\cellcolor{white}\else
\ifdim#1pt=1pt\cellcolor{verylightgreen}\else
\ifdim#1pt=2pt\cellcolor{lightorange}\else
\ifdim#1pt=-1pt\cellcolor{verylightred}\else
\cellcolor{lightcyan}\fi\fi\fi\fi
#1}

\definecolor{electricpurple}{rgb}{0.75, 0.0, 1.0}
\definecolor{shockingpink}{rgb}{0.99, 0.06, 0.75}
\definecolor{saddlebrown}{rgb}{0.55, 0.27, 0.07}
\definecolor{yellow(ncs)}{rgb}{1.0, 0.83, 0.0}

\colorlet{col0}{black}
\colorlet{col1}{red}
\colorlet{col2}{yellow(ncs)}
\colorlet{col3}{green!75!black}
\colorlet{col4}{electricpurple}
\colorlet{col5}{blue}
\colorlet{col6}{orange}
\colorlet{col7}{shockingpink}
\colorlet{col8}{saddlebrown!80!white}

\newcommand{\cross}{\text{\ding{54}}}
\renewcommand{\circle}{\text{\ding{108}}}
\renewcommand{\star}{\text{\ding{72}}}
\renewcommand{\square}{\text{\ding{110}}}

\usepackage{color}
\usepackage{listings} 
\usepackage{xspace}

\lstdefinelanguage{Sage}[]{Python}
{morekeywords={True,False,sage,singular},
sensitive=true}
\definecolor{dblackcolor}{rgb}{0.0,0.0,0.0}
\definecolor{dbluecolor}{rgb}{.01,.02,0.7}
\definecolor{dredcolor}{rgb}{0.8,0,0}
\definecolor{dgraycolor}{rgb}{0.30,0.3,0.30}

\usepackage{titling}
\begin{document}

\newtheoremstyle{thmstyleone}
  {\topsep}
  {\topsep}
  {\itshape}
  {0pt}
  {\bfseries}
  {}
  {.5em}
  {\thmname{#1}\hspace{.3em}\thmnumber{#2.}}

\newtheoremstyle{thmstyletwo}
  {\topsep}
  {\topsep}
  {\normalfont}
  {0pt}
  {\itshape}
  {}
  {.5em}
  {\thmname{#1}\hspace{.3em}\thmnumber{#2.}}

\newtheoremstyle{thmstylethree}
  {\topsep}
  {\topsep}
  {\normalfont}
  {0pt}
  {\bfseries}
  {}
  {.5em}
  {\thmname{#1}\hspace{.3em}\thmnumber{#2.}}

\theoremstyle{thmstyleone}
\newtheorem{thm}{Theorem}[section]
\newtheorem{prop}[thm]{Proposition}
\newtheorem{lem}[thm]{Lemma}
\newtheorem{cor}[thm]{Corollary}

\theoremstyle{thmstyletwo}
\newtheorem{eg}[thm]{Example}
\newtheorem{rmk}[thm]{Remark}

\theoremstyle{thmstylethree}
\newtheorem{defn}[thm]{Definition}

\raggedbottom

\title{Young wall realizations of level 1 irreducible highest weight and Fock space crystals of quantum affine algebras in type E}
\author{Duncan Laurie}
\affil{\normalsize{Mathematical Institute, University of Oxford, Andrew Wiles Building,} \\ \normalsize{Woodstock Road, Oxford, OX2 6GG, United Kingdom.}}
\date{}

\maketitle\blfootnote{E-mail: \url{duncan.laurie@maths.ox.ac.uk}}\blfootnote{ORCID: 0009-0006-9331-4835.}

\vspace{-40pt}

\begin{abstract}
    \noindent
    We construct Young wall models for the crystal bases of level $1$ irreducible highest weight representations and Fock space representations of quantum affine algebras in types $E_{6}^{(1)}$, $E_{7}^{(1)}$ and $E_{8}^{(1)}$.
    In each case, Young walls consist of coloured blocks stacked inside the relevant Young wall pattern which satisfy a certain combinatorial condition.
    Moreover the crystal structure is described entirely in terms of adding and removing blocks.
\end{abstract}

\tableofcontents

\setlength{\parindent}{0pt}

\section{Introduction} \label{Introduction}

Originally introduced by Lusztig \cite{Lusztig90} in the finite $ADE$ case and Kashiwara \cite{Kashiwara90} in classical types, and then extended to all symmetrizable types in \cite{Kashiwara91}, the theory of crystal bases provides a powerful tool for studying representations of quantum groups.
Crystal bases retain much of the structural information of their corresponding representations, and their combinatorial description frequently enables us to reduce challenging questions in representation theory -- such as the decomposition of tensor products -- to far more tractable problems in combinatorics.
\\

It is often possible to obtain concrete realizations of crystal bases, which shed light on the structure of the representations.
Producing such combinatorial models is therefore an important problem in the representation theory of quantum groups.
A key goal in this direction is to construct realizations of the crystal bases of irreducible integrable highest weight representations, since these form the connected components of the crystal bases of all integrable representations.
\\

Young wall and Young tableau models are a particularly significant class of combinatorial realizations.
Here, the vertices of a crystal are represented by stackings of coloured blocks inside certain patterns, and arrows correspond simply to adding or removing a block.
In the case of finite quantum groups, Kashiwara-Nakashima \cite{KN94} described the crystal bases of finite dimensional $\Uq$-modules in all non-exceptional types in terms of (generalised) Young tableaux.
For quantum affine algebras, the theory of perfect crystals and energy functions developed by Kang et al. \cites{KKMMNN92,KKMMNN92(2)} provides a \textit{path realization} of irreducible highest weight crystals $B(\lambda)$.
This was then used by Kang \cite{Kang03} to build Young wall models for the level $1$ irreducible highest weight crystals in types $A_{n}^{(1)}$, $A_{2n-1}^{(2)}$, $A_{2n}^{(2)}$, $B_{n}^{(1)}$, $D_{n}^{(1)}$ and $D_{n+1}^{(2)}$.
Similar constructions were later obtained by Hong-Kang-Lee \cite{HKL04} in type $C_{n}^{(1)}$, and more recently by Fan-Han-Kang-Shin \cite{FHKS23} in types $G_{2}^{(1)}$ and $D_{4}^{(3)}$.
\\

Following initial work by Kashiwara, Miwa and Stern \cites{KMS95,Stern95} in type $A$, Kashiwara-Miwa-Petersen-Yung \cite{KMPY96} gave a semi-infinite wedge construction of Fock space representations $\Fcal(\lambda)$ for quantum affine algebras.
For a $\Udash$-module $V$ satisfying various technical conditions, one considers the semi-infinite limit $\lim_{r\rightarrow\infty} \wpr$ of the $q$-exterior powers of its affinization, taken along a certain ground state vector.
This space can be endowed with the structure of a $\Uaff$-module, and is called a Fock space representation.
Among other things, Kashiwara et al. described the crystal bases $B(\Fcal(\lambda))$ of these representations in terms of the crystal basis $B$ of $V$ and its energy function $H$, and gave the decomposition of $\Fcal(\lambda)$ as a sum of irreducible representations.
Kashiwara \cite{Kashiwara02} later proved that certain \textit{good} $\Udash$-modules possess the properties required to produce Fock spaces via this process, and moreover obtained global bases for the ensuing representations.
Furthermore, just as level $1$ irreducible highest weight crystals $B(\lambda)$ have found realizations in terms of Young walls, so too have Fock space crystals.
Indeed, Kang and Kwon \cites{KK04,KK08} obtained such models for $B(\Fcal(\lambda))$ in types $A_{n}^{(1)}$, $A_{2n-1}^{(2)}$, $A_{2n}^{(2)}$, $B_{n}^{(1)}$, $D_{n}^{(1)}$ and $D_{n+1}^{(2)}$.
\\

The purpose of this paper is to construct Young wall models for both the level $1$ irreducible highest weight crystals $B(\lambda)$ and Fock space crystals $B(\Fcal(\lambda))$ of quantum affine algebras in types $\Eaff{6}$, $\Eaff{7}$ and $\Eaff{8}$.
Our starting point in each case is a Young column model for a certain level $1$ perfect crystal and its affinization.
Then for the irreducible highest weight crystals, we derive a Young wall model for $B(\lambda)$ using the path realization.
For the Fock space crystals, we first prove that our level $1$ perfect crystals are the crystal bases of good $\Udash$-modules, and then we use the description of $B(\Fcal(\lambda))$ provided by \cites{KMPY96,Kashiwara02}.
\\

This paper is organised as follows.
Section \ref{Preliminaries} establishes our notational conventions for quantum affine algebras, and outlines the necessary preliminaries regarding their representations and crystal bases.
Generalising to the abstract notion of an affine or classical crystal, we recall the definitions of perfect crystals and energy functions, and present the path realization of the irreducible highest weight crystals $B(\lambda)$.
We define the notion of a good $\Udash$-module in the sense of Kashiwara \cite{Kashiwara02}, and give the construction of the Fock space representations of $\Uaff$ together with their global and crystal bases.
\\

In Section \ref{Perfect crystal section} we introduce level $1$ perfect crystals in each of the types $\Eaff{6}$, $\Eaff{7}$ and $\Eaff{8}$ and give descriptions of their energy functions.
For types $\Eaff{6}$ and $\Eaff{7}$ we use the crystal basis of a level $0$ fundamental representation associated to a minuscule node of the affine Dynkin diagram.
However, for type $\Eaff{8}$ we require the uniform construction of Benkart-Frenkel-Kang-Lee \cite{BFKL06}.
We then provide Young column realizations for each of these crystals and their affinizations as (equivalence classes of) stackings of blocks inside a corresponding Young column pattern.
These patterns are formed by splitting each cube within an infinite strip of unit cubes into blocks, according to a collection of vertical cuts.
\\

In Section \ref{Irreducible highest weight section} we derive Young wall models for the level $1$ irreducible highest weight crystals.
We define the structure of an affine crystal on a set of \textit{reduced} Young walls, which are particular stackings of blocks inside a Young wall pattern that stabilise to a certain ground state wall.
In order to better understand their structure, we prove that these reduced Young walls satisfy a \textit{right block property} and are built on top of the ground state wall.
We then prove that this Young wall crystal is isomorphic to the path realization of $B(\lambda)$, reinterpreting the energy function condition on adjacent elements of a path as a more combinatorial condition on adjacent columns in a Young wall.
\\

Section \ref{Fock space section} constructs Young wall models for the crystal bases of level $1$ Fock space representations $\Fcal(\lambda)$.
Using our Young column realizations from Section \ref{Perfect crystal section}, we once again define a certain set of Young walls inside the relevant Young wall pattern which stabilise to the ground state wall.
Endowing this set with an affine crystal structure as in Section \ref{Irreducible highest weight section}, we obtain a crystal which provides a combinatorial model for the Fock space crystal $B(\Fcal(\lambda))$.
We then investigate the structure of these Young walls in more detail, proving a right block property (which is slightly weakened in type $\Eaff{8}$) and giving a combinatorial condition for a Young wall to lie inside our Fock space model.
\\

\textbf{Acknowledgements.}
The author would like to thank Prof. Kevin McGerty for helpful discussions and support during the preparation of this paper.
This research was financially supported by the Engineering and Physical Sciences Research Council [grant number EP/T517811/1].

\section{Preliminaries} \label{Preliminaries}

Throughout this paper we shall consider an untwisted affine Kac-Moody algebra $\gaff$ with Cartan matrix $A = (a_{ij})_{i,j\in I}$ and index set $I = \lbrace 0,\dots,n\rbrace$.
Its Cartan subalgebra $\hat{\mathfrak{h}}$ has a basis consisting of the simple coroots $\lbrace h_{i} ~\vert~ i\in I\rbrace$ together with a scaling element $d$.
The corresponding simple roots $\alpha_{i}$ and fundamental weights $\Lambda_{i}$ for each $i\in I$ lie inside the dual space $\hat{\mathfrak{h}}^{*}$.
\\

Each node $i\in I$ of the Dynkin diagram associated to $\gaff$ has numerical labels $a_{i}$ and $a_{i}^{\vee}$ as given in \cite{Kac90}*{Chapter 4}, from which we can define the null root $\delta = \sum_{i\in I} a_{i}\alpha_{i}$ and canonical central element $c = \sum_{i\in I} a_{i}^{\vee}h_{i}$.
The affine Cartan matrix $A$ is symmetrized by the diagonal matrix $D = \mathrm{diag}(d_{0},\dots,d_{n})$ with each $d_{i} = a_{i}^{\vee}a_{i}^{-1}$, which is to say that the product $DA$ is symmetric.
\\

A vertex $i\in I$ is minuscule if it is sent to $0$ by some automorphism of the affine Dynkin diagram (or equivalently if $a_{i} = 1$), and we denote the set of minuscule nodes by $\Imin$.
An automorphism is inner if it fixes the $0$ vertex, and thus restricts to an automorphism of the finite Dynkin diagram.
The outer automorphism group $\Omega$ is then the quotient of the usual automorphism group by the subgroup of inner automorphisms, and therefore has elements indexed by $\Imin$.
In particular, for each $i\in\Imin$ we let $\pi_{i}$ be the corresponding element of $\Omega$, which is uniquely determined by the condition $\pi_{i}(0) = i$.
\\

The natural pairing between $\hat{\mathfrak{h}}$ and $\hat{\mathfrak{h}}^{*}$ is given by
$\langle\Lambda_{i},h_{j}\rangle = \delta_{ij}$,
$\langle\Lambda_{i},d\rangle = \langle\delta,h_{j}\rangle = 0$
and $\langle\delta,d\rangle = 1$,
which allows us to define the following:
\begin{itemize}
    \item the affine weight lattice $P = \bigoplus_{i\in I}\Zbb\Lambda_{i}\oplus\Zbb\delta$,
    \item the dual lattice $P^{\vee} = \bigoplus_{i\in I}\Zbb h_{i}\oplus\Zbb d$,
    \item the set of dominant affine integral weights $P^{+} = \bigoplus_{i\in I}\Nbb\Lambda_{i}\oplus\Zbb\delta$,
    \item the classical weight lattice $\Pbar = \bigoplus_{i\in I}\Zbb\Lambda_{i}$ which can be viewed both as a sublattice and a quotient of $P$,
    \item the set of dominant classical integral weights $\Pbar^{+} = \bigoplus_{i\in I}\Nbb\Lambda_{i}$.
\end{itemize}
The affine Weyl group $W = \langle s_{i} ~\vert~ i\in I\rangle$ acts on both $P$ and $\Pbar$ via $s_{i}(x) = x - \langle x,h_{i}\rangle\alpha_{i}$.
The level of an affine or classical weight $\lambda$ is given by its pairing $\langle\lambda,c\rangle$ with the canonical central element, and is invariant under the Weyl group action.
\\

The $q$-integers, $q$-factorials and $q$-binomial coefficients are defined as
\begin{align*}
    [s] = \frac{q^{s}-q^{-s}}{q-q^{-1}},
    \quad
    [s]! = \prod_{\ell=1}^{s} [\ell],
    \quad
    \begin{bmatrix}{s}\\ {r}\end{bmatrix} = \frac{[s]!}{[s-r]!\,[s]!},
\end{align*}
respectively for all non-negative integers $s\geq r$.
More generally, we furthermore let $\{x\} = (x - x^{-1})/(q - q^{-1})$ and
\begin{align*}
    \begin{Bmatrix}
        x \\
        r
    \end{Bmatrix}
    =
    \frac{\{x\}\{q^{-1}x\}\dots\{q^{1-r}x\}}{[r]!}.
\end{align*}

Our affine Lie algebras $\gaff$ shall be of simply laced type throughout, and so we can define the quantum affine algebra $\Uaff$ to be the unital associative $\Qbb(q)$-algebra generated by $e_{i}$, $f_{i}$ ($i\in I$) and $q^{h}$ ($h\in P^{\vee}$) subject to the relations
\begin{itemize}
    \item $\displaystyle q^{0} = 1, \quad q^{h}q^{h'} = q^{h+h'}$,
    \item $\displaystyle q^{h}e_{i}q^{-h} = q^{\langle\alpha_{i},h\rangle}e_{i}, \quad q^{h}f_{i}q^{-h} = q^{-\langle\alpha_{i},h\rangle}f_{i}$,
    \item $\displaystyle [e_{i},f_{i}] = \frac{\delta_{ij}}{q-q^{-1}} (k_{i} - k_{i}^{-1})$,
    \item $\displaystyle \sum_{s=0}^{1-a_{ij}} (-1)^{s} e_{i}^{(s)} e_{j} e_{i}^{(1-a_{ij}-s)} = 0$ whenever $i\not= j$,
    \item $\displaystyle \sum_{s=0}^{1-a_{ij}} (-1)^{s} f_{i}^{(s)} f_{j} f_{i}^{(1-a_{ij}-s)} = 0$ whenever $i\not= j$,
\end{itemize}
where $e_{i}^{(s)} = e_{i}^{s}/[s]!$, $f_{i}^{(s)} = f_{i}^{s}/[s]!$ and $k_{i} = q^{h_{i}}$ for each $i\in I$.
We let $\Udash$ be the subalgebra generated by all $e_{i}$, $f_{i}$ and $k_{i}^{\pm 1}$, which can alternatively be obtained by replacing $P^{\vee}$ with $\Pbar^{\vee} = \bigoplus_{i\in I}\Zbb h_{i}$ in the above.
Both $\Uaff$ and $\Udash$ have a coproduct $\Delta$ given by
\begin{align} \label{coproduct}
    \Delta(q^{h}) = q^{h}\otimes q^{h}, \,\,
    \Delta(e_{i}) = e_{i}\otimes k_{i}^{-1} + 1\otimes e_{i}, \,\,
    \Delta(f_{i}) = f_{i}\otimes 1 + k_{i}\otimes f_{i}.
\end{align}

\begin{rmk}
    Our coproduct $\Delta$ is the same as the one used for example in \cites{KKMMNN92,Kashiwara02}, while that of \cite{KMPY96} is obtained by exchanging the tensor factors in (\ref{coproduct}).
    We refer the reader to \cite{KMPY96}*{\S 2.2} for a nice explanation of how the various coproducts and tensor products of crystals for quantum affine algebras relate to one another.
\end{rmk}

\subsection{Representations of quantum affine algebras}

Recall that a $\Uaff$-module $V$ is integrable if it decomposes as a direct sum of its weight spaces
$V_{\lambda} = \lbrace u\in V ~\vert~ q^{h}\cdot u = q^{\langle \lambda,h\rangle}u~\mathrm{for~all}~h\in P^{\vee}\rbrace$ and all $e_{i}$ and $f_{i}$ generators act locally nilpotently.
An element $u\in V$ is said to be extremal if there exists a set of vectors $\lbrace u_{w}\rbrace_{w\in W}$ such that
\begin{itemize}
    \item $u_{e} = u$,
    \item if $\langle w\lambda,h_{i}\rangle\geq 0$ then $e_{i}u_{w}=0$ and $f_{i}^{(\langle w\lambda,h_{i}\rangle)}u_{w}=u_{s_{i}w}$,
    \item if $\langle w\lambda,h_{i}\rangle\leq 0$ then $f_{i}u_{w}=0$ and $e_{i}^{(-\langle w\lambda,h_{i}\rangle)}u_{w}=u_{s_{i}w}$.
\end{itemize}
If such a set exists then it must be unique, and each $u_{w}$ spans $V_{w\lambda}$.
In this case, we say that $V$ is an extremal weight module.
For each $\lambda\in P$ let $V(\lambda)$ be the $\Uaff$-module generated by a non-zero vector $u_{\lambda}$, subject only to the condition that it is an extremal vector of weight $\lambda$.
In particular, if $\lambda$ is dominant then $V(\lambda)$ is the irreducible highest weight module of highest weight $\lambda$.
\\

Consider $V(\varpi_{i})$ where $\varpi_{i} = \Lambda_{i} - a_{i}^{\vee}\Lambda_{0}$ is the $i$th level $0$ fundamental weight in $P$.
Since $\varpi_{i} + \Zbb \delta \subset W\varpi_{i}$ there exists a unique $\Udash$-module automorphism $z$ of $V(\varpi_{i})$ sending $u_{\varpi_{i}}$ to $u_{\varpi_{i}+\delta}$.
The quotient $W(\varpi_{i}) = V(\varpi_{i})/(z-1)V(\varpi_{i})$ is an irreducible integrable $\Udash$-module, called the $i$th level $0$ fundamental representation (see \cite{Kashiwara02}*{\S 5.3} for more details).
\\

Conversely, for each integrable $\Udash$-module $V = \bigoplus_{\lambda\in\Pbar}V_{\lambda}$ we can define the structure of a $\Uaff$-module on its affinization $\Vaff = \Cbb[z,z^{-1}]\otimes V$ as follows.
The actions of the $q^{h}$ are determined by setting $(\Vaff)_{\lambda+n\delta} = z^{n}V_{\lambda}$ for all $\lambda\in\Pbar$ and $n\in\Zbb$, while all $e_{i}$ and $f_{i}$ act by $z^{\delta_{i0}}\otimes e_{i}$ and $z^{-\delta_{i0}}\otimes f_{i}$ respectively.
It is easy to see that passing $W(\varpi_{i})$ through this affinization process recovers the extremal weight module $V(\varpi_{i})$.

\subsection{Crystals of quantum affine algebras} \label{crystals preliminaries}

We start by briefly recalling the definition of the crystal basis of an integrable $\Uaff$-module.
Note that the corresponding definitions for representations of $\Udash$ are obtained simply by replacing the weight lattice $P$ with $\Pbar$.
For a more detailed introduction to crystal bases, we refer the reader to \cites{Kashiwara91,Kashiwara95}.
\\

Consider an integrable $\Uaff$-module $V = \bigoplus_{\lambda\in P} V_{\lambda}$.
For each $i\in I$, we can write any $u\in V_{\lambda}$ uniquely as a sum
\begin{align*}
    u = \sum f_{i}^{(n)}u_{n}
\end{align*}
over integers $n\geq\max\lbrace-\langle\lambda,h_{i}\rangle,0\rbrace$ where each $u_{n}\in V_{\lambda+n\alpha_{i}}\cap\ker(e_{i})$.
We then define linear endomorphisms $\et_{i}$ and $\ft_{i}$ of $V$ by
\begin{align*}
    \et_{i}u = \sum f_{i}^{(n-1)}u_{n}, \quad
    \ft_{i}u = \sum f_{i}^{(n+1)}u_{n},
\end{align*}
which however do not in general respect the $\Uaff$-module structure.
Let $A$ be the subring of functions in $\Qbb(q)$ that are regular at $q=0$.
\begin{defn}
    A crystal lattice $L$ of an integrable $\Uaff$-module $V$ is a free $A$-submodule of $V$ such that
    \begin{itemize}
        \item $V \cong \Qbb\otimes_{A} L$,
        \item $L = \bigoplus_{\lambda\in P}L_{\lambda}$ where $L_{\lambda} = L\cap V_{\lambda}$,
        \item $\et_{i}L\subset L$ and $\ft_{i}L\subset L$ for all $i\in I$.
    \end{itemize}
\end{defn}

\begin{defn}
    A crystal basis of an integrable $\Uaff$-module $V$ is a pair $(L,B)$ such that
    \begin{itemize}
        \item $L$ is a crystal lattice of $V$,
        \item $B$ is a basis of $L/qL$ as a vector space over $A/qA\cong\Qbb$,
        \item $B = \bigsqcup_{\lambda\in P}B_{\lambda}$ where $B_{\lambda} = B\cap L_{\lambda}/qL_{\lambda}$,
        \item $\et_{i}B\subset B\sqcup\lbrace 0\rbrace$ and $\ft_{i}B\subset B\sqcup\lbrace 0\rbrace$ for all $i\in I$,
        \item $\ft_{i}b = b'$ if and only if $b = \et_{i}b'$ for all $b,b'\in B$ and $i\in I$.
    \end{itemize}
\end{defn}

To avoid confusion, we shall occasionally denote the crystal basis of a representation $V$ by $(L(V),B(V))$.
It is clear that if a collection $\lbrace V_{j}\rbrace_{j\in J}$ of $\Uaff$-modules each has a crystal basis, then their direct sum $\bigoplus_{j\in J} V_{j}$ has crystal basis $(\bigoplus L(V_{j}), \bigoplus B(V_{j}))$.
\\

Every crystal basis has associated ($I$-coloured, directed) crystal graph, formed on the vertex set $B$ by including an $i$-arrow $b \xrightarrow{i} b'$ whenever $\ft_{i}b = b'$.
Here, a connected component of the spanning subgraph containing only $i$-arrows is called an $i$-string.
Moreover, we can define a weight function $\wt : B \rightarrow P$ by sending elements of $B_{\lambda}$ to $\lambda$, as well as maps $\varepsilon_{i},\varphi_{i} : B \rightarrow \Nbb$ for each $i\in I$ given by
\begin{align} \label{varphi and varepsilon maximum definitions}
    \varepsilon_{i}(b) = \max\lbrace n ~\vert~ \et_{i}^{n}b\neq 0\rbrace, \quad
    \varphi_{i}(b) = \max\lbrace n ~\vert~ \ft_{i}^{n}b\neq 0\rbrace.
\end{align}
It was proven in \cite{Kashiwara94} that for any $\lambda\in P$, the irreducible highest weight representation $V(\lambda)$ has a crystal basis $(L(\lambda),B(\lambda))$.
In particular, when $\lambda$ is dominant this has the following description.
\begin{thm}
    \cite{Kashiwara91} For each $\lambda\in P^{+}$ the lattice $L(\lambda)$ is the smallest $A$-submodule of $V(\lambda)$ containing $u_{\lambda}$, and $B(\lambda)$ is the set of all non-zero vectors in $L(\lambda)/qL(\lambda)$ of the form $\ft_{i_{\ell}}\dots\ft_{i_{1}}u_{\lambda}$.
\end{thm}

We can generalise the idea of these crystal bases for integrable representations of $\Uaff$ and $\Udash$ to the following more abstract notion.

\begin{defn}
    An affine (resp. classical) crystal is a set $B$ together with maps $\et_{i},\ft_{i} : B \rightarrow B\sqcup\lbrace 0\rbrace$ and $\varepsilon_{i},\varphi_{i} : B \rightarrow \Zbb\cup\lbrace -\infty\rbrace$ for each $i\in I$, and a weight function $\wt : B \rightarrow P$ (resp. $\wt : B \rightarrow \Pbar$), such that
    \begin{itemize}
        \item $\varphi_{i}(b) - \varepsilon_{i}(b) = \langle h_{i},\wt(b)\rangle$,
        \item $\wt(\et_{i}b) = \wt(b) + \alpha_{i}$ if $\et_{i}b\in B$,
        \item $\wt(\ft_{i}b) = \wt(b) - \alpha_{i}$ if $\ft_{i}b\in B$,
        \item $\varepsilon_{i}(\et_{i}b) = \varepsilon_{i}(b) - 1$, $\varphi_{i}(\et_{i}b) = \varphi_{i}(b) + 1$ if $\et_{i}b\in B$,
        \item $\varepsilon_{i}(\ft_{i}b) = \varepsilon_{i}(b) + 1$, $\varphi_{i}(\ft_{i}b) = \varphi_{i}(b) - 1$ if $\ft_{i}b\in B$,
        \item $\ft_{i}b = b'$ if and only if $b = \et_{i}b'$,
        \item if $\varphi_{i}(b) = -\infty$ then $\et_{i}b = \ft_{i}b = 0$,
    \end{itemize}
    for all $b,b'\in B$ and $i\in I$.
\end{defn}

Similarly to how we defined extremal vectors of quantum affine representations, an element $b\in B$ is extremal if there exists a subset $\lbrace b_{w}\rbrace_{w\in W}$ of $B$ such that
\begin{itemize}
    \item $b_{e} = b$,
    \item if $\langle w\lambda,h_{i}\rangle\geq 0$ then $\et_{i}b_{w}=0$ and $\ft_{i}^{\langle w\lambda,h_{i}\rangle}b_{w}=b_{s_{i}w}$,
    \item if $\langle w\lambda,h_{i}\rangle\leq 0$ then $\ft_{i}b_{w}=0$ and $\et_{i}^{-\langle w\lambda,h_{i}\rangle}b_{w}=b_{s_{i}w}$.
\end{itemize}

An element of a crystal is therefore extremal if and only if it lies at the start or end of each $i$-string it is contained in.
Just as we introduced the idea of affinizing $\Udash$-modules in the previous subsection, classical crystals can also be affinized.

\begin{defn}
    The affinization $\Baff$ of a classical crystal $B$ is the set $\lbrace z^{n}b ~\vert~ n\in\Zbb, b\in B\rbrace$, with the structure of an affine crystal as follows:
    \begin{itemize}
        \item $\wt(z^{n}b) = \wt(b) + n\delta$,
        \item $\varepsilon_{i}(z^{n}b) = \varepsilon_{i}(b)$,
        \item $\varphi_{i}(z^{n}b) = \varphi_{i}(b)$,
        \item $\et_{i}(z^{n}b) = z^{n+\delta_{i0}}(\et_{i}b)$,
        \item $\ft_{i}(z^{n}b) = z^{n-\delta_{i0}}(\ft_{i}b)$.
    \end{itemize}
\end{defn}

We remark that if a $\Udash$-module $V$ has crystal basis $(L,B)$ then its affinization $\Vaff$ has crystal basis $(\Laff,\Baff)$, where $\Laff$ is defined similarly.
In particular, this applies to the crystal bases of the level $0$ fundamental representations $W(\varpi_{i})$ and extremal weight modules $V(\varpi_{i})$.
\\

A crystal morphism $\Psi : B \rightarrow B'$ between two affine or classical crystals is a map $\Psi : B\sqcup\lbrace 0\rbrace \rightarrow B'\sqcup\lbrace 0\rbrace$ satisfying the following conditions for all $b,b'\in B$ and $i\in I$:
\begin{itemize}
    \item $\Psi(0) = 0$,
    \item if $\Psi(b)\in B'$ then $\wt(\Psi(b)) = \wt(b)$, $\varepsilon_{i}(\Psi(b)) = \varepsilon_{i}(b)$ and $\varphi_{i}(\Psi(b)) = \varphi_{i}(b)$,
    \item if $\Psi(b),\Psi(b')\in B'$ and $\ft_{i}b = b'$ then $\ft_{i}\Psi(b) = \Psi(b')$ and $\Psi(b) = \et_{i}\Psi(b')$.
\end{itemize}
And $\Psi$ is an isomorphism if $\Psi : B\sqcup\lbrace 0\rbrace \rightarrow B'\sqcup\lbrace 0\rbrace$ is a bijection.
Given an (iso)morphism $\Psi : B \rightarrow B'$ of classical crystals, $\Psi_{\mathrm{aff}}(z^{n}b) := z^{n}\Psi(b)$ defines an (iso)morphism $\Psi_{\mathrm{aff}} : \Baff \rightarrow \Baff'$ between their affinizations.
\\

Two affine or classical crystals $B$ and $B'$ have a tensor product $B\otimes B'$ defined on the set $B\times B'$ with crystal structure given by
\begin{align} \label{tensor product of crystals}
\begin{split}
    \et_{i}(b\otimes b') &=
    \begin{cases}
    \et_{i}b\otimes b' & \mathrm{if~}\varphi_{i}(b)\geq \varepsilon_{i}(b'), \\
    b\otimes \et_{i}b' & \mathrm{if~}\varphi_{i}(b) < \varepsilon_{i}(b'),
    \end{cases} \\
    \ft_{i}(b\otimes b') &=
    \begin{cases}
    \ft_{i}b\otimes b' & \mathrm{if~}\varphi_{i}(b) > \varepsilon_{i}(b'), \\
    b\otimes \ft_{i}b' & \mathrm{if~}\varphi_{i}(b) \leq \varepsilon_{i}(b'),
    \end{cases} \\
    \wt(b\otimes b') &= \wt(b) + \wt(b'), \\
    \varepsilon_{i}(b\otimes b') &= \max(\varepsilon_{i}(b),\varepsilon_{i}(b') - \langle h_{i},\wt(b)\rangle), \\
    \varphi_{i}(b\otimes b') &= \max(\varphi_{i}(b'), \varphi_{i}(b) + \langle h_{i},\wt(b')\rangle).
\end{split}
\end{align}
\begin{prop}
    If $(L,B)$ and $(L',B')$ are the crystal bases of $\Uaff$ or $\Udash$ modules $V_{1}$ and $V_{2}$, then $(L\otimes_{A} L',B\otimes B')$ is the crystal basis of $V_{1}\otimes V_{2}$.
\end{prop}

We next introduce the notion of a perfect crystal as developed by Kang et al. in \cites{KKMMNN92,KKMMNN92(2)}.
For each element $b$ of a classical crystal $B$, define associated weights $\varepsilon(b) = \sum_{i\in I} \varepsilon_{i}(b)\Lambda_{i}$ and $\varphi(b) = \sum_{i\in I} \varphi_{i}(b)\Lambda_{i}$ in $\Pbar$.

\begin{defn} \label{perfect crystal definition}
    A perfect crystal of level $\ell\in\Zbb_{>0}$ is a classical crystal $B$ such that
    \begin{itemize}
        \item there is a finite dimensional irreducible $\Udash$-module whose crystal graph is isomorphic to $B$,
        \item $B\otimes B$ is connected,
        \item there exists some weight $\mu\in\Pbar$ with $\wt(B) \subset \mu + \sum_{i\in I\setminus\lbrace 0\rbrace} \Zbb_{\leq 0}\alpha_{i}$ and $\#\lbrace b\in B ~\vert~ \wt(b) = \mu\rbrace = 1$,
        \item $\langle\varepsilon(b),c\rangle\geq \ell$ for all $b\in B$,
        \item for any $\lambda\in\Pbar^{+}$ with $\langle\lambda,c\rangle = \ell$ there exist unique $b^{\lambda},b_{\lambda}\in B$ with $\varepsilon(b^{\lambda}) = \varphi(b_{\lambda}) = \lambda$.
    \end{itemize}
\end{defn}

The importance of perfect crystals is demonstrated by the following results.

\begin{thm}
    \cite{KKMMNN92} Let $B$ be a perfect crystal of level $\ell$.
    If $\lambda\in\Pbar^{+}$ has $\langle\lambda,c\rangle = \ell$ then there is a unique isomorphism of classical crystals
    \begin{align*}
        \Psi : B(\lambda) \xrightarrow{\sim} B(\varepsilon(b_{\lambda}))\otimes B
    \end{align*}
    given by $u_{\lambda} \mapsto u_{\varepsilon(b_{\lambda})}\otimes b_{\lambda}$, where $u_{\lambda}$ is the highest weight element of $B(\lambda)$ and $b_{\lambda}\in B$ is as in Definition \ref{perfect crystal definition}.
\end{thm}

Letting $\lambda_{0} = \lambda$, $b_{r} = b_{\lambda_{r}}$ and $\lambda_{r+1} = \varepsilon(b_{\lambda_{r}})$ for all $r\in\Nbb$, we can then obtain a sequence of crystal isomorphisms by applying the theorem repeatedly:
\begin{alignat*}{6}
    &B(\lambda) ~~&\longrightarrow&~~ &B(\lambda_{1})\otimes B& ~~&\longrightarrow&~~ &B(\lambda_{2})\otimes B\otimes B& ~~&\longrightarrow&~~ \cdots \\
    &~~u_{\lambda} ~~&\longmapsto&~~ &u_{\lambda_{1}}\otimes b_{0}& ~~&\longmapsto&~~ &u_{\lambda_{2}}\otimes b_{1}\otimes b_{0}& ~~&\longmapsto&~~ \cdots
\end{alignat*}
The infinite sequence $\pb_{\lambda} = (b_{r})_{r=0}^{\infty} = \dots \otimes b_{2}\otimes b_{1}\otimes b_{0}$ is called the ground state sequence of weight $\lambda$.
The set of $\lambda$-paths
\begin{align*}
    \Pcal(\lambda) = \lbrace \pb = (p_{r})_{r=0}^{\infty} = \dots \otimes p_{2}\otimes p_{1}\otimes p_{0} ~\vert~ \mathrm{all~} p_{r}\in B,~ p_{r} = b_{r} \mathrm{~for~} r \gg 0 \rbrace
\end{align*}
can be endowed with the structure of a classical crystal as follows.
If $\pb\in\Pcal(\lambda)$ has $p_{r}=b_{r}$ for all $r\geq k$ then set
\begin{align} \label{crystal structure on paths}
\begin{split}
    &\wt(\pb) = \lambda_{k} + \wt(\pb'), \\
    &\et_{i}\pb = \dots\otimes p_{k+1}\otimes\et_{i}(p_{k}\otimes\dots\otimes p_{0}), \\
    &\ft_{i}\pb = \dots\otimes p_{k+1}\otimes\ft_{i}(p_{k}\otimes\dots\otimes p_{0}), \\
    &\varepsilon_{i}(\pb) = \max(\varepsilon_{i}(\pb') - \varphi_{i}(b_{k}),0), \\
    &\varphi_{i}(\pb) = \varphi_{i}(\pb') + \max(\varphi_{i}(b_{k}) - \varepsilon_{i}(\pb'),0),
\end{split}
\end{align}
where $\pb' = p_{k-1}\otimes\dots\otimes p_{0}$.

\begin{thm} \label{path realization theorem}
    \cite{KKMMNN92} There is an isomorphism $\Psi_{\lambda} : B(\lambda) \xrightarrow{\sim} \Pcal(\lambda)$ of classical crystals given by $u_{\lambda} \mapsto \pb_{\lambda}$.
\end{thm}

This is called the path realization of the irreducible highest weight crystal $B(\lambda)$.

\subsection{Energy functions}

\begin{defn}
    An energy function for an affine or classical crystal $B$ is a map $H : B\otimes B \rightarrow \Zbb$ satisfying the condition
    \begin{align*}
        H(\ft_{i}(b\otimes b')) =
        \begin{cases}
            H(b\otimes b') & \mathrm{if~} i\not= 0, \\
            H(b\otimes b') + 1 & \mathrm{if~} i = 0 \mathrm{~and~} \ft_{i}(b\otimes b') = (\ft_{i}b)\otimes b' \not= 0, \\
            H(b\otimes b') - 1 & \mathrm{if~} i = 0 \mathrm{~and~} \ft_{i}(b\otimes b') = b\otimes (\ft_{i}b') \not= 0. \\
        \end{cases}
    \end{align*}
\end{defn}

Such an $H$ is therefore determined up to constant shift provided that $B\otimes B$ is connected, and can often be normalised by specifying that $H(b \otimes b) = 0$ for any extremal element $b\in B$.
Moreover existence for all perfect crystals was shown in \cite{KKMMNN92}.
\\

Given an energy function $H$ on a classical crystal $B$, the affine energy function $\Haff$ for $\Baff$ is defined by
\begin{align} \label{Haff definition}
    \Haff(z^{m}a\otimes z^{n}b) = H(a\otimes b) + m - n
\end{align}
for all $m,n\in\Zbb$ and $a,b\in B$.

\begin{rmk} \label{energy function difference remark}
    It is crucial to note that while our definitions of energy functions and affine energy functions match those of \cite{Kashiwara02}, they are equal to \emph{minus} those of references such as \cites{BFKL06,FHKS23,KKMMNN92,KKMMNN92(2),KMPY96}.
\end{rmk}

\begin{lem} \label{Haff constant on components}
    The affine energy function $\Haff$ is constant on each connected component of $\Baff\otimes\Baff$.
\end{lem}
\begin{proof}
    This is exactly the same as in \cite{FHKS23}*{\S 3.3} after adjusting for the sign difference mentioned in Remark \ref{energy function difference remark}.
\end{proof}

Returning to the path realization of the previous subsection, the notion of an energy function $H$ on $B$ allows us to endow $\Pcal(\lambda)$ with the structure of an affine crystal by replacing the weight function in (\ref{crystal structure on paths}) with
\begin{align} \label{affine weights on paths}
    \wt(\pb) = \lambda_{k} + \wt(\pb') + \delta\,\sum_{r=0}^{\infty} (r+1)(H(p_{r+1}\otimes p_{r})\! -\! H(b_{r+1}\otimes b_{r})).
\end{align}
Theorem \ref{path realization theorem} can then be strengthened to an isomorphism of affine crystals by viewing $\lambda$ inside $P^{+}$ rather than $\Pbar^{+}$.

\subsection{Fock space representations} \label{Fock space preliminaries}

Kashiwara-Miwa-Petersen-Yung \cite{KMPY96} constructed Fock space representations of $\Uaff$ as semi-infinite limits of $q$-wedge products $\wpr$ along a certain ground state vector, where $V$ is a $\Udash$-module satisfying particular assumptions.
Kashiwara \cite{Kashiwara02} later showed that these assumptions are satisfied by so-called `good modules' and further studied the resulting Fock spaces, for example obtaining global bases for the representations.
In this subsection we shall recall the definition of a good $\Udash$-module and the construction of the Fock spaces, as well as some fundamental results on their structure.
\\

We say that a subset $J\subset I$ is of finite type if $\g_{J} = \langle e_{i},f_{i},k_{i}^{\pm 1} ~\vert~ i\in J\rangle$ is a finite dimensional semisimple Lie algebra.
An affine or classical crystal $B$ is regular if for any finite type $J\subset I$, if we forget all arrows with labels in $I\setminus J$ then $B$ is isomorphic to the crystal basis of an integrable $\U_{q}(\g_{J})$-module.
By \cite{KKMMNN92}*{Proposition 2.4.4} this is equivalent to the same condition but with $J$ limited to subsets of order $2$.

\begin{defn} \label{simple crystal definition}
    A finite regular classical crystal $B$ whose weights are of level $0$ is simple if there exists some $\lambda\in\Pbar$ such that $\vert B_{\lambda}\vert = 1$, and the weight of any extremal element of $B$ is contained in its Weyl group orbit $W\lambda$.
\end{defn}

We shall also require the notion of a global basis.
Recall that $A$ is the subring of rational functions in $\Qbb(q)$ that are regular at $q=0$.
Let $-$ be the automorphism of $\Qbb(q)$ sending $q\mapsto q^{-1}$.
Then $A_{\infty} = \overline{A}$ is the subring of functions in $\Qbb(q)$ which are regular at $q=\infty$.

\begin{defn}
    A balanced triple $(L,L_{\infty},V_{\Qbb})$ for a $\Qbb(q)$-vector space $V$ consists of
    \begin{itemize}
        \item an $A$-submodule $L$ of $V$,
        \item an $A_{\infty}$-submodule $L_{\infty}$ of $V$,
        \item a $\Qbb[q,q^{-1}]$-submodule $V_{\Qbb}$ of $V$,
    \end{itemize}
    each of which generates $V$ as a $\Qbb(q)$-vector space, such that the following equivalent conditions are satisfied:
    \begin{itemize}
        \item the natural projection $E \rightarrow L/qL$ is an isomorphism,
        \item the natural projection $E \rightarrow L_{\infty}/q^{-1}L_{\infty}$ is an isomorphism,
        \item $(L\cap V_{\Qbb})\oplus (q^{-1}L_{\infty}\cap V_{\Qbb}) \rightarrow V_{\Qbb}$ is an isomorphism,
        \item $A\otimes_{\Qbb} E \rightarrow L$,
        $\overline{A}\otimes_{\Qbb} E \rightarrow L_{\infty}$,
        $\Qbb[q,q^{-1}]\otimes_{\Qbb} E \rightarrow V_{\Qbb}$
        and
        $\Qbb(q)\otimes_{\Qbb} E \rightarrow V$
        are all isomorphisms,
    \end{itemize}
    where $E = L\cap L_{\infty}\cap V_{\Qbb}$.
\end{defn}

We shall also denote by $-$ the ring automorphism of $\Udash$ which sends $q$, $q^{h}$, $e_{i}$ and $f_{i}$ to $q^{-1}$, $q^{-h}$, $e_{i}$ and $f_{i}$ respectively.
A bar involution of a $\Udash$-module $V$ is then an involution $-$ satisfying $\overline{a\cdot u} = \overline{a}\cdot\overline{u}$ for all $a\in\Udash$ and $u\in V$.
\\

Let $\Udash_{\Qbb}$ be the $\Qbb[q,q^{-1}]$-subalgebra of $\Udash$ generated by
$\genfrac\{\}{0pt}{2}{q^{h}}{n}$, $e_{i}^{(n)}$ and $f_{i}^{(n)}$ for all $h\in\Pbar^{\vee}$, $i\in I$ and $n\geq 0$.
\begin{defn}
    An integrable $\Udash$-module $V$ with crystal basis $(L,B)$ has a global basis if there exists a $\Udash_{\Qbb}$-submodule $V_{\Qbb}$ such that
    \begin{itemize}
        \item $\overline{V_{\Qbb}} = V_{\Qbb}$ and $u - \overline{u} \in (q-1)V_{\Qbb}$ for all $u\in V_{\Qbb}$,
        \item $(L,\overline{L},V_{\Qbb})$ is a balanced triple.
    \end{itemize}
\end{defn}

In this case, letting $G : L/qL \xrightarrow{\sim} E$ be the inverse of the natural projection, $\lbrace G(b) ~\vert~ b\in B\rbrace$ forms a basis for $V$ called the (lower) global basis.
Such bases satisfy the following nice properties, among many others:
\begin{itemize}
    \item $\overline{G(b)} = G(b)$ for all $b\in B$,
    \item for all $i\in I$ and $b\in B$ we have
    \begin{align*}
        f_{i}G(b) &=
        [\varepsilon_{i}(b) + 1]
        G(\ft_{i}b) + \sum_{\varphi_{i}(b')\geq\varphi_{i}(b)} F^{i}_{b,b'}G(b'), \\
        e_{i}G(b) &=
        [\varphi_{i}(b) + 1]
        G(\et_{i}b) + \sum_{\varepsilon_{i}(b')\geq\varepsilon_{i}(b)} E^{i}_{b,b'}G(b'),
    \end{align*}
    where the coefficients above satisfy
    \begin{alignat*}{6}
        &F^{i}_{b,b'}& &\in&~&q^{1-\varphi_{i}(b')}\Zbb[q]\cup q^{\varphi_{i}(b')-1}\Zbb[q^{-1}],& \\
        &E^{i}_{b,b'}& &\in&~&q^{1-\varepsilon_{i}(b')}\Zbb[q]\cup q^{\varepsilon_{i}(b')-1}\Zbb[q^{-1}].&
    \end{alignat*}
    The conditions on the sums are each equivalent to requiring that $b'$ lies in a strictly longer $i$-string than $b$.
\end{itemize}
The above definitions also hold for integrable $\Uaff$-modules by simply replacing $\Pbar$ with $P$.
For a more detailed introduction to global bases and their properties, see for example \cite{Kashiwara93}.
We are now ready to define a good $\Udash$-module in the sense of Kashiwara \cite{Kashiwara02}.

\begin{defn}
    A finite dimensional $\Udash$-module $V$ is good if it has a crystal basis $(L,B)$ with $B$ simple, a bar involution, and a global basis.
\end{defn}

Since a simple crystal is always connected \cite{AK97}, such a module must be irreducible.

\begin{eg}
    \begin{itemize}
        \item Any level $0$ fundamental representation $W(\varpi_{i})$ is a good $\Udash$-module \cite{Kashiwara02}.
        \item In Proposition \ref{BFKL crystals are good} we prove that the level $1$ perfect crystals of Benkart-Frenkel-Kang-Lee \cite{BFKL06} are crystal bases of good modules in all affine types.
    \end{itemize}
\end{eg}

We shall now explain the construction of Fock space representations for quantum affine algebras.
Consider a good $\Udash$-module $V$ with crystal basis $(L,B)$.
We define a $\Uaff$-submodule
\begin{align*}
    N = \Uaff[z^{\pm 1}\otimes z^{\pm 1},z\otimes 1 + 1\otimes z](u\otimes u)
\end{align*}
of $\Vaff^{\otimes 2}$, which is independent of a choice of extremal vector $u\in\Vaff$.
Kashiwara-Miwa-Petersen-Yung \cite{KMPY96}*{\S 3.3} introduced the $q$-wedge product $\wpr$ as the quotient of $\Vaff^{\otimes r}$ by
\begin{align*}
    N_{r} = \sum_{k=0}^{r-2} \Vaff^{\otimes k}\otimes N \otimes \Vaff^{\otimes(r-k-2)},
\end{align*}
which has a $\Uaff$-module structure coming from the coproduct (\ref{coproduct}).
For each $v_{1},\dots,v_{r}\in\Vaff$ the image of $v_{1}\otimes\dots\otimes v_{r}$ in $\wpr$ is denoted by $v_{1}\wedge\dots\wedge v_{r}$.

\begin{rmk}
    An equivalent definition of the submodule $N$ is as $\ker(R-1)$ where $R$ is the action of the R-matrix on $\Vaff\otimes\Vaff$.
\end{rmk}

We say that a sequence $(s_{k})_{k=1}^{r}$ in $\Baff$ is normally ordered if each $\Haff(s_{k+1}\otimes s_{k})>0$, and call $G(s_{r})\wedge\dots\wedge G(s_{1})$ a normally ordered wedge.
It was shown in \cite{KMPY96} that $\wpr$ has a crystal basis with
\begin{itemize}
    \item $L(\wpr)$ the image of $L(\Vaff^{\otimes r}) = L(\Vaff)\otimes_{A}\dots\otimes_{A} L(\Vaff)$ in $\wpr$,
    \item $B(\wpr)$ the set of normally ordered sequences in $\Baff$.
\end{itemize}
Here we identify each element of $B(\wpr)$ with the image of its associated normally ordered wedge in $L(\wpr)/qL(\wpr)$.
\\

Suppose further that $B$ is a perfect crystal of level $\ell$ as defined in Section \ref{crystals preliminaries}.
Recall that for any weight $\lambda\in\Pbar^{+}$ with $\langle\lambda,c\rangle = \ell$ the ground state sequence of weight $\lambda$ is the unique sequence $\pb_{\lambda} = (b_{r})_{r=0}^{\infty}$ in $B$ with $\varphi(b_{r+1}) = \varepsilon(b_{r})$ for all $r\geq 0$ and $\varphi(b_{0}) = \lambda$.
Letting $m_{r}\in\Zbb$ be such that $\Haff(z^{m_{r+1}}b_{r+1}\otimes z^{m_{r}}b_{r}) = 1$ for all $r\geq 0$, we can also define a ground state sequence $\sbold_{\lambda} = (g_{r})_{r=0}^{\infty} = (z^{m_{r}}b_{r})_{r=0}^{\infty}$ in $\Baff$.
\\

Let $\overline{\Fcal(\lambda)} = \lim_{r\rightarrow\infty} \wpr$ be the inductive limit of the $q$-exterior powers with respect to the maps $\wpr\rightarrow \bigwedge^{r+1}\Vaff$ which send $v \mapsto G(g_{r})\wedge v$.
We similarly define $L(\overline{\Fcal(\lambda)})$ to be the inductive limit of the $L(\wpr)$.

\begin{defn}
    The Fock space $\Fcal(\lambda)$ is the quotient $\overline{\Fcal(\lambda)}/\bigcap_{m>0}q^{m}L(\overline{\Fcal(\lambda)})$.
\end{defn}

\begin{rmk}
    We can endow the Fock space with a separated $q$-adic topology by letting $\lbrace q^{m}L(\Fcal(\lambda)) ~\vert~ m\geq 0\rbrace$ be the neighbourhood system of $0$, where $L(\Fcal(\lambda))$ is the image of $L(\overline{\Fcal(\lambda)})$ in $\Fcal(\lambda)$.
    Moreover, the expression for the action of $f_{i}$ in Proposition \ref{Fock space action} converges with respect to this topology.
\end{rmk}

Any element of $\Fcal(\lambda)$ can be written as a linear combination of vectors $\vac{r}\wedge v$ where $v\in\wpr$ and we define $\vac{r} = \dots\wedge G(g_{r+1})\wedge G(g_{r})$ for each $r\in\Nbb$.

\begin{prop} \label{Fock space action}
    \cite{KMPY96} The Fock space $\Fcal(\lambda)$ is equipped with a well-defined $\Uaff$-module structure given by
    \begin{itemize}
        \item $\wt(\vac{r}\wedge v) = \lambda_{r} + \wt(v)$,
        \item $e_{i}(\vac{r}\wedge v) = \vac{r}\wedge e_{i}v$,
        \item $\displaystyle f_{i}(\vac{r}\wedge v) = k_{i}\vac{r}\wedge f_{i}v + \sum_{s>0} k_{i}\vac{r+s}\wedge f_{i}G(g_{r+s-1})\wedge\dots\wedge G(g_{r})\wedge v$,
    \end{itemize}
    for any $v\in\wpr$ and $r\geq 0$.
    Here $\wt(v)$, $e_{i}v$ and $f_{i}v$ come from the action of $\Uaff$ on $\wpr$ via the coproduct (\ref{coproduct}).
\end{prop}

We say that an infinite sequence $\sbold = (s_{r})_{r=0}^{\infty}$ in $\Baff$ is normally ordered if all $\Haff(s_{r+1}\otimes s_{r})>0$, and call $G(\sbold) := \dots \wedge G(s_{2})\wedge G(s_{1})\wedge G(s_{0})$ a normally ordered wedge.
The set $B(\Fcal(\lambda))$ of normally ordered sequences with $s_{r} = g_{r}$ for $r\gg 0$ has the structure of an affine crystal via (\ref{crystal structure on paths}) and (\ref{affine weights on paths}), just as for $\Pcal(\lambda)$.
\\

The following are just some of the properties of Fock spaces proven in \cites{KMPY96,Kashiwara02}:
\begin{itemize}
    \item $\Fcal(\lambda)$ is an integrable highest weight $\Uaff$-module with weights contained in $\lambda + \sum_{i\in I} \Zbb_{\leq 0}\alpha_{i}$.
    In particular, its character is given by
    \begin{align*}
        \mathrm{ch}(\Fcal(\lambda)) = \mathrm{ch}(V(\lambda))\cdot\prod_{k>0}(1 - e^{-k\delta})^{-1}
    \end{align*}
    and it is therefore a direct sum of irreducible highest weight modules $V(\lambda - m\delta)$ with $m\in\Nbb$.
    \item $(L(\Fcal(\lambda)),B(\Fcal(\lambda)))$ is a crystal basis of $\Fcal(\lambda)$ by identifying each $\sbold\in B(\Fcal(\lambda))$ with the image of $G(\sbold)$ in $L(\Fcal(\lambda))/qL(\Fcal(\lambda))$.
    \item The highest weight elements of $B(\Fcal(\lambda))$ are precisely those of the form $\dots\wedge z^{n_{2}}g_{2}\wedge z^{n_{1}}g_{1}\wedge z^{n_{0}}g_{0}$ with $n_{0}\leq n_{1}\leq n_{2}\leq \dots$ and $n_{r} = 0$ for $r\gg 0$.
    \item If $\sbold\in B(\Fcal(\lambda))$ has $s_{k} = g_{k}$ for some $k\in\Nbb$, then $s_{r} = g_{r}$ for all $r\geq k$.
    \item The normally ordered wedges $\lbrace G(\sbold) ~\vert~ \sbold\in B(\Fcal(\lambda))\rbrace$ form a global basis for $\Fcal(\lambda)$.
\end{itemize}
It is clear that the submodule generated by the ground state vector $G(\sbold_{\lambda})$ is a copy of the irreducible highest weight representation $V(\lambda)$.
On the level of affine crystals, this tells us that the connected component of $\sbold_{\lambda}$ in $B(\Fcal(\lambda))$ is isomorphic to $B(\lambda)$.
The following proposition provides a more concrete description of this subcrystal.

\begin{prop} \label{component of ground state sequence proposition}
    The connected component of $\sbold_{\lambda}$ in $B(\Fcal(\lambda))$ consists of the sequences $\sbold = (s_{r})_{r=0}^{\infty}$ with all $\Haff(s_{r+1} \otimes s_{r}) = 1$.
\end{prop}
\begin{proof}
    It is easy to show that if $n_{r} = n_{r+1} - 1 + H(p_{r+1} \otimes p_{r})$ for all $r \geq 0$ with $n_{r} = m_{r}$ for $r \gg 0$ then
    \begin{align*}
        (p_{r})_{r=0}^{\infty} \mapsto (z^{n_{r}}p_{r})_{r=0}^{\infty}
    \end{align*}
    defines a morphism $B(\lambda) \rightarrow B(\Fcal(\lambda))$ which bijects onto the desired subcrystal.
    Since $\pb_{\lambda}$ is sent to $\sbold_{\lambda}$ and $B(\lambda)$ is connected, we are done.
\end{proof}

We can therefore further think of this subcrystal as the set of $\sbold = (s_{r})_{r=0}^{\infty}$ such that applying $z$ to any entry gives a sequence which no longer lies in $B(\Fcal(\lambda))$.

\section{Level 1 perfect crystals in type E} \label{Perfect crystal section}

In this section we present a level $1$ perfect crystal in types $\Eaff{6}$, $\Eaff{7}$ and $\Eaff{8}$ and describe the energy function in each case.
We shall also confirm that they are the crystal bases of good $\Udash$-modules, and can therefore be used to construct the Fock space crystals $B(\Fcal(\lambda))$ for all level $1$ weights $\lambda\in\Pbar^{+}$.
Finally, we construct a Young column model for each level $1$ perfect crystal, which we will use in later sections to build Young wall models for $B(\lambda)$ and $B(\Fcal(\lambda))$.
\\

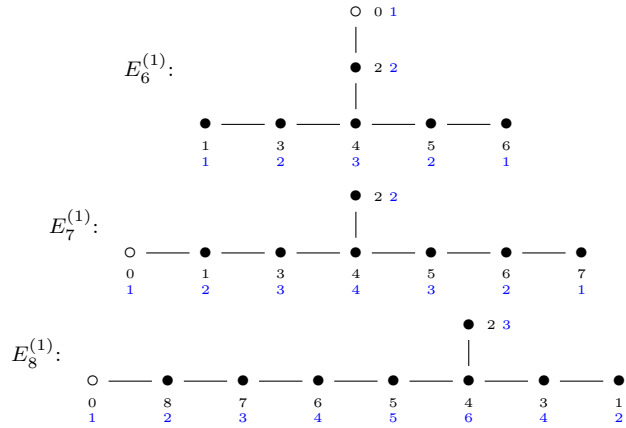
\begin{figure}[htbp]
    \centering
    \begin{tikzpicture}
    \node (label) at (-2.75,-0.25) {\footnotesize $E^{(1)}_{6}$:};
    \node (0) at (0,0.5) {$\circ$};
    \node (1) at (-2,-1) {$\bullet$};
    \node (6) at (0,-0.25) {$\bullet$};
    \node (2) at (-1,-1) {$\bullet$};
    \node (3) at (0,-1) {$\bullet$};
    \node (4) at (1,-1) {$\bullet$};
    \node (5) at (2,-1) {$\bullet$};
    \node (a0) at (0.4,0.5) {\tiny
    $\color{black} 0$
    $\color{blue} 1$ };
    \node (a6) at (0.4,-0.25) {\tiny
    $\color{black} 2$
    $\color{blue} 2$ };
    \node (a1) at (-2,-1.4) {\tiny
    $\begin{array}{c}
        \color{black} 1 \\
        \color{blue} 1
    \end{array}$ };
    \node (a2) at (-1,-1.4) {\tiny
    $\begin{array}{c}
        \color{black} 3 \\
        \color{blue} 2
    \end{array}$ };
    \node (a3) at (0,-1.4) {\tiny
    $\begin{array}{c}
        \color{black} 4 \\
        \color{blue} 3
    \end{array}$ };
    \node (a4) at (1,-1.4) {\tiny
    $\begin{array}{c}
        \color{black} 5 \\
        \color{blue} 2
    \end{array}$ };
    \node (a5) at (2,-1.4) {\tiny
    $\begin{array}{c}
        \color{black} 6 \\
        \color{blue} 1
    \end{array}$ };
    \draw (0) edge (6) (6) edge (3) (1) edge (2) (2) edge (3) (3) edge (4) (4) edge (5);
    \end{tikzpicture}
    
    \begin{tikzpicture}
    \node (label) at (-3.75,-0.625) {\footnotesize $E^{(1)}_{7}$:};
    \node (0) at (-3,-1) {$\circ$};
    \node (1) at (-2,-1) {$\bullet$};
    \node (7) at (0,-0.25) {$\bullet$};
    \node (2) at (-1,-1) {$\bullet$};
    \node (3) at (0,-1) {$\bullet$};
    \node (4) at (1,-1) {$\bullet$};
    \node (5) at (2,-1) {$\bullet$};
    \node (6) at (3,-1) {$\bullet$};
    \node (a0) at (-3,-1.4) {\tiny
    $\begin{array}{c}
        \color{black} 0 \\
        \color{blue} 1
    \end{array}$ };
    \node (a7) at (0.4,-0.25) {\tiny
    $\color{black} 2$
    $\color{blue} 2$ };
    \node (a1) at (-2,-1.4) {\tiny
    $\begin{array}{c}
        \color{black} 1 \\
        \color{blue} 2
    \end{array}$ };
    \node (a2) at (-1,-1.4) {\tiny
    $\begin{array}{c}
        \color{black} 3 \\
        \color{blue} 3
    \end{array}$ };
    \node (a3) at (0,-1.4) {\tiny
    $\begin{array}{c}
        \color{black} 4 \\
        \color{blue} 4
    \end{array}$ };
    \node (a4) at (1,-1.4) {\tiny
    $\begin{array}{c}
        \color{black} 5 \\
        \color{blue} 3
    \end{array}$ };
    \node (a5) at (2,-1.4) {\tiny
    $\begin{array}{c}
        \color{black} 6 \\
        \color{blue} 2
    \end{array}$ };
    \node (a6) at (3,-1.4) {\tiny
    $\begin{array}{c}
        \color{black} 7 \\
        \color{blue} 1
    \end{array}$ };
    \draw (0) edge (1) (7) edge (3) (1) edge (2) (2) edge (3) (3) edge (4) (4) edge (5) (5) edge (6);
    \end{tikzpicture}
    
    \begin{tikzpicture}
    \node (label) at (-3.75,-0.625) {\footnotesize $E^{(1)}_{8}$:};
    \node (0) at (-3,-1) {$\circ$};
    \node (1) at (-2,-1) {$\bullet$};
    \node (8) at (2,-0.25) {$\bullet$};
    \node (2) at (-1,-1) {$\bullet$};
    \node (3) at (0,-1) {$\bullet$};
    \node (4) at (1,-1) {$\bullet$};
    \node (5) at (2,-1) {$\bullet$};
    \node (6) at (3,-1) {$\bullet$};
    \node (7) at (4,-1) {$\bullet$};
    \node (a0) at (-3,-1.4) {\tiny
    $\begin{array}{c}
        \color{black} 0 \\
        \color{blue} 1
    \end{array}$ };
    \node (a8) at (2.4,-0.25) {\tiny
    $\color{black} 2$
    $\color{blue} 3$ };
    \node (a1) at (-2,-1.4) {\tiny
    $\begin{array}{c}
        \color{black} 8 \\
        \color{blue} 2
    \end{array}$ };
    \node (a2) at (-1,-1.4) {\tiny
    $\begin{array}{c}
        \color{black} 7 \\
        \color{blue} 3
    \end{array}$ };
    \node (a3) at (0,-1.4) {\tiny
    $\begin{array}{c}
        \color{black} 6 \\
        \color{blue} 4
    \end{array}$ };
    \node (a4) at (1,-1.4) {\tiny
    $\begin{array}{c}
        \color{black} 5 \\
        \color{blue} 5
    \end{array}$ };
    \node (a5) at (2,-1.4) {\tiny
    $\begin{array}{c}
        \color{black} 4 \\
        \color{blue} 6
    \end{array}$ };
    \node (a6) at (3,-1.4) {\tiny
    $\begin{array}{c}
        \color{black} 3 \\
        \color{blue} 4
    \end{array}$ };
    \node (a7) at (4,-1.4) {\tiny
    $\begin{array}{c}
        \color{black} 1 \\
        \color{blue} 2
    \end{array}$ };
    \draw (0) edge (1) (8) edge (5) (1) edge (2) (2) edge (3) (3) edge (4) (4) edge (5) (5) edge (6) (6) edge (7);
    \end{tikzpicture}
    
    \caption{\hspace{.5em}The type $E$ untwisted affine Dynkin diagrams.
    Black labels are vertex numbers, and
    blue labels are the numerical values $a_{i} = a_{i}^{\vee}$}
    \label{fig:Affine_Dynkin_diagrams}
\end{figure}

Figure \ref{fig:Affine_Dynkin_diagrams} gives the type $E$ untwisted affine Dynkin diagrams, with vertices numbered as in Bourbaki.
Recall that the canonical central element $c$ is equal to $\sum_{i\in I} a_{i}^{\vee}h_{i}$ and hence the level $1$ dominant weights in $\Pbar^{+}$ are precisely the fundamental weights $\Lambda_{i}$ with $i\in\Imin$, namely
\begin{itemize}
    \item $\Lambda_{0}$, $\Lambda_{1}$ and $\Lambda_{6}$ in type $\Eaff{6}$,
    \item $\Lambda_{0}$ and $\Lambda_{7}$ in type $\Eaff{7}$,
    \item $\Lambda_{0}$ in type $\Eaff{8}$.
\end{itemize}

Let us establish a colouring convention for the arrows in our crystal graphs and the blocks in Young columns and Young walls.
This should make these structures easier to understand and analyse, and their patterns and symmetries easier to notice.
Each label $i\in I$ shall correspond to a particular colour, uniformly across types $\Eaff{6}$, $\Eaff{7}$ and $\Eaff{8}$:
\begin{align} \label{colours}
\begin{split}
    &\mathcolor{col0}{0\mathrm{~is~black}};~
    \mathcolor{col1}{1\mathrm{~is~red}};~
    \mathcolor{col2}{2\mathrm{~is~yellow}};~
    \mathcolor{col3}{3\mathrm{~is~green}};~
    \mathcolor{col4}{4\mathrm{~is~purple}}; \\
    &
    \mathcolor{col5}{5\mathrm{~is~blue}};~
    \mathcolor{col6}{6\mathrm{~is~orange}};~
    \mathcolor{col7}{7\mathrm{~is~pink}};~
    \mathcolor{col8}{8\mathrm{~is~brown}}.
\end{split}
\end{align}

Let $B_{6}$ be the crystal basis of the level $0$ fundamental representation $W(\varpi_{1})$ in type $\Eaff{6}$, and $B_{7}$ be that of $W(\varpi_{7})$ in type $\Eaff{7}$.
Figures \ref{B6 crystal graph} and \ref{B7 crystal graph} contain the associated crystal graphs, with arrows coloured according to (\ref{colours}).
It is easy to verify directly that they are level $1$ perfect crystals.
(Alternatively, this is a consequence of \cite{Hiroshima21}*{Theorem 5.1} since they are examples of Kirillov-Reshetikhin crystals $B^{r,s}$ with $r$ minuscule.)
So via the path realization of Theorem \ref{path realization theorem} these crystals can be used to construct the crystal basis $B(\lambda)$ for each level $1$ weight $\lambda\in\Pbar^{+}$.
Furthermore, Kashiwara \cite{Kashiwara02} proved that level $0$ fundamental representations are good, and hence they can also be used to construct the Fock space crystals $B(\Fcal(\lambda))$ as in Section \ref{Fock space preliminaries}.
\\

The basis elements of $B_{6}$ and $B_{7}$ can be labelled according to their incident edges as follows, precisely as in Jones-Schilling \cite{JS10}.
In each case, the weight function is injective and all $i$-strings have length at most $1$, so we may denote any element $b$ by $\overline{i_{1}}\dots \overline{i_{m}}j_{1}\dots j_{n}$ where an $\overline{i}$ (resp. $i$) records an $i$-arrow into (resp. out of) $b$.
Without loss of generality we shall always take $i_{1}<\dots<i_{m}$ and $j_{1}<\dots<j_{n}$.
\\

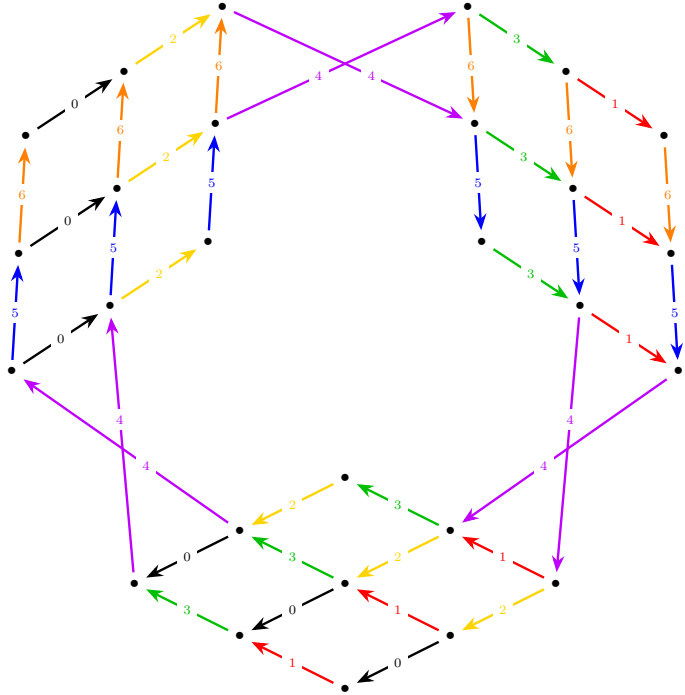
\begin{figure}[H]
    \centering
    \begin{tikzpicture}[-{Stealth[scale=0.9]}, scale=0.7, transform shape, line width = 0.9pt]
    \node(d30) at (30:3) {$\bullet$};
    \node(e30) at (30:5) {$\bullet$};
    \node(f30) at (30:7) {$\bullet$};
    \node(g30) at (56.56505118:4.4721359549996) {$\bullet$};
    \node(h30) at (48.43494882:6.3245553203368) {$\bullet$};
    \node(i30) at (68.65980825:6.4031242374328) {$\bullet$};
    \node(b30) at (3.434948823:4.4721359549996) {$\bullet$};
    \node(c30) at (11.56505118:6.3245553203368) {$\bullet$};
    \node(a30) at (-8.659808254:6.4031242374328) {$\bullet$};
    \node(d150) at (150:3) {$\bullet$};
    \node(e150) at (150:5) {$\bullet$};
    \node(f150) at (150:7) {$\bullet$};
    \node(g150) at (176.5650512:4.4721359549996) {$\bullet$};
    \node(h150) at (168.4349488:6.3245553203368) {$\bullet$};
    \node(i150) at (188.6598083:6.4031242374328) {$\bullet$};
    \node(b150) at (123.4349488:4.4721359549996) {$\bullet$};
    \node(c150) at (131.5650512:6.3245553203368) {$\bullet$};
    \node(a150) at (111.3401917:6.4031242374328) {$\bullet$};
    \node(d90) at (-90:3) {$\bullet$};
    \node(e90) at (-90:5) {$\bullet$};
    \node(f90) at (-90:7) {$\bullet$};
    \node(g90) at (-63.434948822922:4.4721359549996) {$\bullet$};
    \node(h90) at (-71.565051177078:6.3245553203368) {$\bullet$};
    \node(i90) at (-51.34019174591:6.4031242374328) {$\bullet$};
    \node(b90) at (-116.5650511771:4.4721359549996) {$\bullet$};
    \node(c90) at (-108.4349488229:6.3245553203368) {$\bullet$};
    \node(a90) at (-128.6598082541:6.4031242374328) {$\bullet$};
    \tikzstyle{every node}=[midway, fill=white]
    \draw[col0] (h90) -- (f90) node{$\scriptstyle 0$};
    \draw[col0] (e90) -- (c90) node{$\scriptstyle 0$};
    \draw[col0] (b90) -- (a90) node{$\scriptstyle 0$};
    \draw[col1] (f90) -- (c90) node{$\scriptstyle 1$};
    \draw[col1] (h90) -- (e90) node{$\scriptstyle 1$};
    \draw[col1] (i90) -- (g90) node{$\scriptstyle 1$};
    \draw[col2] (d90) -- (b90) node{$\scriptstyle 2$};
    \draw[col2] (g90) -- (e90) node{$\scriptstyle 2$};
    \draw[col2] (i90) -- (h90) node{$\scriptstyle 2$};
    \draw[col3] (c90) -- (a90) node{$\scriptstyle 3$};
    \draw[col3] (e90) -- (b90) node{$\scriptstyle 3$};
    \draw[col3] (g90) -- (d90) node{$\scriptstyle 3$};
    \draw[col1] (h30) -- (f30) node{$\scriptstyle 1$};
    \draw[col1] (e30) -- (c30) node{$\scriptstyle 1$};
    \draw[col1] (b30) -- (a30) node{$\scriptstyle 1$};
    \draw[col6] (f30) -- (c30) node{$\scriptstyle 6$};
    \draw[col6] (h30) -- (e30) node{$\scriptstyle 6$};
    \draw[col6] (i30) -- (g30) node{$\scriptstyle 6$};
    \draw[col3] (d30) -- (b30) node{$\scriptstyle 3$};
    \draw[col3] (g30) -- (e30) node{$\scriptstyle 3$};
    \draw[col3] (i30) -- (h30) node{$\scriptstyle 3$};
    \draw[col5] (c30) -- (a30) node{$\scriptstyle 5$};
    \draw[col5] (e30) -- (b30) node{$\scriptstyle 5$};
    \draw[col5] (g30) -- (d30) node{$\scriptstyle 5$};
    \draw[col6] (h150) -- (f150) node{$\scriptstyle 6$};
    \draw[col6] (e150) -- (c150) node{$\scriptstyle 6$};
    \draw[col6] (b150) -- (a150) node{$\scriptstyle 6$};
    \draw[col0] (f150) -- (c150) node{$\scriptstyle 0$};
    \draw[col0] (h150) -- (e150) node{$\scriptstyle 0$};
    \draw[col0] (i150) -- (g150) node{$\scriptstyle 0$};
    \draw[col5] (d150) -- (b150) node{$\scriptstyle 5$};
    \draw[col5] (g150) -- (e150) node{$\scriptstyle 5$};
    \draw[col5] (i150) -- (h150) node{$\scriptstyle 5$};
    \draw[col2] (c150) -- (a150) node{$\scriptstyle 2$};
    \draw[col2] (e150) -- (b150) node{$\scriptstyle 2$};
    \draw[col2] (g150) -- (d150) node{$\scriptstyle 2$};
    \draw[col4] (a30)--(g90) node[pos=0.6]{$\scriptstyle 4$};
    \draw[col4] (b30)--(i90) node[pos=0.4]{$\scriptstyle 4$};
    \draw[col4] (a90)--(g150) node[pos=0.6]{$\scriptstyle 4$};
    \draw[col4] (b90)--(i150) node[pos=0.4]{$\scriptstyle 4$};
    \draw[col4] (a150)--(g30) node[pos=0.6]{$\scriptstyle 4$};
    \draw[col4] (b150)--(i30) node[pos=0.4]{$\scriptstyle 4$};
    \end{tikzpicture}
    \caption{\hspace{.5em}The crystal graph of $B_{6}$}\label{B6 crystal graph}
\end{figure}

\input{E7_level_1_perfect_crystal}

Recall from Definition \ref{perfect crystal definition} that in a level $1$ perfect crystal, for each $\lambda\in\Pbar^{+}$ with $\langle\lambda,c\rangle = 1$ we denote by $b^{\lambda}$ and $b_{\lambda}$ the unique elements with $\varepsilon(b^{\lambda}) = \varphi(b_{\lambda}) = \lambda$.
\begin{itemize}
    \item For $B_{6}$ these are $b^{\Lambda_{6}} = b_{\Lambda_{0}} = \overline{6}0$, $b^{\Lambda_{1}} = b_{\Lambda_{6}} = \overline{1}6$ and $b^{\Lambda_{0}} = b_{\Lambda_{1}} = \overline{0}1$.
    \item For $B_{7}$ these are $b^{\Lambda_{7}} = b_{\Lambda_{0}} = \overline{7}0$ and $b^{\Lambda_{0}} = b_{\Lambda_{7}} = \overline{0}7$.
\end{itemize}

\begin{rmk}
    \begin{itemize}
        \item The outer automorphism group $\Omega$ in type $\Eaff{6}$ can be seen in the rotational symmetry of $B_{6}$ in Figure \ref{B6 crystal graph}.
        Indeed, twisting $W(\varpi_{1})$ by the automorphism of $\Udash$ induced from an anticlockwise rotation $\sigma\in\Omega$ corresponds on the crystal level to the anticlockwise rotation which sends each $b_{\Lambda_{i}}$ to $b_{\Lambda_{\sigma(i)}}$.
        \item Similarly, twisting $W(\varpi_{7})$ by the automorphism of $\Udash$ induced from the reflection of the $\Eaff{7}$ Dynkin diagram corresponds to rotating $B_{7}$ by $180^{\circ}$ around the centre of Figure \ref{B7 crystal graph}.
    \end{itemize}
\end{rmk}

\subsection{The level 1 perfect crystals of Benkart-Frenkel-Kang-Lee}

Since in type $\Eaff{8}$ there are no non-zero minuscule vertices in $I$, our level $1$ perfect crystal must be formed in a different manner.
Benkart-Frenkel-Kang-Lee \cite{BFKL06} provide a uniform construction of a level $1$ perfect crystal for all affine types, which we include here (in the untwisted case only).
Denote the sets of roots, positive roots and negative roots of the corresponding finite Lie algebra $\g$ by $\Phi$, $\Phi^{+}$ and $\Phi^{-}$ respectively, let $I_{0} = I\setminus\lbrace 0\rbrace$, and let $\theta = \sum_{i\in I_{0}} a_{i}\alpha_{i}$ be the highest root in $\Phi$.

\begin{thm} \label{uniform level 1 perfect crystal theorem}
    \cite{BFKL06} The classical crystal $B$ with elements
    \begin{align*}
        \lbrace \xa ~\vert~ \alpha\in\Phi^{+}\sqcup\Phi^{-}\rbrace \sqcup \lbrace y_{i} ~\vert~ i\in I_{0}\rbrace \sqcup \lbrace \emptyset \rbrace
    \end{align*}
    and functions $\wt$, $\et_{i}$, $\ft_{i}$, $\varepsilon_{i}$ and $\varphi_{i}$ given by the crystal graph structure
    \begin{itemize}
        \item $\xa \xrightarrow{i} \xb$ whenever $\alpha - \alpha_{i} = \beta$ for some $\alpha,\beta\in\Phi^{+}\sqcup\Phi^{-}$ and $i\in I_{0}$,
        \item $x_{\alpha_{i}} \xrightarrow{i} y_{i} \xrightarrow{i} x_{-\alpha_{i}}$ for all $i\in I_{0}$,
        \item $\xa \xrightarrow{0} \xb$ whenever $\alpha + \theta = \beta$ for some $\alpha,\beta\in(\Phi^{+}\sqcup\Phi^{-})\setminus\lbrace\pm\theta\rbrace$,
        \item $\xmt \xrightarrow{0} \emptyset \xrightarrow{0} \xt$,
    \end{itemize}
    is a level $1$ perfect crystal with $b^{\Lambda_{0}} = b_{\Lambda_{0}} = \emptyset$.
\end{thm}

Note that if we forget all $0$-arrows then we are left with the crystal basis $B(0)\sqcup B(\theta)$ for the finite quantum group $\Uq$.
So in all untwisted types, the crystal graph of $B$ is essentially formed by taking the Hasse diagram of the root poset, separating out

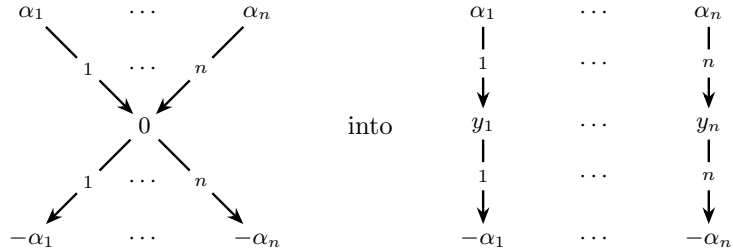
\begin{center}
    \begin{tikzpicture}[-{Stealth[scale=0.9]}, scale=0.75, line width = 0.9pt]
    \node(u-1) at (-2,2) {\small $\alpha_{1}$};
    \node(u0) at (0,2) {\small $\cdots$};
    \node(u1) at (2,2) {\small $\alpha_{n}$};
    \node at (0,1) {\small $\cdots$};
    \node(m0) at (0,0) {\small $0$};
    \node at (0,-1) {\small $\cdots$};
    \node(d-1) at (-2,-2) {\small $-\alpha_{1}$};
    \node(d0) at (0,-2) {\small $\cdots$};
    \node(d1) at (2,-2) {\small $-\alpha_{n}$};
    \draw[col0, -] (-1.75,1.75) -- (-1.2,1.2);
    \node[col0, fill=none] at (-1,1) {$\scriptstyle 1$};
    \draw[col0] (-0.8,0.8) -- (-0.2,0.2);
    \draw[col0, -] (1.75,1.75) -- (1.2,1.2);
    \node[col0, fill=none] at (1,1) {$\scriptstyle n$};
    \draw[col0] (0.8,0.8) -- (0.2,0.2);
    \draw[col0] (-1.2,-1.2) -- (-1.75,-1.75);
    \node[col0, fill=none] at (-1,-1) {$\scriptstyle 1$};
    \draw[col0, -] (-0.25,-0.25) -- (-0.8,-0.8);
    \draw[col0] (1.2,-1.2) -- (1.75,-1.75);
    \node[col0, fill=none] at (1,-1) {$\scriptstyle n$};
    \draw[col0, -] (0.25,-0.25) -- (0.8,-0.8);

    \node at (4,0) {$\mathrm{into}$};

    \node at (6,2) {\small $\alpha_{1}$};
    \node at (8,2) {\small $\cdots$};
    \node at (10,2) {\small $\alpha_{n}$};
    \node at (8,1.1) {\small $\cdots$};
    \node at (6,0) {\small $y_{1}$};
    \node at (8,0) {\small $\cdots$};
    \node at (10,0) {\small $y_{n}$};
    \node at (8,-0.9) {\small $\cdots$};
    \node at (6,-2) {\small $-\alpha_{1}$};
    \node at (8,-2) {\small $\cdots$};
    \node at (10,-2) {\small $-\alpha_{n}$};
    \draw[col0, -] (6,1.75) -- (6,1.35);
    \node[col0, fill=none] at (6,1.1) {$\scriptstyle 1$};
    \draw[col0] (6,0.85) -- (6,0.3);
    \draw[col0, -] (10,1.75) -- (10,1.35);
    \node[col0, fill=none] at (10,1.1) {$\scriptstyle n$};
    \draw[col0] (10,0.85) -- (10,0.3);
    \draw[col0] (6,-1.15) -- (6,-1.75);
    \node[col0, fill=none] at (6,-0.9) {$\scriptstyle 1$};
    \draw[col0, -] (6,-0.25) -- (6,-0.65);
    \draw[col0] (10,-1.15) -- (10,-1.75);
    \node[col0, fill=none] at (10,-0.9) {$\scriptstyle n$};
    \draw[col0, -] (10,-0.25) -- (10,-0.65);
\end{tikzpicture}
\end{center}

and adding $\emptyset$ and the $0$-arrows.
Let $\tilde{I}$ be the set of vertices adjacent to $0$ in the affine Dynkin diagram, and define
\begin{align*}
    \Phi^{\pm}_{k} = \lbrace \alpha \in \Phi^{\pm} ~\vert~ \mathrm{the~coefficients~of~} \lbrace\alpha_{i}\rbrace_{i\in\tilde{I}} \mathrm{~in~} \alpha \mathrm{~sum~to~} \pm k\rbrace
\end{align*}
for each $k\in\Nbb$.
Then there is a decomposition $\Phi^{\pm} = \Phi^{\pm}_{0}\sqcup\Phi^{\pm}_{1}\sqcup\Phi^{\pm}_{2}$ and the $0$-arrows in $B$ can be described as follows.

\begin{lem} \label{0-arrows lemma}
    The $0$-strings in $B$ are precisely $\xmt \xrightarrow{0} \emptyset \xrightarrow{0} \xt$ together with $\xa \xrightarrow{0} x_{\alpha + \theta}$ for every $\alpha\in\Phi^{-}_{1}$ (equivalently $x_{\beta - \theta} \xrightarrow{0} \xb$ for every $\beta\in\Phi^{+}_{1}$).
\end{lem}
\begin{proof}
    This follows from the fact that every $0$-string other than $\xmt \xrightarrow{0} \emptyset \xrightarrow{0} \xt$ is of length at most $1$ by construction, and $\langle \alpha_{i},h_{0}\rangle = -\delta_{i \in \tilde{I}}$ for all $i\in I_{0}$.
\end{proof}

Denote the crystal $B$ in type $\Eaff{8}$ by $B_{8}$, which is in particular equal to the Kirillov-Reshetikhin crystal $B^{8,1}$ of the level $0$ fundamental representation $W(\varpi_{8})$.
We include its crystal graph in Appendix \ref{B8 crystal graph appendix} to give an idea of its structure, but to avoid cluttering we remove all vertex and edge labels, as well as all $0$-arrows other than $\xmt \xrightarrow{0} \emptyset \xrightarrow{0} \xt$.
However, edges are coloured according to our conventions (\ref{colours}) and the $0$-arrows can easily be deduced from Lemma \ref{0-arrows lemma}, together with the vertex icons:
\begin{itemize}
    \item $\xa$ with $\alpha\in\Phi^{\pm}_{0}$ are $\cross$
    \item $\xa$ with $\alpha\in\Phi^{\pm}_{1}$ are $\circle$
    \item $\xa$ with $\alpha\in\Phi^{\pm}_{2}$ are $\star$
    \item $\emptyset$ and $y_{i}$ with $i\in I_{0}$ are $\square$
\end{itemize}

Next we shall verify that $B$ is the crystal basis of a good $\Udash$ module, and thus in particular can be used to construct the Fock space crystal $B(\Fcal(\Lambda_{0}))$.
We present the untwisted simply laced case only, but the same proof works for all affine types with only minor edits.
\\

As part of their proof that $B$ is a level $1$ perfect crystal, Benkart-Frenkel-Kang-Lee \cite{BFKL06} show that the space
\begin{align*}
    V = \left( \bigoplus_{\beta\in\Phi^{+}\sqcup\Phi^{-}} \Cbb(q) \xb \right)
    \oplus
    \left( \bigoplus_{j\in I_{0}} \Cbb(q) y_{j} \right)
    \oplus
    \Cbb(q)\emptyset
\end{align*}
has the structure of an integrable $\Udash$-module via
\begin{align} \label{BFKL module structure}
\begin{split}
    &q^{h} \cdot \xb = q^{\langle \alpha,h\rangle} \xb, \quad
    q^{h} \cdot y_{j} = y_{j}, \quad
    q^{h} \cdot \emptyset = \emptyset, \\
    &e_{i} \cdot \xb =
    \begin{cases}
        [\varphi_{i}(\xb) + 1] x_{\beta+\alpha_{i}} &\mathrm{if~} \beta+\alpha_{i}\in\Phi^{+}\sqcup\Phi^{-} \\
        y_{i} &\mathrm{if~} \beta=-\alpha_{i} \\
        0 &\mathrm{otherwise}
    \end{cases}
    \quad (i\not= 0), \\
    &f_{i} \cdot \xb =
    \begin{cases}
        [\varepsilon_{i}(\xb) + 1] x_{\beta-\alpha_{i}} &\mathrm{if~} \beta-\alpha_{i}\in\Phi^{+}\sqcup\Phi^{-} \\
        y_{i} &\mathrm{if~} \beta=\alpha_{i} \\
        0 &\mathrm{otherwise}
    \end{cases}
    \quad (i\not= 0), \\
    &e_{i} \cdot \emptyset = 0, \quad
    e_{i} \cdot y_{j} = [\langle \alpha_{i},h_{j}\rangle] x_{\alpha_{i}}
    \quad (i\not= 0), \\
    &f_{i} \cdot \emptyset = 0, \quad
    f_{i} \cdot y_{j} = [\langle \alpha_{i},h_{j}\rangle] x_{-\alpha_{i}}
    \quad (i\not= 0), \\
    &e_{0} \cdot \xb =
    \begin{cases}
        x_{\beta-\theta} &\mathrm{if~} \beta-\theta\in\Phi^{+}\sqcup\Phi^{-}, \\
        \emptyset &\mathrm{if~} \beta=\theta, \\
        0 &\mathrm{otherwise},
    \end{cases}
    \\
    &f_{0} \cdot \xb =
    \begin{cases}
        x_{\beta+\theta} &\mathrm{if~} \beta+\theta\in\Phi^{+}\sqcup\Phi^{-}, \\
        \emptyset &\mathrm{if~} \beta=-\theta, \\
        0 &\mathrm{otherwise},
    \end{cases}
    \\
    &e_{0} \cdot \emptyset = [2] x_{-\theta}, \quad
    e_{0} \cdot y_{j} = 0, \\
    &f_{0} \cdot \emptyset = [2] x_{-\theta}, \quad
    f_{0} \cdot y_{j} = 0,
\end{split}
\end{align}
and moreover that $V$ has crystal basis $(L,B)$ with $L$ the free $A$-submodule of $V$ generated by $B_{V} := \lbrace x_{\beta},y_{j},\emptyset ~\vert~ \beta\in\Phi^{+}\sqcup\Phi^{-}, j\in I_{0} \rbrace \subset V$.
As the crystal basis of a finite dimensional integrable $\Udash$-module, $B$ is moreover a \emph{regular} finite classical crystal (see for example \cite{HKKOT00}*{p.117}).
\\

Since an element of a crystal is extremal if and only if it lies at the start or end of every $i$-string it is contained in, the extremal elements of $B$ are precisely the elements $x_{\beta}$.
Then as $\Phi\setminus\lbrace 0\rbrace$ is a single Weyl group orbit, by picking $\lambda = \theta$ in Definition \ref{simple crystal definition} we see that $B$ is simple.
\\

Since $\overline{[n]} = [n]$ for all $n\in\Zbb$ there is a bar involution $-$ of $V$ which sends $q\rightarrow q^{-1}$ and fixes each element of $B_{V}$.
Then letting $V_{\Qbb}$ be the $\Udash_{\Qbb}$-submodule of $V$ generated by $B_{V}$, it is clear that $\overline{V_{\Qbb}} = V_{\Qbb}$.
From (\ref{BFKL module structure}) any $u\in V_{\Qbb}$ is a linear combination of elements of $B_{V}$ where each coefficient is some $f \in \Qbb[q,q^{-1}]$ multiplied by $[n]^{\pm 1}$ factors.
So as $f - \overline{f} \in (q-1)\Qbb[q,q^{-1}]$ we have $u - \overline{u} \in (q-1)V_{\Qbb}$.
\\

It remains to show that $(L,\overline{L},V_{\Qbb})$ is a balanced triple.
First note that $A\cap\overline{A}\cap\Qbb[q,q^{-1}] = \Qbb$ since an element of $\Qbb[q,q^{-1}]$ is regular at both $q=0$ and $q=\infty$ if and only if it is constant.
Then since $\overline{L}$ is the free $\overline{A}$-module generated by $B_{V}$ we have that $E = L\cap\overline{L}\cap V_{\Qbb}$ is the $\Qbb$-span of $B_{V}$.
\\

Now, $qL$ is the $qA$-span of $B_{V}$ and $qA$ is the set of $f(q)\in \Qbb(q)$ without a pole at $0$ and with $f(0) = 0$.
So $L/qL$ is isomorphic to the $(A/qA)$-span of $B_{V}$.
But $A/qA \cong \Qbb$ via the evaluation-at-$0$ map, so $E \rightarrow L/qL$ via inclusion into $L$ and then projection to $L/qL$ is an isomorphism, and hence $(L,\overline{L},V_{\Qbb})$ is balanced.
\\

The preimage of $B\subset L/qL$ in $E$ is a global basis for $V$, and is precisely equal to $B_{V}$.
So we may conclude the following.

\begin{prop} \label{BFKL crystals are good}
    The level $1$ perfect crystal of Benkart-Frenkel-Kang-Lee is the crystal basis of a good $\Udash$-module in all affine types.
\end{prop}

\subsection{Energy functions}

A quick check verifies that for both $B_{6}$ and $B_{7}$ the energy function $H$ admits the following nice description.

\begin{lem} \label{B6 and B7 energy function lemma}
    For all $a,b\in B_{6}$ (resp. $a,b\in B_{7}$), $H(b\otimes a)$ is equal to the minimum number of $0$-arrows in a path $a\rightarrow\dots\rightarrow b$ in $B_{6}$ (resp. $B_{7}$), or equivalently the number of $0$-arrows in a minimal length path $a\rightarrow\dots\rightarrow b$.
\end{lem}

In \cite{BFKL06} the authors give a description of the energy function for their uniform level $1$ perfect crystal.
However, in the process of writing this paper we noticed some inaccuracies in the results there.
Here we provide an updated description in the untwisted case, which in particular includes $B_{8}$.
We have communicated proofs of these corrections to the authors.

\begin{rmk}
    Aside from the aforementioned differences, our energy function $H$ is obtained from that of \cite{BFKL06} by first multiplying by $-1$ (see Remark \ref{energy function difference remark}), and then adding $2$ to renormalise so that $H(b \otimes b) = 0$ for any extremal element $b\in B$.
\end{rmk}

In all untwisted types, let $B$ be the uniform level $1$ perfect crystal of Theorem \ref{uniform level 1 perfect crystal theorem}.
Call an element $b_{1} \otimes b_{2}$ of $B \otimes B$ maximal if $\et_{i}(b_{1}\otimes b_{2}) = 0$ for all $i\in I_{0}$.
We denote by $\Ccal(b_{1}\otimes b_{2})$ its \textit{classical connected component} in $B \otimes B$ after removing all $0$-arrows, on which the energy function $H$ is constant by definition.

\begin{thm} \label{uniform crystal maximal vector theorem}
    The maximal vectors and their energy functions are as follows:
    \begin{itemize}
        \item $\emptyset \otimes \emptyset$ with $H=2$ and classical connected component $\lbrace \emptyset \otimes \emptyset \rbrace$,
        \item $x_{\theta} \otimes x_{-\theta}$ with $H=2$ and classical connected component $\lbrace x_{\theta} \otimes x_{-\theta} \rbrace$,
        \item $\emptyset \otimes x_{\theta}$ with $H=1$ and classical connected component $\lbrace \emptyset \otimes b ~\vert~ b\not= \emptyset \rbrace$,
        \item $x_{\theta} \otimes \emptyset$ with $H=1$ and classical connected component $\lbrace b \otimes \emptyset ~\vert~ b\not= \emptyset \rbrace$,
        \item $x_{\theta} \otimes y_{i}$ for each $i\in\tilde{I}$, with $H=2$ and classical connected components as described below,
        \item $x_{\theta}\otimes x_{\theta-\alpha_{i}}$ with $H=1$ for each $i\in\tilde{I}$,
        \item $x_{\theta}\otimes x_{\theta}$ with $H=0$,
        \item $x_{\theta}\otimes x_{\beta}$ with $H=2$ for each maximal $\beta \in \Lambda_{0}^{+}$,
    \end{itemize}
    and in type $C_{n}^{(1)}$, since a double arrow connects $0$ to its adjacent vertex,
    \begin{itemize}
        \item $x_{\theta} \otimes x_{-\alpha_{i}}$ with $H=2$ for $i\in\tilde{I}$.
    \end{itemize}
\end{thm}

We do not provide a complete description of every classical connected component since this will not be required for our work, but do include the case $\Ccal(x_{\theta} \otimes y_{i})$ -- for further details and its crystal structure, see the proof of \cite{BFKL06}*{Proposition 5.3}.
To this end, define the support of any $\gamma\in\Phi$ to be $\mathrm{supp}(\gamma) = \lbrace j\in I_{0} ~\vert~ \langle \Lambda_{j},\gamma\rangle \not= 0\rbrace$.

\begin{prop} \label{component of xt otimes yi}
    Outside types $A_{n}^{(1)}$ and $C_{n}^{(1)}$ the component $\Ccal(x_{\theta} \otimes y_{i})$ consists of the elements:
    \begin{itemize}
        \item $\xt\otimes y_{i}$ and $y_{i}\otimes x_{-\theta}$,
        \item $x_{\theta-\alpha}\otimes x_{-\beta}$ and $x_{\beta}\otimes x_{-\theta+\alpha}$ for all $\gamma\in\Phi^{+}\setminus\lbrace\theta\rbrace$, with $\alpha = \alpha_{j_{1}}+\dots+\alpha_{j_{t}}$ and $\beta = \theta-\gamma-\alpha$, where $j_{1}\sim\dots\sim j_{t} = i$ is the shortest sequence of vertices in $I_{0}\setminus\mathrm{supp}(\gamma)$ connecting $\mathrm{supp}(\gamma)$ to $i$,
        \item $x_{\theta-\alpha_{j_{1}}-\dots-\alpha_{j_{t}}} \otimes x_{-\theta+\alpha_{j_{1}}+\dots+\alpha_{j_{t}}}$ for all $j\in I_{0}$, where $i = j_{1}\sim\dots\sim j_{t} = j$ is the shortest sequence of vertices connecting $i$ to $j$.
    \end{itemize}
    In type $A_{n}^{(1)}$ the elements of $\Ccal(x_{\theta} \otimes y_{i})$ are:
    \begin{itemize}
        \item $x_{\gamma}\otimes y_{i}$ and $y_{i}\otimes x_{-\gamma}$ for all $\gamma\in\Phi^{+}$ with $i\in\mathrm{supp}(\gamma)$,
        \item $x_{\theta-\alpha}\otimes x_{-\beta}$ and $x_{\beta}\otimes x_{-\theta+\alpha}$ for all $\gamma\in\Phi^{+}$ with $i\not\in\mathrm{supp}(\gamma)$, with $\alpha = \alpha_{j_{1}}+\dots+\alpha_{j_{t}}$ and $\beta = \theta-\gamma-\alpha$, where $j_{1}\sim\dots\sim j_{t} = i$ is the shortest sequence of vertices in $I_{0}\setminus\mathrm{supp}(\gamma)$ connecting $\mathrm{supp}(\gamma)$ to $i$,
        \item $y_{i}\otimes y_{i}$ and $x_{\theta-\alpha_{j_{1}}-\dots-\alpha_{j_{t}}} \otimes x_{-\theta+\alpha_{j_{1}}+\dots+\alpha_{j_{t}}}$ for all $j\in I_{0}\setminus\lbrace i\rbrace$, where $n+1-i = j_{1}\sim\dots\sim j_{t} = j$ is the shortest sequence of vertices connecting $n+1-i$ to $j$.
    \end{itemize}
    In type $C_{n}^{(1)}$ the elements of $\Ccal(x_{\theta} \otimes y_{i})$ are:
    \begin{itemize}
        \item $x_{\gamma}\otimes y_{i}$ and $y_{i}\otimes x_{-\gamma}$ for all $\gamma\in\Phi^{+}$ with $i\in\mathrm{supp}(\gamma)$,
        \item $x_{\gamma+\alpha}\otimes x_{-\alpha}$ and $x_{\beta}\otimes x_{-\theta+\alpha}$ for all $\gamma\in\Phi^{+}$ with $i\not\in\mathrm{supp}(\gamma)$, with $\alpha = \alpha_{j_{1}}+\dots+\alpha_{j_{t}}$ and $\beta = \theta-\gamma-\alpha$, where $j_{1}\sim\dots\sim j_{t} = i$ is the shortest sequence of vertices in $I_{0}\setminus\mathrm{supp}(\gamma)$ connecting $\mathrm{supp}(\gamma)$ to $i$,
        \item $y_{i}\otimes y_{i}$ and $x_{\alpha_{j_{1}}+\dots+\alpha_{j_{t}}} \otimes x_{-\alpha_{j_{1}}-\dots-\alpha_{j_{t}}}$ for all $j\in I_{0}\setminus\lbrace i\rbrace$, where $i = j_{1}\sim\dots\sim j_{t} = j$ is the shortest sequence of vertices connecting $i$ to $j$.
    \end{itemize}
\end{prop}

We do however note that $B$ is a Kirillov-Reshetikhin crystal in each type $X_{n}^{(1)} \not= A_{n}^{(1)}$.
In particular, it equals $B^{i,s} = B(s\varpi_{i})$ where $i$ is the unique element in $\tilde{I}$ and $s = -a_{0i}$.
Therefore the maximal vectors, together with their energy functions and classical connected components, can easily be computed using SageMath \cite{SageMath}.

\begin{lstlisting}
sage: K=crystals.kirillov_reshetikhin.LSPaths(['X',n,1],i,s)
sage: H=K.local_energy_function(K)
sage: K2=crystals.TensorProduct(K,K)
sage: hw=K2.classically_highest_weight_vectors()
sage: for b in hw:
          print("({},{}) {}".format(b[1],b[0],H(b)))
sage: C=K2.subcrystal(generators=[hw[k]],index_set=[1,...,n])
\end{lstlisting}

Here we reverse the factors {\tt b[0]} and {\tt b[1]} of {\tt b} in the penultimate line of code since by default SageMath uses an alternative tensor crystal structure, where tensor factors are swapped compared to (\ref{tensor product of crystals}).
Type $A_{n}^{(1)}$ can be dealt with similarly, by using the fact that $B = B(\varpi_{1} + \varpi_{n})$ is a connected component of $B^{1,1}\otimes B^{n,1}$.

\subsection{Young column realizations}

We are now ready to present for each of our level $1$ perfect crystals $B_{6}$, $B_{7}$ and $B_{8}$ a new combinatorial model in terms of \emph{Young columns}.
In each case, a Young column pattern splits the infinite vertical column of unit cubes into a particular arrangement of $I$-coloured blocks.
Vertices in our finite crystals then correspond to certain valid stackings of blocks inside this pattern, up to an equivalence relation, and $i$-arrows correspond to adding / removing an $i$-coloured block ($i$-block).
\\

An essential condition for all Young columns (except for three in type $\Eaff{8}$) is that there is no empty space below any block, which can be thought of as a \textit{stable under gravity} condition.
So it is important to make clear exactly how each unit cube is cut into blocks.
We therefore introduce the following diagrammatic conventions.
\\

Our unit cubes shall always be split via a collection of vertical cuts, and hence all of the information is contained in the horizontal cross-section at any point.
So we shall represent each Young column pattern as a vertical strip formed by lining up the cross-sections of each unit cube, one on top of another.
Thicker lines separate the cross-sections of cubes, and thinner lines indicate how each cube is cut.
\\

In types $\Eaff{7}$ and $\Eaff{8}$ we shall also stretch the cross-section diagrams horizontally by factors of $1.5$ and $3$ respectively, in order to make our diagrams more clear.
This is all demonstrated in Figure \ref{Block splitting diagrams}, which contains examples of this process.

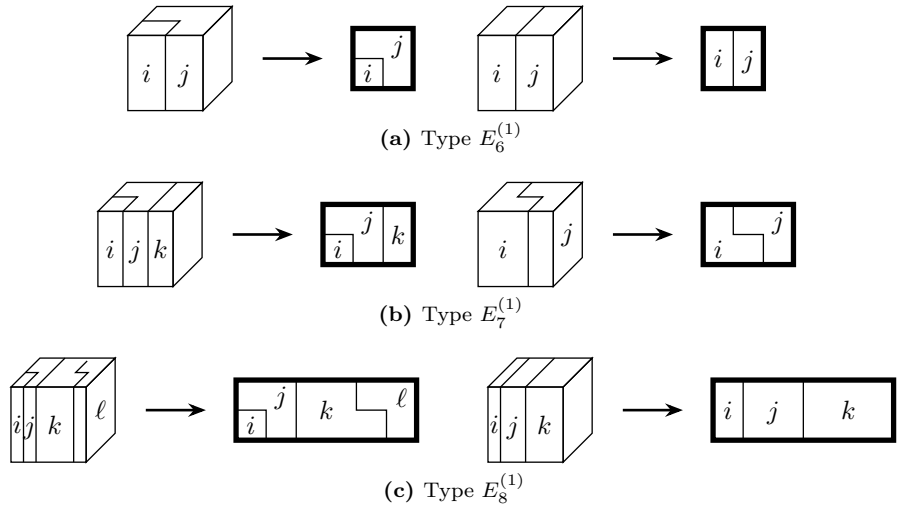
\begin{figure}[H]
    \centering
    \subfloat[Type $\Eaff{6}$]{
    \begin{tikzpicture}[scale=0.2, line width = 0.5]

\draw (0,0,0) -- ++(5,0,0) -- ++(0,5,0) -- ++(-5,0,0) -- cycle;
\draw (0,5,0) -- ++(0,0,-5) -- ++(5,0,0) -- ++(0,0,5) -- cycle;
\draw (5,0,0) -- ++(0,0,-5) -- ++(0,5,0) -- ++(0,0,5) -- cycle;

\draw (2.5,0,0) -- ++(0,5,0) -- ++(0,0,-2.5) -- ++(-2.5,0,0);
\node at (1.25,2.5,0) {$i$};
\node at (3.75,2.3,0) {$j$};

\draw[-{Stealth[scale=0.9]}, line width = 1] (8,2.5,-2.5) -- (12,2.5,-2.5);

\draw[line width = 2] (14,0.5,-2.5) rectangle ++(4,4,0);

\draw (16,0.5,-2.5) -- ++(0,2,0) -- ++(-2,0,0);
\node at (15,1.5,-2.5) {$i$};
\node at (17,3.3,-2.5) {$j$};

\end{tikzpicture}

\qquad

\begin{tikzpicture}[scale=0.2, line width = 0.5]

\draw (0,0,0) -- ++(5,0,0) -- ++(0,5,0) -- ++(-5,0,0) -- cycle;
\draw (0,5,0) -- ++(0,0,-5) -- ++(5,0,0) -- ++(0,0,5) -- cycle;
\draw (5,0,0) -- ++(0,0,-5) -- ++(0,5,0) -- ++(0,0,5) -- cycle;

\draw (2.5,0,0) -- ++(0,5,0) -- ++(0,0,-5);
\node at (1.25,2.5,0) {$i$};
\node at (3.75,2.3,0) {$j$};

\draw[-{Stealth[scale=0.9]}, line width = 1] (8,2.5,-2.5) -- (12,2.5,-2.5);

\draw[line width = 2] (14,0.5,-2.5) rectangle ++(4,4,0);

\draw (16,0.5,-2.5) -- ++(0,4,0);
\node at (15,2.5,-2.5) {$i$};
\node at (17,2.3,-2.5) {$j$};

\end{tikzpicture}
    }
    \\
    \subfloat[Type $\Eaff{7}$]{
    \begin{tikzpicture}[scale=0.2, line width = 0.5]

\draw (0,0,0) -- ++(5,0,0) -- ++(0,5,0) -- ++(-5,0,0) -- cycle;
\draw (0,5,0) -- ++(0,0,-5) -- ++(5,0,0) -- ++(0,0,5) -- cycle;
\draw (5,0,0) -- ++(0,0,-5) -- ++(0,5,0) -- ++(0,0,5) -- cycle;

\draw (5/3,0,0) -- ++(0,5,0) -- ++(0,0,-2.5) -- ++(-5/3,0,0);
\draw (10/3,0,0) -- ++(0,5,0) -- ++(0,0,-5);
\node at (5/6,2.5,0) {$i$};
\node at (2.5,2.3,0) {$j$};
\node at (25/6,2.5,0) {$k$};

\draw[-{Stealth[scale=0.9]}, line width = 1] (8,2.5,-2.5) -- (12,2.5,-2.5);

\draw[line width = 2] (14,0.5,-2.5) rectangle ++(6,4,0);

\draw (16,0.5,-2.5) -- ++(0,2,0) -- ++(-2,0,0);
\draw (18,0.5,-2.5) -- ++(0,4,0);
\node at (15,1.5,-2.5) {$i$};
\node at (17,3.3,-2.5) {$j$};
\node at (19,2.5,-2.5) {$k$};

\end{tikzpicture}

\qquad

\begin{tikzpicture}[scale=0.2, line width = 0.5]

\draw (0,0,0) -- ++(5,0,0) -- ++(0,5,0) -- ++(-5,0,0) -- cycle;
\draw (0,5,0) -- ++(0,0,-5) -- ++(5,0,0) -- ++(0,0,5) -- cycle;
\draw (5,0,0) -- ++(0,0,-5) -- ++(0,5,0) -- ++(0,0,5) -- cycle;

\draw (20/6,0,0) -- ++(0,5,0) -- ++(0,0,-2.5) -- ++(-5/3,0,0) -- ++(0,0,-2.5);
\node at (5/3,2.5,0) {$i$};
\node at (5,2.5,-2.5) {$j$};

\draw[-{Stealth[scale=0.9]}, line width = 1] (8,2.5,-2.5) -- (12,2.5,-2.5);

\draw[line width = 2] (14,0.5,-2.5) rectangle ++(6,4,0);

\draw (18,0.5,-2.5) -- ++(0,2,0) -- ++(-2,0,0) -- ++(0,2,0);
\node at (15,1.5,-2.5) {$i$};
\node at (19,3.3,-2.5) {$j$};

\end{tikzpicture}
    }
    \\
    \subfloat[Type $\Eaff{8}$]{
    \begin{tikzpicture}[scale=0.2, line width = 0.5]

\draw (0,0,0) -- ++(5,0,0) -- ++(0,5,0) -- ++(-5,0,0) -- cycle;
\draw (0,5,0) -- ++(0,0,-5) -- ++(5,0,0) -- ++(0,0,5) -- cycle;
\draw (5,0,0) -- ++(0,0,-5) -- ++(0,5,0) -- ++(0,0,5) -- cycle;

\draw (5/6,0,0) -- ++(0,5,0) -- ++(0,0,-2.5) -- ++(-5/6,0,0);
\draw (10/6,0,0) -- ++(0,5,0) -- ++(0,0,-5);
\draw (25/6,0,0) -- ++(0,5,0) -- ++(0,0,-2.5) -- ++(-5/6,0,0) -- ++(0,0,-2);
\node at (5/12,2.5,0) {$i$};
\node at (15/12,2.3,0) {$j$};
\node at (35/12,2.5,0) {$k$};
\node at (5,2.5,-2.5) {$\ell$};

\draw[-{Stealth[scale=0.9]}, line width = 1] (8,2.5,-2.5) -- (12,2.5,-2.5);

\draw[line width = 2] (14,0.5,-2.5) rectangle ++(12,4,0);

\draw (16,0.5,-2.5) -- ++(0,2,0) -- ++(-2,0,0);
\draw (18,0.5,-2.5) -- ++(0,4,0);
\draw (24,0.5,-2.5) -- ++(0,2,0) -- ++(-2,0,0) -- ++(0,2,0);
\node at (15,1.5,-2.5) {$i$};
\node at (17,3.3,-2.5) {$j$};
\node at (20,2.5,-2.5) {$k$};
\node at (25,3.3,-2.5) {$\ell$};

\end{tikzpicture}

\qquad

\begin{tikzpicture}[scale=0.2, line width = 0.5]

\draw (0,0,0) -- ++(5,0,0) -- ++(0,5,0) -- ++(-5,0,0) -- cycle;
\draw (0,5,0) -- ++(0,0,-5) -- ++(5,0,0) -- ++(0,0,5) -- cycle;
\draw (5,0,0) -- ++(0,0,-5) -- ++(0,5,0) -- ++(0,0,5) -- cycle;

\draw (5/6,0,0) -- ++(0,5,0) -- ++(0,0,-5);
\draw (15/6,0,0) -- ++(0,5,0) -- ++(0,0,-5);
\node at (5/12,2.5,0) {$i$};
\node at (20/12,2.3,0) {$j$};
\node at (45/12,2.5,0) {$k$};

\draw[-{Stealth[scale=0.9]}, line width = 1] (8,2.5,-2.5) -- (12,2.5,-2.5);

\draw[line width = 2] (14,0.5,-2.5) rectangle ++(12,4,0);

\draw (16,0.5,-2.5) -- ++(0,4,0);
\draw (20,0.5,-2.5) -- ++(0,4,0);
\node at (15,2.5,-2.5) {$i$};
\node at (18,2.3,-2.5) {$j$};
\node at (23,2.5,-2.5) {$k$};

\end{tikzpicture}
    }
    \caption{\hspace{.5em}Examples of displaying cut unit cubes}\label{Block splitting diagrams}
\end{figure}

Moreover, blocks in our diagrams shall be coloured according to (\ref{colours}) when they are contained in the relevant Young column or Young wall, and coloured white if they are not.

\subsubsection{Types \texorpdfstring{$\Eaff{6}$}{E6(1)} and \texorpdfstring{$\Eaff{7}$}{E7(1)}} \label{E6 and E7 Young columns}

Figure \ref{E6 and E7 Young column patterns} contains the Young column patterns for types $\Eaff{6}$ and $\Eaff{7}$.

\begin{figure}[H]
    \centering
    \subfloat[Type $\Eaff{6}$]{
    \qquad
    \begin{tikzpicture}[scale=0.18, line width = 0.5]

    \node at (2,-2) {\LARGE $\vdots$};

    \draw[fill=col0] (0,0) rectangle ++(2,2);
    \node[color=white] at (1,1) {$0$};
    \draw[fill=col4] (0,2) |- (4,4) |- (2,0) |- cycle;
    \node[color=white] at (3,3) {$4$};

    \draw[fill=col2] (0,4) rectangle ++(2,4);
    \node[color=white] at (1,6) {$2$};
    \draw[fill=col5] (2,4) rectangle ++(2,4);
    \node[color=white] at (3,6) {$5$};

    \draw[fill=col6] (2,10) rectangle ++(2,2);
    \node[color=white] at (3,11) {$6$};
    \draw[fill=col4] (0,8) |- (2,12) |- (4,10) |- cycle;
    \node[color=white] at (1,9) {$4$};

    \draw[fill=col3] (0,12) rectangle ++(2,4);
    \node[color=white] at (1,14) {$3$};
    \draw[fill=col5] (2,12) rectangle ++(2,4);
    \node[color=white] at (3,14) {$5$};

    \draw[fill=col1] (0,16) rectangle ++(2,2);
    \node[color=white] at (1,17) {$1$};
    \draw[fill=col4] (0,18) |- (4,20) |- (2,16) |- cycle;
    \node[color=white] at (3,19) {$4$};

    \draw[fill=col3] (0,20) rectangle ++(2,4);
    \node[color=white] at (1,22) {$3$};
    \draw[fill=col2] (2,20) rectangle ++(2,4);
    \node[color=white] at (3,22) {$2$};

    \draw[fill=col0] (2,26) rectangle ++(2,2);
    \node[color=white] at (3,27) {$0$};
    \draw[fill=col4] (0,24) |- (2,28) |- (4,26) |- cycle;
    \node[color=white] at (1,25) {$4$};

    \draw[fill=col5] (0,28) rectangle ++(2,4);
    \node[color=white] at (1,30) {$5$};
    \draw[fill=col2] (2,28) rectangle ++(2,4);
    \node[color=white] at (3,30) {$2$};

    \draw[fill=col6] (0,32) rectangle ++(2,2);
    \node[color=white] at (1,33) {$6$};
    \draw[fill=col4] (0,34) |- (4,36) |- (2,32) |- cycle;
    \node[color=white] at (3,35) {$4$};

    \draw[fill=col5] (0,36) rectangle ++(2,4);
    \node[color=white] at (1,38) {$5$};
    \draw[fill=col3] (2,36) rectangle ++(2,4);
    \node[color=white] at (3,38) {$3$};

    \draw[fill=col1] (2,42) rectangle ++(2,2);
    \node[color=white] at (3,43) {$1$};
    \draw[fill=col4] (0,40) |- (2,44) |- (4,42) |- cycle;
    \node[color=white] at (1,41) {$4$};

    \draw[fill=col2] (0,44) rectangle ++(2,4);
    \node[color=white] at (1,46) {$2$};
    \draw[fill=col3] (2,44) rectangle ++(2,4);
    \node[color=white] at (3,46) {$3$};

    \draw[line width = 2] (0,0) rectangle ++(4,4);
    \draw[line width = 2] (0,4) rectangle ++(4,4);
    \draw[line width = 2] (0,8) rectangle ++(4,4);
    \draw[line width = 2] (0,12) rectangle ++(4,4);
    \draw[line width = 2] (0,16) rectangle ++(4,4);
    \draw[line width = 2] (0,20) rectangle ++(4,4);
    \draw[line width = 2] (0,24) rectangle ++(4,4);
    \draw[line width = 2] (0,28) rectangle ++(4,4);
    \draw[line width = 2] (0,32) rectangle ++(4,4);
    \draw[line width = 2] (0,36) rectangle ++(4,4);
    \draw[line width = 2] (0,40) rectangle ++(4,4);
    \draw[line width = 2] (0,44) rectangle ++(4,4);

    \node at (2,51) {\LARGE $\vdots$};

    \draw[line width = 2] (4,-2) -- (4,50);
    \draw[line width = 2] (0,-2) -- (0,50);
    \end{tikzpicture}
    \qquad
    }
    ~~~~~~~~~
    \subfloat[Type $\Eaff{7}$]{
    \qquad
    \begin{tikzpicture}[scale=0.18, line width = 0.5]

    \node at (3,-2) {\LARGE $\vdots$};

    \draw[fill=col0] (0,0) rectangle ++(2,2);
    \node[color=white] at (1,1) {$0$};
    \draw[fill=col3] (0,2) |- (4,4) |- (2,0) |- cycle;
    \node[color=white] at (3,3) {$3$};
    \draw[fill=col2] (4,0) rectangle ++(2,4);
    \node[color=white] at (5,2) {$2$};

    \draw[fill=col1] (0,4) rectangle ++(2,4);
    \node[color=white] at (1,6) {$1$};
    \draw[fill=col4] (2,4) rectangle ++(4,4);
    \node[color=white] at (4,6) {$4$};

    \draw[fill=col5] (2,10) |- (6,12) |- (4,8) |- cycle;
    \node[color=white] at (5,11) {$5$};
    \draw[fill=col3] (0,8) |- (2,12) |- (4,10) |- cycle;
    \node[color=white] at (1,9) {$3$};

    \draw[fill=col4] (0,12) rectangle ++(4,4);
    \node[color=white] at (2,14) {$4$};
    \draw[fill=col6] (4,12) rectangle ++(2,4);
    \node[color=white] at (5,14) {$6$};

    \draw[fill=col5] (2,16) |- (4,20) |- (6,18) |- cycle;
    \node[color=white] at (3,17) {$5$};
    \draw[fill=col2] (0,16) rectangle ++(2,4);
    \node[color=white] at (1,18) {$2$};
    \draw[fill=col7] (4,18) rectangle ++(2,2);
    \node[color=white] at (5,19) {$7$};

    \draw[fill=col4] (0,20) rectangle ++(4,4);
    \node[color=white] at (2,22) {$4$};
    \draw[fill=col6] (4,20) rectangle ++(2,4);
    \node[color=white] at (5,22) {$6$};

    \draw[fill=col5] (2,26) |- (6,28) |- (4,24) |- cycle;
    \node[color=white] at (5,27) {$5$};
    \draw[fill=col3] (0,24) |- (2,28) |- (4,26) |- cycle;
    \node[color=white] at (1,25) {$3$};

    \draw[fill=col1] (0,28) rectangle ++(2,4);
    \node[color=white] at (1,30) {$1$};
    \draw[fill=col4] (2,28) rectangle ++(4,4);
    \node[color=white] at (4,30) {$4$};

    \draw[line width = 2] (0,0) rectangle ++(6,4);
    \draw[line width = 2] (0,4) rectangle ++(6,4);
    \draw[line width = 2] (0,8) rectangle ++(6,4);
    \draw[line width = 2] (0,12) rectangle ++(6,4);
    \draw[line width = 2] (0,16) rectangle ++(6,4);
    \draw[line width = 2] (0,20) rectangle ++(6,4);
    \draw[line width = 2] (0,24) rectangle ++(6,4);
    \draw[line width = 2] (0,28) rectangle ++(6,4);

    \node at (3,35) {\LARGE $\vdots$};

    \draw[line width = 2] (6,-2) -- (6,34);
    \draw[line width = 2] (0,-2) -- (0,34);
    \end{tikzpicture}
    \qquad
    }
    \caption{\hspace{.5em}Young column patterns for types $\Eaff{6}$ and $\Eaff{7}$}\label{E6 and E7 Young column patterns}
\end{figure}
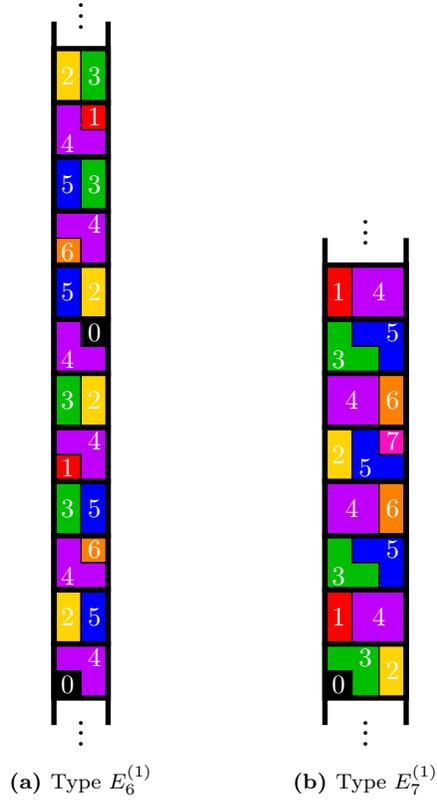

We now define the notion of a Young column inside each pattern, as well as an equivalence relation on these Young columns.

\begin{defn}
    In types $\Eaff{6}$ and $\Eaff{7}$ a Young column is an collection of blocks inside the Young column pattern such that
    \begin{itemize}
    \item the height of the blocks is bounded above, $\hfill \refstepcounter{equation}(\theequation)\label{columns bounded above}$
    \item there is no empty space below any block. $\hfill \refstepcounter{equation}(\theequation)\label{column gravity condition}$
    \end{itemize}
\end{defn}

\begin{eg}
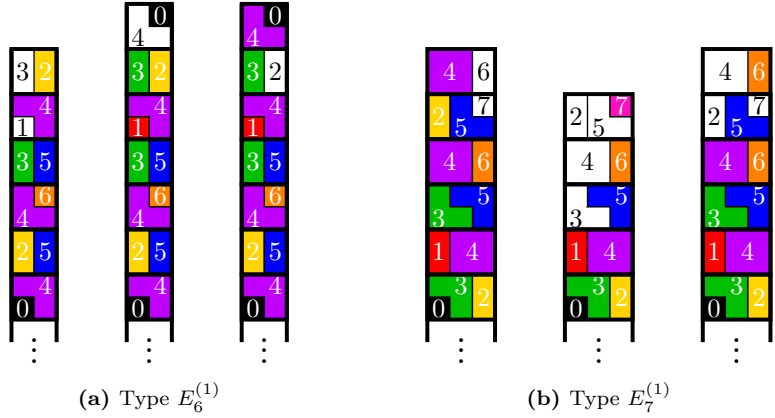
    Figure \ref{E6 and E7 Young column examples and non-examples} contains some examples and non-examples of Young columns.
    In each type, the first two diagrams represent valid stackings of blocks inside the relevant Young column pattern, while the rightmost column fails (\ref{column gravity condition}).
\end{eg}

\input{E6_and_E7_Young_column_examples_and_non-examples}

\begin{defn}
    Young columns are equivalent if they can be obtained from one another via vertical shift, or rotation by $180^{\circ}$ around the vertical axis.
\end{defn}

Note that from Figure \ref{E6 and E7 Young column patterns}, rotation shall only really be relevant in type $\Eaff{6}$ at this stage.
The set of (equivalence classes of) Young columns can be endowed with a crystal structure as follows.

\begin{defn}
    \begin{itemize}
        \item An $i$-block in a Young column $y$ is removable if removing it from $y$ gives another Young column.
        \item An $i$-block in the Young column pattern that is not in $y$ is addable if adding it to $y$ gives another Young column.
    \end{itemize}
\end{defn}

A quick check then verifies that the following endows the set of Young columns inside each pattern with the structure of an affine crystal.
\begin{itemize}
    \item The $\ft_{i}$ act on a Young column $y$ by adding an addable $i$-block if this exists, and sending to $0$ otherwise.
    \item Similarly, $\et_{i}$ removes a removable $i$-block from $y$ if possible, and maps to $0$ otherwise.
    \item As in (\ref{varphi and varepsilon maximum definitions}) let $\varphi_{i}(y) = \max\lbrace n ~\vert~ \ft_{i}^{n}y\neq 0\rbrace$ and $\varepsilon_{i}(y) = \max\lbrace n ~\vert~ \et_{i}^{n}y\neq 0\rbrace$.
    \item Specify the weight of a particular Young column whose only addable block is a $0$-block to be $\Lambda_{0}$.
    The weight function is then determined by extending via $\wt(\ft_{i}y) = \wt(y) - \alpha_{i}$ if $\ft_{i}y\not= 0$ and $\wt(\et_{i}y) = \wt(y) + \alpha_{i}$ if $\et_{i}y\not= 0$.
\end{itemize}

Furthermore, by projecting the weights to $\Pbar$ this descends to a classical crystal structure on the set of equivalence classes of Young columns, which we denote by $C_{6}$ and $C_{7}$ in types $\Eaff{6}$ and $\Eaff{7}$ respectively.
It is clear that the affinizations $(C_{6})_{\aff}$ and $(C_{7})_{\aff}$ recover the sets of Young columns together with the affine crystal structure above.
\\

From Figures \ref{B6 crystal graph} and \ref{B7 crystal graph} we see that the equivalence classes of Young columns inside our patterns provide us with combinatorial models our level $1$ perfect crystals $B_{6}$ and $B_{7}$.

\begin{prop} \label{E6 and E7 Young column isomorphisms}
    There are isomorphisms $\psi : B_{6} \rightarrow C_{6}$ and $\psi : B_{7} \rightarrow C_{7}$ of classical crystals.
\end{prop}

\begin{eg}
    Figure \ref{E6 and E7 ground state columns} gives the equivalence classes $\psi(b_{\Lambda_{i}})$ for all $i\in \Imin$ in each type, from which it is easy to deduce the remaining values of $\psi$ by adding and removing blocks.
\end{eg}

\begin{figure}[H]
    \centering
    \subfloat[Type $\Eaff{6}$]{
    \begin{tikzpicture}[scale=0.15, line width = 0.5]

    \node at (2,-10) {$\psi(b_{\Lambda_{0}})$};
    \node at (2,-6) {\LARGE $\vdots$};

    \draw[fill=col2] (0,-4) rectangle ++(2,4);
    \node[color=white] at (1,-2) {$2$};
    \draw[fill=col3] (2,-4) rectangle ++(2,4);
    \node[color=white] at (3,-2) {$3$};

    \node[color=black] at (1,1) {$0$};
    \draw[fill=col4] (0,2) |- (4,4) |- (2,0) |- cycle;
    \node[color=white] at (3,3) {$4$};

    \node[color=black] at (1,6) {$2$};
    \draw[fill=col5] (2,4) rectangle ++(2,4);
    \node[color=white] at (3,6) {$5$};

    \draw[fill=col6] (2,10) rectangle ++(2,2);
    \node[color=white] at (3,11) {$6$};
    \node[color=black] at (1,9) {$4$};

    \draw[line width = 1.5] (0,-4) rectangle ++(4,4);
    \draw[line width = 1.5] (0,0) rectangle ++(4,4);
    \draw[line width = 1.5] (0,4) rectangle ++(4,4);
    \draw[line width = 1.5] (0,8) rectangle ++(4,4);

    \draw[line width = 1.5] (4,-6) -- (4,12);
    \draw[line width = 1.5] (0,-6) -- (0,12);
    \end{tikzpicture}

    \qquad

    \begin{tikzpicture}[scale=0.15, line width = 0.5]

    \node at (2,-10) {$\psi(b_{\Lambda_{1}})$};
    \node at (2,-6) {\LARGE $\vdots$};

    \draw[fill=col3] (0,-4) rectangle ++(2,4);
    \node[color=white] at (1,-2) {$3$};
    \draw[fill=col5] (2,-4) rectangle ++(2,4);
    \node[color=white] at (3,-2) {$5$};

    \node[color=black] at (1,1) {$1$};
    \draw[fill=col4] (0,2) |- (4,4) |- (2,0) |- cycle;
    \node[color=white] at (3,3) {$4$};

    \node[color=black] at (1,6) {$3$};
    \draw[fill=col2] (2,4) rectangle ++(2,4);
    \node[color=white] at (3,6) {$2$};

    \draw[fill=col0] (2,10) rectangle ++(2,2);
    \node[color=white] at (3,11) {$0$};
    \node[color=black] at (1,9) {$4$};

    \draw[line width = 1.5] (0,-4) rectangle ++(4,4);
    \draw[line width = 1.5] (0,0) rectangle ++(4,4);
    \draw[line width = 1.5] (0,4) rectangle ++(4,4);
    \draw[line width = 1.5] (0,8) rectangle ++(4,4);

    \draw[line width = 1.5] (4,-6) -- (4,12);
    \draw[line width = 1.5] (0,-6) -- (0,12);
    \end{tikzpicture}

    \qquad

    \begin{tikzpicture}[scale=0.15, line width = 0.5]

    \node at (2,-10) {$\psi(b_{\Lambda_{6}})$};
    \node at (2,-6) {\LARGE $\vdots$};

    \draw[fill=col5] (0,-4) rectangle ++(2,4);
    \node[color=white] at (1,-2) {$5$};
    \draw[fill=col2] (2,-4) rectangle ++(2,4);
    \node[color=white] at (3,-2) {$2$};

    \node[color=black] at (1,1) {$6$};
    \draw[fill=col4] (0,2) |- (4,4) |- (2,0) |- cycle;
    \node[color=white] at (3,3) {$4$};

    \node[color=black] at (1,6) {$5$};
    \draw[fill=col3] (2,4) rectangle ++(2,4);
    \node[color=white] at (3,6) {$3$};

    \draw[fill=col1] (2,10) rectangle ++(2,2);
    \node[color=white] at (3,11) {$1$};
    \node[color=black] at (1,9) {$4$};

    \draw[line width = 1.5] (0,-4) rectangle ++(4,4);
    \draw[line width = 1.5] (0,0) rectangle ++(4,4);
    \draw[line width = 1.5] (0,4) rectangle ++(4,4);
    \draw[line width = 1.5] (0,8) rectangle ++(4,4);

    \draw[line width = 1.5] (4,-6) -- (4,12);
    \draw[line width = 1.5] (0,-6) -- (0,12);
    \end{tikzpicture}
    }
    ~~~~~~~~~
    \subfloat[Type $\Eaff{7}$]{
    \begin{tikzpicture}[scale=0.15, line width = 0.5]

    \node at (2,-10) {$\psi(b_{\Lambda_{0}})$};
    \node at (3,-6) {\LARGE $\vdots$};

    \draw[fill=col1] (0,-4) rectangle ++(2,4);
    \node[color=white] at (1,-2) {$1$};
    \draw[fill=col4] (2,-4) rectangle ++(4,4);
    \node[color=white] at (4,-2) {$4$};

    \node[color=black] at (1,1) {$0$};
    \draw[fill=col3] (0,2) |- (4,4) |- (2,0) |- cycle;
    \node[color=white] at (3,3) {$3$};
    \draw[fill=col2] (4,0) rectangle ++(2,4);
    \node[color=white] at (5,2) {$2$};

    \node[color=black] at (1,6) {$1$};
    \draw[fill=col4] (2,4) rectangle ++(4,4);
    \node[color=white] at (4,6) {$4$};

    \draw[fill=col5] (2,10) |- (6,12) |- (4,8) |- cycle;
    \node[color=white] at (5,11) {$5$};
    \node[color=black] at (1,9) {$3$};

    \node[color=black] at (2,14) {$4$};
    \draw[fill=col6] (4,12) rectangle ++(2,4);
    \node[color=white] at (5,14) {$6$};

    \draw[fill=none] (2,16) |- (4,20) |- (6,18) |- cycle;
    \draw[fill=col7] (4,18) rectangle ++(2,2);
    \node[color=black] at (3,17) {$5$};
    \node[color=black] at (1,18) {$2$};
    \node[color=white] at (5,19) {$7$};

    \draw[line width = 1.5] (0,-4) rectangle ++(6,4);
    \draw[line width = 1.5] (0,0) rectangle ++(6,4);
    \draw[line width = 1.5] (0,4) rectangle ++(6,4);
    \draw[line width = 1.5] (0,8) rectangle ++(6,4);
    \draw[line width = 1.5] (0,12) rectangle ++(6,4);
    \draw[line width = 1.5] (0,16) rectangle ++(6,4);

    \draw[line width = 1.5] (6,-6) -- (6,20);
    \draw[line width = 1.5] (0,-6) -- (0,20);
    \end{tikzpicture}

    \qquad

    \begin{tikzpicture}[scale=0.15, line width = 0.5]

    \node at (2,-10) {$\psi(b_{\Lambda_{7}})$};
    \node at (3,-6) {\LARGE $\vdots$};

    \draw[fill=col6] (0,-4) rectangle ++(2,4);
    \node[color=white] at (1,-2) {$6$};
    \draw[fill=col4] (2,-4) rectangle ++(4,4);
    \node[color=white] at (4,-2) {$4$};

    \node[color=black] at (1,1) {$7$};
    \draw[fill=col5] (0,2) |- (4,4) |- (2,0) |- cycle;
    \node[color=white] at (3,3) {$5$};
    \draw[fill=col2] (4,0) rectangle ++(2,4);
    \node[color=white] at (5,2) {$2$};

    \node[color=black] at (1,6) {$6$};
    \draw[fill=col4] (2,4) rectangle ++(4,4);
    \node[color=white] at (4,6) {$4$};

    \draw[fill=col3] (2,10) |- (6,12) |- (4,8) |- cycle;
    \node[color=white] at (5,11) {$3$};
    \node[color=black] at (1,9) {$5$};

    \node[color=black] at (2,14) {$4$};
    \draw[fill=col1] (4,12) rectangle ++(2,4);
    \node[color=white] at (5,14) {$1$};

    \draw[fill=none] (2,16) |- (4,20) |- (6,18) |- cycle;
    \draw[fill=col0] (4,18) rectangle ++(2,2);
    \node[color=black] at (3,17) {$3$};
    \node[color=black] at (1,18) {$2$};
    \node[color=white] at (5,19) {$0$};

    \draw[line width = 1.5] (0,-4) rectangle ++(6,4);
    \draw[line width = 1.5] (0,0) rectangle ++(6,4);
    \draw[line width = 1.5] (0,4) rectangle ++(6,4);
    \draw[line width = 1.5] (0,8) rectangle ++(6,4);
    \draw[line width = 1.5] (0,12) rectangle ++(6,4);
    \draw[line width = 1.5] (0,16) rectangle ++(6,4);

    \draw[line width = 1.5] (6,-6) -- (6,20);
    \draw[line width = 1.5] (0,-6) -- (0,20);
    \end{tikzpicture}
    }
    \caption{\hspace{.5em}Equivalence classes $\psi(b_{\Lambda_{i}})$ of Young columns in types $\Eaff{6}$ and $\Eaff{7}$}\label{E6 and E7 ground state columns}
\end{figure}
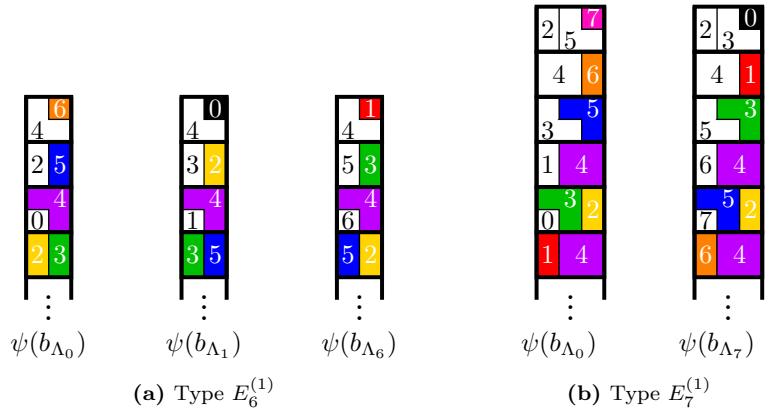

\subsubsection{Type \texorpdfstring{$\Eaff{8}$}{E8(1)}}

Figure \ref{E8 Young column pattern} contains the Young column pattern for type $\Eaff{8}$.

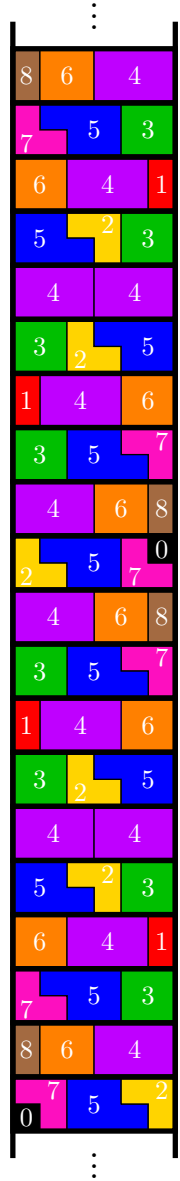
\begin{figure}[H]
    \centering
    \begin{tikzpicture}[scale=0.18, line width = 0.5]

    \node at (6,-2) {\LARGE $\vdots$};

    \draw[fill=col0] (0,0) rectangle ++(2,2);
    \node[color=white] at (1,1) {$0$};
    \draw[fill=col7] (0,2) |- (4,4) |- (2,0) |- cycle;
    \node[color=white] at (3,3) {$7$};
    \draw[fill=col5] (4,4) |- (10,0) |- (8,2) |- cycle;
    \node[color=white] at (6,2) {$5$};
    \draw[fill=col2] (8,2) |- (12,4) |- (10,0) |- cycle;
    \node[color=white] at (11,3) {$2$};

    \draw[fill=col8] (0,4) rectangle ++(2,4);
    \node[color=white] at (1,6) {$8$};
    \draw[fill=col6] (2,4) rectangle ++(4,4);
    \node[color=white] at (4,6) {$6$};
    \draw[fill=col4] (6,4) rectangle ++(6,4);
    \node[color=white] at (9,6) {$4$};

    \draw[fill=col5] (2,10) |- (8,12) |- (4,8) |- cycle;
    \node[color=white] at (6,10) {$5$};
    \draw[fill=col7] (0,8) |- (2,12) |- (4,10) |- cycle;
    \node[color=white] at (1,9) {$7$};
    \draw[fill=col3] (8,8) rectangle ++(4,4);
    \node[color=white] at (10,10) {$3$};

    \draw[fill=col6] (0,12) rectangle ++(4,4);
    \node[color=white] at (2,14) {$6$};
    \draw[fill=col4] (4,12) rectangle ++(6,4);
    \node[color=white] at (7,14) {$4$};
    \draw[fill=col1] (10,12) rectangle ++(2,4);
    \node[color=white] at (11,14) {$1$};

    \draw[fill=col2] (4,18) |- (8,20) |- (6,16) |- cycle;
    \node[color=white] at (7,19) {$2$};
    \draw[fill=col5] (0,16) |- (4,20) |- (6,18) |- cycle;
    \node[color=white] at (2,18) {$5$};
    \draw[fill=col3] (8,16) rectangle ++(4,4);
    \node[color=white] at (10,18) {$3$};

    \draw[fill=col4] (0,20) rectangle ++(6,4);
    \node[color=white] at (3,22) {$4$};
    \draw[fill=col4] (6,20) rectangle ++(6,4);
    \node[color=white] at (9,22) {$4$};

    \draw[fill=col3] (0,24) rectangle ++(4,4);
    \node[color=white] at (2,26) {$3$};
    \draw[fill=col2] (4,24) |- (6,28) |- (8,26) |- cycle;
    \node[color=white] at (5,25) {$2$};
    \draw[fill=col5] (6,26) |- (12,28) |- (8,24) |- cycle;
    \node[color=white] at (10,26) {$5$};

    \draw[fill=col6] (8,28) rectangle ++(4,4);
    \node[color=white] at (10,30) {$6$};
    \draw[fill=col4] (2,28) rectangle ++(6,4);
    \node[color=white] at (5,30) {$4$};
    \draw[fill=col1] (0,28) rectangle ++(2,4);
    \node[color=white] at (1,30) {$1$};

    \draw[fill=col7] (8,34) |- (12,36) |- (10,32) |- cycle;
    \node[color=white] at (11,35) {$7$};
    \draw[fill=col5] (4,32) |- (8,36) |- (10,34) |- cycle;
    \node[color=white] at (6,34) {$5$};
    \draw[fill=col3] (0,32) rectangle ++(4,4);
    \node[color=white] at (2,34) {$3$};

    \draw[fill=col6] (6,36) rectangle ++(4,4);
    \node[color=white] at (8,38) {$6$};
    \draw[fill=col4] (0,36) rectangle ++(6,4);
    \node[color=white] at (3,38) {$4$};
    \draw[fill=col8] (10,36) rectangle ++(2,4);
    \node[color=white] at (11,38) {$8$};

    \draw[fill=col2] (0,40) |- (2,44) |- (4,42) |- cycle;
    \node[color=white] at (1,41) {$2$};
    \draw[fill=col5] (2,42) |- (8,44) |- (4,40) |- cycle;
    \node[color=white] at (6,42) {$5$};
    \draw[fill=col7] (8,40) |- (10,44) |- (12,42) |- cycle;
    \node[color=white] at (9,41) {$7$};
    \draw[fill=col0] (10,42) rectangle ++(2,2);
    \node[color=white] at (11,43) {$0$};

    \draw[fill=col6] (6,44) rectangle ++(4,4);
    \node[color=white] at (8,46) {$6$};
    \draw[fill=col4] (0,44) rectangle ++(6,4);
    \node[color=white] at (3,46) {$4$};
    \draw[fill=col8] (10,44) rectangle ++(2,4);
    \node[color=white] at (11,46) {$8$};

    \draw[fill=col7] (8,50) |- (12,52) |- (10,48) |- cycle;
    \node[color=white] at (11,51) {$7$};
    \draw[fill=col5] (4,48) |- (8,52) |- (10,50) |- cycle;
    \node[color=white] at (6,50) {$5$};
    \draw[fill=col3] (0,48) rectangle ++(4,4);
    \node[color=white] at (2,50) {$3$};

    \draw[fill=col6] (8,52) rectangle ++(4,4);
    \node[color=white] at (10,54) {$6$};
    \draw[fill=col4] (2,52) rectangle ++(6,4);
    \node[color=white] at (5,54) {$4$};
    \draw[fill=col1] (0,52) rectangle ++(2,4);
    \node[color=white] at (1,54) {$1$};

    \draw[fill=col3] (0,56) rectangle ++(4,4);
    \node[color=white] at (2,58) {$3$};
    \draw[fill=col2] (4,56) |- (6,60) |- (8,58) |- cycle;
    \node[color=white] at (5,57) {$2$};
    \draw[fill=col5] (6,58) |- (12,60) |- (8,56) |- cycle;
    \node[color=white] at (10,58) {$5$};

    \draw[fill=col4] (0,60) rectangle ++(6,4);
    \node[color=white] at (3,62) {$4$};
    \draw[fill=col4] (6,60) rectangle ++(6,4);
    \node[color=white] at (9,62) {$4$};

    \draw[fill=col2] (4,66) |- (8,68) |- (6,64) |- cycle;
    \node[color=white] at (7,67) {$2$};
    \draw[fill=col5] (0,64) |- (4,68) |- (6,66) |- cycle;
    \node[color=white] at (2,66) {$5$};
    \draw[fill=col3] (8,64) rectangle ++(4,4);
    \node[color=white] at (10,66) {$3$};

    \draw[fill=col6] (0,68) rectangle ++(4,4);
    \node[color=white] at (2,70) {$6$};
    \draw[fill=col4] (4,68) rectangle ++(6,4);
    \node[color=white] at (7,70) {$4$};
    \draw[fill=col1] (10,68) rectangle ++(2,4);
    \node[color=white] at (11,70) {$1$};

    \draw[fill=col5] (2,74) |- (8,76) |- (4,72) |- cycle;
    \node[color=white] at (6,74) {$5$};
    \draw[fill=col7] (0,72) |- (2,76) |- (4,74) |- cycle;
    \node[color=white] at (1,73) {$7$};
    \draw[fill=col3] (8,72) rectangle ++(4,4);
    \node[color=white] at (10,74) {$3$};

    \draw[fill=col8] (0,76) rectangle ++(2,4);
    \node[color=white] at (1,78) {$8$};
    \draw[fill=col6] (2,76) rectangle ++(4,4);
    \node[color=white] at (4,78) {$6$};
    \draw[fill=col4] (6,76) rectangle ++(6,4);
    \node[color=white] at (9,78) {$4$};

    \draw[line width = 2] (0,0) rectangle ++(12,4);
    \draw[line width = 2] (0,4) rectangle ++(12,4);
    \draw[line width = 2] (0,8) rectangle ++(12,4);
    \draw[line width = 2] (0,12) rectangle ++(12,4);
    \draw[line width = 2] (0,16) rectangle ++(12,4);
    \draw[line width = 2] (0,20) rectangle ++(12,4);
    \draw[line width = 2] (0,24) rectangle ++(12,4);
    \draw[line width = 2] (0,28) rectangle ++(12,4);
    \draw[line width = 2] (0,32) rectangle ++(12,4);
    \draw[line width = 2] (0,36) rectangle ++(12,4);
    \draw[line width = 2] (0,40) rectangle ++(12,4);
    \draw[line width = 2] (0,44) rectangle ++(12,4);
    \draw[line width = 2] (0,48) rectangle ++(12,4);
    \draw[line width = 2] (0,52) rectangle ++(12,4);
    \draw[line width = 2] (0,56) rectangle ++(12,4);
    \draw[line width = 2] (0,60) rectangle ++(12,4);
    \draw[line width = 2] (0,64) rectangle ++(12,4);
    \draw[line width = 2] (0,68) rectangle ++(12,4);
    \draw[line width = 2] (0,72) rectangle ++(12,4);
    \draw[line width = 2] (0,76) rectangle ++(12,4);

    \node at (6,83) {\LARGE $\vdots$};

    \draw[line width = 2] (12,-2) -- (12,82);
    \draw[xshift=-12cm, line width = 2] (12,-2) -- (12,82);
    \end{tikzpicture}
    \caption{\hspace{.5em}Young column pattern for type $\Eaff{8}$}\label{E8 Young column pattern}
\end{figure}

The Young columns -- together with their equivalence relation, addable blocks and removable blocks -- are defined exactly as in Section \ref{E6 and E7 Young columns}, except that we need to add two extra valid equivalence classes of Young columns.
These are displayed in Figure \ref{E8 additional Young columns}, and shall correspond to $x_{\pm\alpha_{2}}$.

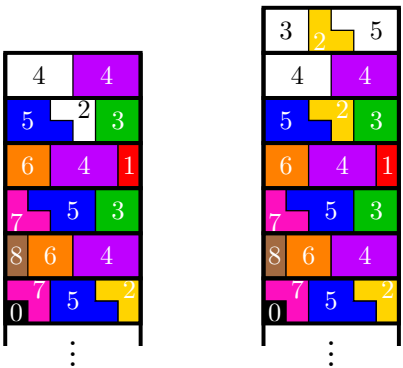
\begin{figure}[H]
    \centering
    \begin{tikzpicture}[scale=0.15, line width = 0.5]

    \node at (6,-2) {\LARGE $\vdots$};

    \draw[fill=col0] (0,0) rectangle ++(2,2);
    \node[color=white] at (1,1) {$0$};
    \draw[fill=col7] (0,2) |- (4,4) |- (2,0) |- cycle;
    \node[color=white] at (3,3) {$7$};
    \draw[fill=col5] (4,4) |- (10,0) |- (8,2) |- cycle;
    \node[color=white] at (6,2) {$5$};
    \draw[fill=col2] (8,2) |- (12,4) |- (10,0) |- cycle;
    \node[color=white] at (11,3) {$2$};

    \draw[fill=col8] (0,4) rectangle ++(2,4);
    \node[color=white] at (1,6) {$8$};
    \draw[fill=col6] (2,4) rectangle ++(4,4);
    \node[color=white] at (4,6) {$6$};
    \draw[fill=col4] (6,4) rectangle ++(6,4);
    \node[color=white] at (9,6) {$4$};

    \draw[fill=col5] (2,10) |- (8,12) |- (4,8) |- cycle;
    \node[color=white] at (6,10) {$5$};
    \draw[fill=col7] (0,8) |- (2,12) |- (4,10) |- cycle;
    \node[color=white] at (1,9) {$7$};
    \draw[fill=col3] (8,8) rectangle ++(4,4);
    \node[color=white] at (10,10) {$3$};

    \draw[fill=col6] (0,12) rectangle ++(4,4);
    \node[color=white] at (2,14) {$6$};
    \draw[fill=col4] (4,12) rectangle ++(6,4);
    \node[color=white] at (7,14) {$4$};
    \draw[fill=col1] (10,12) rectangle ++(2,4);
    \node[color=white] at (11,14) {$1$};

    \node[color=black] at (7,19) {$2$};
    \draw[fill=col5] (0,16) |- (4,20) |- (6,18) |- cycle;
    \node[color=white] at (2,18) {$5$};
    \draw[fill=col3] (8,16) rectangle ++(4,4);
    \node[color=white] at (10,18) {$3$};

    \node[color=black] at (3,22) {$4$};
    \draw[fill=col4] (6,20) rectangle ++(6,4);
    \node[color=white] at (9,22) {$4$};

    \draw[line width = 1.5] (0,0) rectangle ++(12,4);
    \draw[line width = 1.5] (0,4) rectangle ++(12,4);
    \draw[line width = 1.5] (0,8) rectangle ++(12,4);
    \draw[line width = 1.5] (0,12) rectangle ++(12,4);
    \draw[line width = 1.5] (0,16) rectangle ++(12,4);
    \draw[line width = 1.5] (0,20) rectangle ++(12,4);

    \draw[line width = 1.5] (12,-2) -- (12,24);
    \draw[xshift=-12cm, line width = 1.5] (12,-2) -- (12,24);
    \end{tikzpicture}
    \qquad \qquad
    \begin{tikzpicture}[scale=0.15, line width = 0.5]

    \node at (6,-2) {\LARGE $\vdots$};

    \draw[fill=col0] (0,0) rectangle ++(2,2);
    \node[color=white] at (1,1) {$0$};
    \draw[fill=col7] (0,2) |- (4,4) |- (2,0) |- cycle;
    \node[color=white] at (3,3) {$7$};
    \draw[fill=col5] (4,4) |- (10,0) |- (8,2) |- cycle;
    \node[color=white] at (6,2) {$5$};
    \draw[fill=col2] (8,2) |- (12,4) |- (10,0) |- cycle;
    \node[color=white] at (11,3) {$2$};

    \draw[fill=col8] (0,4) rectangle ++(2,4);
    \node[color=white] at (1,6) {$8$};
    \draw[fill=col6] (2,4) rectangle ++(4,4);
    \node[color=white] at (4,6) {$6$};
    \draw[fill=col4] (6,4) rectangle ++(6,4);
    \node[color=white] at (9,6) {$4$};

    \draw[fill=col5] (2,10) |- (8,12) |- (4,8) |- cycle;
    \node[color=white] at (6,10) {$5$};
    \draw[fill=col7] (0,8) |- (2,12) |- (4,10) |- cycle;
    \node[color=white] at (1,9) {$7$};
    \draw[fill=col3] (8,8) rectangle ++(4,4);
    \node[color=white] at (10,10) {$3$};

    \draw[fill=col6] (0,12) rectangle ++(4,4);
    \node[color=white] at (2,14) {$6$};
    \draw[fill=col4] (4,12) rectangle ++(6,4);
    \node[color=white] at (7,14) {$4$};
    \draw[fill=col1] (10,12) rectangle ++(2,4);
    \node[color=white] at (11,14) {$1$};

    \draw[fill=col2] (4,18) |- (8,20) |- (6,16) |- cycle;
    \node[color=white] at (7,19) {$2$};
    \draw[fill=col5] (0,16) |- (4,20) |- (6,18) |- cycle;
    \node[color=white] at (2,18) {$5$};
    \draw[fill=col3] (8,16) rectangle ++(4,4);
    \node[color=white] at (10,18) {$3$};

    \node[color=black] at (3,22) {$4$};
    \draw[fill=col4] (6,20) rectangle ++(6,4);
    \node[color=white] at (9,22) {$4$};

    \node[color=black] at (2,26) {$3$};
    \draw[fill=col2] (4,24) |- (6,28) |- (8,26) |- cycle;
    \node[color=white] at (5,25) {$2$};
    \node[color=black] at (10,26) {$5$};

    \draw[line width = 1.5] (0,0) rectangle ++(12,4);
    \draw[line width = 1.5] (0,4) rectangle ++(12,4);
    \draw[line width = 1.5] (0,8) rectangle ++(12,4);
    \draw[line width = 1.5] (0,12) rectangle ++(12,4);
    \draw[line width = 1.5] (0,16) rectangle ++(12,4);
    \draw[line width = 1.5] (0,20) rectangle ++(12,4);
    \draw[line width = 1.5] (0,24) rectangle ++(12,4);

    \draw[line width = 1.5] (12,-2) -- (12,28);
    \draw[xshift=-12cm, line width = 1.5] (12,-2) -- (12,28);
    \end{tikzpicture}
    \caption{\hspace{.5em}The additional equivalence classes of Young columns in type $\Eaff{8}$}\label{E8 additional Young columns}
\end{figure}

As in Section \ref{E6 and E7 Young columns} the set of Young columns has the structure of an affine crystal, but here we need to be more careful with our definitions of $\et_{i}$ and $\ft_{i}$.
This is to account for the fact that our $B_{8}$ crystal has a more complex structure than $B_{6}$ or $B_{7}$, for example with $i$-strings of length greater than $1$.

\begin{defn} \label{E8 Young column crystal structure definition}
    $\ft_{i}$ acts on a Young column by adding an addable $i$-block if it exists and mapping to $0$ otherwise, with the following caveats.
    \begin{itemize}
        \item If there are two addable $i$-blocks then $\ft_{i}$ adds the higher one for $i\not= 4$, and the one lying on top of a $5$-block for $i = 4$.
        \item If a Young column $y$ is obtained by adding an $i$-block to some other Young column under the rules above, and moreover has an addable $i$-block itself, then
        $\ft_{j}(y) = \et_{j}(y) = 0$ for all $j\not= i$.
    \end{itemize}
    Similarly, $\et_{i}$ acts on a Young column by removing a removable $i$-block if it exists and mapping to $0$ otherwise, with the following caveats.
    \begin{itemize}
        \item If there are two removable $i$-blocks then $\et_{i}$ removes the lower one for $i\not= 4$, and the one lying on top of a $3$-block for $i = 4$.
        \item If a Young column $y$ is obtained by removing an $i$-block from some other Young column under the rules above, and moreover has a removable $i$-block itself, then
        $\ft_{j}(y) = \et_{j}(y) = 0$ for all $j\not= i$.
    \end{itemize}
\end{defn}

\begin{rmk}
    Fan-Han-Kang-Shin similarly require extra conditions for their Young column models in types $D_{4}^{(3)}$ and $G_{2}^{(1)}$, given in Definition 3.1 (2) and Remark 3.4 of \cite{FHKS23}.
\end{rmk}

As in types $\Eaff{6}$ and $\Eaff{7}$ this affine crystal structure descends to a classical crystal structure on the set $C_{8}$ of equivalence classes of Young columns, by projecting the weights to $\Pbar$.
And once again, the affinization $(C_{8})_{\aff}$ recovers the original affine crystal.
\\

From Appendix \ref{B8 crystal graph appendix} and the definitions above, we see that the equivalence classes of Young columns provide a combinatorial realization of the level $1$ perfect crystal $B_{8}$.

\begin{prop} \label{E8 Young column isomorphism}
    There is an isomorphism $\psi : B_{8} \rightarrow C_{8}$ of classical crystals.
\end{prop}

Notice that Lemma \ref{0-arrows lemma} can therefore be concretely seen in the Young column pattern of Figure \ref{E8 Young column pattern} by looking at the dependency between the $0$-blocks and the $1$-blocks.

\begin{rmk}
    The caveats in Definition \ref{E8 Young column crystal structure definition}, together with the additional Young columns in Figure \ref{E8 additional Young columns}, ensure in particular that our $C_{8}$ crystal:
    \begin{itemize}
        \item has well-defined $\et_{i}$ and $\ft_{i}$ maps,
        \item contains $i$-strings of length $2$ corresponding to
        $x_{\alpha_{i}}\xrightarrow{i} y_{i} \xrightarrow{i} x_{-\alpha_{i}}$
        and
        $x_{-\theta}\xrightarrow{0} \emptyset \xrightarrow{0} x_{\theta}$ with no other arrows incident to $y_{i}$ and $\emptyset$,
        \item accurately models the $B_{8}$ section in Figure \ref{B8 example section a} with Figure \ref{C8 example section a}.
    \end{itemize}
\end{rmk}

We conclude this section with some example pieces of $B_{8}$ and their images under $\psi$, from which all values of $\psi$ can be derived using the crystal structure on $C_{8}$ and Appendix \ref{B8 crystal graph appendix}.
These examples should also make clear the role of the additional Young columns from Figure \ref{E8 additional Young columns} within the crystal, as well as the ordering rules for adding and removing blocks.

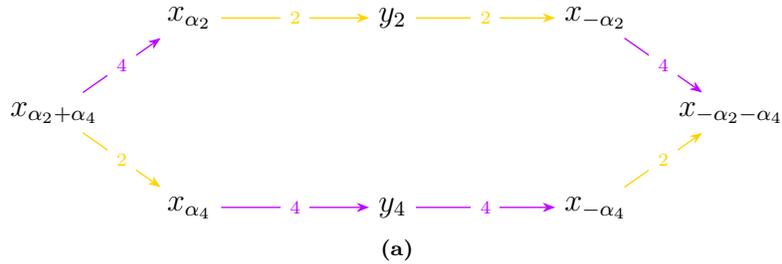
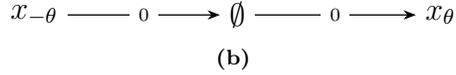
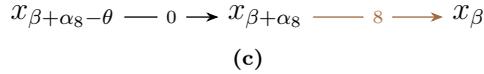
\begin{figure}[H]
    \centering
    \subfloat[\label{B8 example section a}]{
    \begin{tikzpicture}[-{Stealth[scale=0.9]}, scale=0.18, line width = 0.5]
    \node(0) at (0,0) {\large $x_{\alpha_{2} + \alpha_{4}}$};
    \node(1u) at (10,7) {\large $x_{\alpha_{2}}$};
    \node(1d) at (10,-7) {\large $x_{\alpha_{4}}$};
    \node(2u) at (25,7) {\large $y_{2}$};
    \node(2d) at (25,-7) {\large $y_{4}$};
    \node(3u) at (40,7) {\large $x_{-\alpha_{2}}$};
    \node(3d) at (40,-7) {\large $x_{-\alpha_{4}}$};
    \node(4) at (50,0) {\large $x_{-\alpha_{2} - \alpha_{4}}$};

    \tikzstyle{every node}=[midway, fill=white]
    \draw[col4] (0) -- (1u) node{$\scriptstyle 4$};
    \draw[col4] (1d) -- (2d) node{$\scriptstyle 4$};
    \draw[col4] (2d) -- (3d) node{$\scriptstyle 4$};
    \draw[col4] (3u) -- (4) node{$\scriptstyle 4$};
    \draw[col2] (0) -- (1d) node{$\scriptstyle 2$};
    \draw[col2] (1u) -- (2u) node{$\scriptstyle 2$};
    \draw[col2] (2u) -- (3u) node{$\scriptstyle 2$};
    \draw[col2] (3d) -- (4) node{$\scriptstyle 2$};
    \end{tikzpicture}
    }
    \\
    \bigskip
    \subfloat[]{
    \begin{tikzpicture}[-{Stealth[scale=0.9]}, scale=0.18, line width = 0.5]
    \node(0) at (0,0) {\large $x_{-\theta}$};
    \node(1) at (15,0) {\large $\emptyset$};
    \node(2) at (30,0) {\large $x_{\theta}$};

    \tikzstyle{every node}=[midway, fill=white]
    \draw[col0] (0) -- (1) node{$\scriptstyle 0$};
    \draw[col0] (1) -- (2) node{$\scriptstyle 0$};
    \end{tikzpicture}
    }
    \\
    \bigskip
    \subfloat[]{
    \begin{tikzpicture}[-{Stealth[scale=0.9]}, scale=0.18, line width = 0.5]
    \node(0) at (0,0) {\large $x_{\beta+\alpha_{8}-\theta}$};
    \node(1) at (15,0) {\large $x_{\beta+\alpha_{8}}$};
    \node(2) at (30,0) {\large $x_{\beta}$};

    \tikzstyle{every node}=[midway, fill=white]
    \draw[col0] (0) -- (1) node{$\scriptstyle 0$};
    \draw[col8] (1) -- (2) node{$\scriptstyle 8$};
    \end{tikzpicture}
    }
    \caption{\hspace{.5em}Example sections of $B_{8}$}\label{E8 example sections of B8}
\end{figure}

\input{E8_example_sections_of_C8}

\section{Level 1 irreducible highest weight crystals in type E} \label{Irreducible highest weight section}

Recall the path realization of the level $1$ irreducible highest weight crystals from Section \ref{Preliminaries}, in which the highest weight element $u_{\lambda}\in B(\lambda)$ corresponds to a ground state sequence $\pb_{\lambda}\in \Pcal(\lambda)$.
Table \ref{highest weight ground state sequences table} gives the sequence $\pb_{\lambda}$ in each case, with entries in the relevant level $1$ perfect crystal $B_{6}$, $B_{7}$ or $B_{8}$ from Section \ref{Perfect crystal section}.
It also contains the values of the energy function on consecutive factors, which is required for both the affine weight formula (\ref{affine weights on paths}) on $\Pcal(\lambda)$ and the ensuing Young wall realization.

\begin{table}[H]
\begin{center}
\begin{tabular}{|c|c|c|c|}
\hline
Type &
\begin{tabular}{@{}c@{}}
$\lambda\in\Pbar^{+}$ \\[1pt] of level $1$
\end{tabular}
&
\begin{tabular}{@{}c@{}}
Ground state sequence \\[1pt] $\pb_{\lambda} = (b_{r})_{r=0}^{\infty}$
\end{tabular}
&
$(H(b_{r+1}\otimes b_{r}))_{r=0}^{\infty}$
\\
\hline
& & & \\[-11pt]
\hline
& & & \\[-11pt]
$\Eaff{6}$
&
\begin{tabular}{@{}c@{}}
$\Lambda_{0}$ \\[1pt] $\Lambda_{1}$ \\[1pt] $\Lambda_{6}$
\end{tabular}
&
\begin{tabular}{@{}c@{}}
$\dots \otimes \overline{0}1 \otimes \overline{1}6 \otimes \overline{6}0 \otimes \overline{0}1 \otimes \overline{1}6 \otimes \overline{6}0$ \\[1pt]
$\dots \otimes \overline{1}6 \otimes \overline{6}0 \otimes \overline{0}1 \otimes \overline{1}6 \otimes \overline{6}0 \otimes \overline{0}1$ \\[1pt]
$\dots \otimes \overline{6}0 \otimes \overline{0}1 \otimes \overline{1}6 \otimes \overline{6}0 \otimes \overline{0}1 \otimes \overline{1}6$
\end{tabular}
&
\begin{tabular}{@{}c@{}}
$(\dots,0,2,2,0,2,2)$ \\[1pt]
$(\dots,2,2,0,2,2,0)$ \\[1pt]
$(\dots,2,0,2,2,0,2)$
\end{tabular} \\
\hline
& & & \\[-11pt]
$\Eaff{7}$
&
\begin{tabular}{@{}c@{}}
$\Lambda_{0}$ \\[1pt] $\Lambda_{7}$
\end{tabular}
&
\begin{tabular}{@{}c@{}}
$\dots \otimes \overline{0}7 \otimes \overline{7}0 \otimes \overline{0}7 \otimes \overline{7}0 \otimes \overline{0}7 \otimes \overline{7}0$ \\[1pt]
$\dots \otimes \overline{7}0 \otimes \overline{0}7 \otimes \overline{7}0 \otimes \overline{0}7 \otimes \overline{7}0 \otimes \overline{0}7$
\end{tabular}
&
\begin{tabular}{@{}c@{}}
$(\dots,0,3,0,3,0,3)$ \\[1pt]
$(\dots,3,0,3,0,3,0)$
\end{tabular} \\
\hline
& & & \\[-11pt]
$\Eaff{8}$
&
$\Lambda_{0}$
&
$\dots \otimes \hspace{.25em} \emptyset \hspace{.25em} \otimes \hspace{.25em} \emptyset \hspace{.25em} \otimes \hspace{.25em} \emptyset \hspace{.25em} \otimes \hspace{.25em} \emptyset \hspace{.25em} \otimes \hspace{.25em} \emptyset \hspace{.25em} \otimes \hspace{.25em} \emptyset \hspace{.25em}$
&
$(\dots,2,2,2,2,2,2)$
\\[1pt]
\hline
\end{tabular}
\caption{\hspace{.5em}Ground state sequences for level $1$ irreducible highest weight crystals}\label{highest weight ground state sequences table}
\end{center}
\end{table}

\begin{rmk} \label{removing first columns of reduced Young walls remark}
    In type $\Eaff{6}$ we see that $\pb_{\Lambda_{6}}$ and $\pb_{\Lambda_{1}}$ are just $\pb_{\Lambda_{0}}$ without the first one or two entries, while removing the first entry of $\pb_{\Lambda_{0}}$ in type $\Eaff{7}$ gives $\pb_{\Lambda_{7}}$.
    It follows that Young wall models for all level $1$ irreducible highest weight crystals $B(\Lambda_{i})$ can be obtained by simply removing the first one or two columns from a Young wall model for $B(\Lambda_{0})$.
    For the remainder of this section we shall therefore only deal with the $B(\Lambda_{0})$ case, unless stated otherwise.
\end{rmk}

As shorthand we shall often denote each of the crystals $B_{6}$, $B_{7}$ and $B_{8}$ by $B$, provided that it is clear which type(s) we are referring to.
Similarly, we shall denote each of the crystals $C_{6}$, $C_{7}$ and $C_{8}$ of equivalence classes of Young columns by $C$.
\\

Recall that equivalence classes of Young columns inside the relevant pattern from Figures \ref{E6 and E7 Young column patterns} and \ref{E8 Young column pattern} provide a combinatorial model for $B$, and removing the equivalence relation gives a model for $\Baff$.
Pick some representative Young column in the equivalence class $\psi(b_{r})$ corresponding to each entry of $\pb_{\Lambda_{0}}$ (given in Figures \ref{E6 and E7 ground state columns} and \ref{E8 second example section of C8}).
Then lining them up at the same height and orientation -- so they occupy the same spaces within each vertical strip of unit cubes -- we obtain the following \emph{Young column patterns} and \emph{ground state walls}.

\begin{figure}[H]
    \centering
    \begin{tabular}{cc}
    \adjustbox{valign=m}{
    \begin{tabular}{@{}c@{}}
    \subfloat[Type $\Eaff{6}$]{
    \input{E6_Young_wall_pattern}
    } \\
    \subfloat[Type $\Eaff{7}$]{
    \begin{tikzpicture}[scale=0.15, line width = 0.5]

    \node at (3,-2) {\LARGE $\vdots$};

    \draw[fill=col0] (0,0) rectangle ++(2,2);
    \node[color=white] at (1,1) {$0$};
    \draw[fill=col3] (0,2) |- (4,4) |- (2,0) |- cycle;
    \node[color=white] at (3,3) {$3$};
    \draw[fill=col2] (4,0) rectangle ++(2,4);
    \node[color=white] at (5,2) {$2$};

    \draw[fill=col1] (0,4) rectangle ++(2,4);
    \node[color=white] at (1,6) {$1$};
    \draw[fill=col4] (2,4) rectangle ++(4,4);
    \node[color=white] at (4,6) {$4$};

    \draw[fill=col5] (2,10) |- (6,12) |- (4,8) |- cycle;
    \node[color=white] at (5,11) {$5$};
    \draw[fill=col3] (0,8) |- (2,12) |- (4,10) |- cycle;
    \node[color=white] at (1,9) {$3$};

    \draw[fill=col4] (0,12) rectangle ++(4,4);
    \node[color=white] at (2,14) {$4$};
    \draw[fill=col6] (4,12) rectangle ++(2,4);
    \node[color=white] at (5,14) {$6$};

    \draw[fill=col5] (2,16) |- (4,20) |- (6,18) |- cycle;
    \node[color=white] at (3,17) {$5$};
    \draw[fill=col2] (0,16) rectangle ++(2,4);
    \node[color=white] at (1,18) {$2$};
    \draw[fill=col7] (4,18) rectangle ++(2,2);
    \node[color=white] at (5,19) {$7$};

    \draw[fill=col4] (0,20) rectangle ++(4,4);
    \node[color=white] at (2,22) {$4$};
    \draw[fill=col6] (4,20) rectangle ++(2,4);
    \node[color=white] at (5,22) {$6$};

    \draw[fill=col5] (2,26) |- (6,28) |- (4,24) |- cycle;
    \node[color=white] at (5,27) {$5$};
    \draw[fill=col3] (0,24) |- (2,28) |- (4,26) |- cycle;
    \node[color=white] at (1,25) {$3$};

    \draw[fill=col1] (0,28) rectangle ++(2,4);
    \node[color=white] at (1,30) {$1$};
    \draw[fill=col4] (2,28) rectangle ++(4,4);
    \node[color=white] at (4,30) {$4$};

    \draw[line width = 1.5] (0,0) rectangle ++(6,4);
    \draw[line width = 1.5] (0,4) rectangle ++(6,4);
    \draw[line width = 1.5] (0,8) rectangle ++(6,4);
    \draw[line width = 1.5] (0,12) rectangle ++(6,4);
    \draw[line width = 1.5] (0,16) rectangle ++(6,4);
    \draw[line width = 1.5] (0,20) rectangle ++(6,4);
    \draw[line width = 1.5] (0,24) rectangle ++(6,4);
    \draw[line width = 1.5] (0,28) rectangle ++(6,4);

    \node at (3,35) {\LARGE $\vdots$};

    \node at (-3,-2) {\LARGE $\vdots$};

    \draw[fill=col7] (-6,0) rectangle ++(2,2);
    \node[color=white] at (-5,1) {$7$};
    \draw[fill=col5] (-6,2) |- (-2,4) |- (-4,0) |- cycle;
    \node[color=white] at (-3,3) {$5$};
    \draw[fill=col2] (-2,0) rectangle ++(2,4);
    \node[color=white] at (-1,2) {$2$};

    \draw[fill=col6] (-6,4) rectangle ++(2,4);
    \node[color=white] at (-5,6) {$6$};
    \draw[fill=col4] (-4,4) rectangle ++(4,4);
    \node[color=white] at (-2,6) {$4$};

    \draw[fill=col3] (-4,10) |- (0,12) |- (-2,8) |- cycle;
    \node[color=white] at (-1,11) {$3$};
    \draw[fill=col5] (-6,8) |- (-4,12) |- (-2,10) |- cycle;
    \node[color=white] at (-5,9) {$5$};

    \draw[fill=col4] (-6,12) rectangle ++(4,4);
    \node[color=white] at (-4,14) {$4$};
    \draw[fill=col1] (-2,12) rectangle ++(2,4);
    \node[color=white] at (-1,14) {$1$};

    \draw[fill=col3] (-4,16) |- (-2,20) |- (0,18) |- cycle;
    \node[color=white] at (-3,17) {$3$};
    \draw[fill=col2] (-6,16) rectangle ++(2,4);
    \node[color=white] at (-5,18) {$2$};
    \draw[fill=col0] (-2,18) rectangle ++(2,2);
    \node[color=white] at (-1,19) {$0$};

    \draw[fill=col4] (-6,20) rectangle ++(4,4);
    \node[color=white] at (-4,22) {$4$};
    \draw[fill=col1] (-2,20) rectangle ++(2,4);
    \node[color=white] at (-1,22) {$1$};

    \draw[fill=col3] (-4,26) |- (0,28) |- (-2,24) |- cycle;
    \node[color=white] at (-1,27) {$3$};
    \draw[fill=col5] (-6,24) |- (-4,28) |- (-2,26) |- cycle;
    \node[color=white] at (-5,25) {$5$};

    \draw[fill=col6] (-6,28) rectangle ++(2,4);
    \node[color=white] at (-5,30) {$6$};
    \draw[fill=col4] (-4,28) rectangle ++(4,4);
    \node[color=white] at (-2,30) {$4$};

    \draw[line width = 1.5] (-6,0) rectangle ++(6,4);
    \draw[line width = 1.5] (-6,4) rectangle ++(6,4);
    \draw[line width = 1.5] (-6,8) rectangle ++(6,4);
    \draw[line width = 1.5] (-6,12) rectangle ++(6,4);
    \draw[line width = 1.5] (-6,16) rectangle ++(6,4);
    \draw[line width = 1.5] (-6,20) rectangle ++(6,4);
    \draw[line width = 1.5] (-6,24) rectangle ++(6,4);
    \draw[line width = 1.5] (-6,28) rectangle ++(6,4);

    \node at (-3,35) {\LARGE $\vdots$};

    \node at (-9,-2) {\LARGE $\vdots$};

    \draw[fill=col0] (-12,0) rectangle ++(2,2);
    \node[color=white] at (-11,1) {$0$};
    \draw[fill=col3] (-12,2) |- (-8,4) |- (-10,0) |- cycle;
    \node[color=white] at (-9,3) {$3$};
    \draw[fill=col2] (-8,0) rectangle ++(2,4);
    \node[color=white] at (-7,2) {$2$};

    \draw[fill=col1] (-12,4) rectangle ++(2,4);
    \node[color=white] at (-11,6) {$1$};
    \draw[fill=col4] (-10,4) rectangle ++(4,4);
    \node[color=white] at (-8,6) {$4$};

    \draw[fill=col5] (-10,10) |- (-6,12) |- (-8,8) |- cycle;
    \node[color=white] at (-7,11) {$5$};
    \draw[fill=col3] (-12,8) |- (-10,12) |- (-8,10) |- cycle;
    \node[color=white] at (-11,9) {$3$};

    \draw[fill=col4] (-12,12) rectangle ++(4,4);
    \node[color=white] at (-10,14) {$4$};
    \draw[fill=col6] (-8,12) rectangle ++(2,4);
    \node[color=white] at (-7,14) {$6$};

    \draw[fill=col5] (-10,16) |- (-8,20) |- (-6,18) |- cycle;
    \node[color=white] at (-9,17) {$5$};
    \draw[fill=col2] (-12,16) rectangle ++(2,4);
    \node[color=white] at (-11,18) {$2$};
    \draw[fill=col7] (-8,18) rectangle ++(2,2);
    \node[color=white] at (-7,19) {$7$};

    \draw[fill=col4] (-12,20) rectangle ++(4,4);
    \node[color=white] at (-10,22) {$4$};
    \draw[fill=col6] (-8,20) rectangle ++(2,4);
    \node[color=white] at (-7,22) {$6$};

    \draw[fill=col5] (-10,26) |- (-6,28) |- (-8,24) |- cycle;
    \node[color=white] at (-7,27) {$5$};
    \draw[fill=col3] (-12,24) |- (-10,28) |- (-8,26) |- cycle;
    \node[color=white] at (-11,25) {$3$};

    \draw[fill=col1] (-12,28) rectangle ++(2,4);
    \node[color=white] at (-11,30) {$1$};
    \draw[fill=col4] (-10,28) rectangle ++(4,4);
    \node[color=white] at (-8,30) {$4$};

    \draw[line width = 1.5] (-12,0) rectangle ++(6,4);
    \draw[line width = 1.5] (-12,4) rectangle ++(6,4);
    \draw[line width = 1.5] (-12,8) rectangle ++(6,4);
    \draw[line width = 1.5] (-12,12) rectangle ++(6,4);
    \draw[line width = 1.5] (-12,16) rectangle ++(6,4);
    \draw[line width = 1.5] (-12,20) rectangle ++(6,4);
    \draw[line width = 1.5] (-12,24) rectangle ++(6,4);
    \draw[line width = 1.5] (-12,28) rectangle ++(6,4);

    \node at (-9,35) {\LARGE $\vdots$};

    \node at (-15,-2) {\LARGE $\vdots$};

    \draw[fill=col7] (-18,0) rectangle ++(2,2);
    \node[color=white] at (-17,1) {$7$};
    \draw[fill=col5] (-18,2) |- (-14,4) |- (-16,0) |- cycle;
    \node[color=white] at (-15,3) {$5$};
    \draw[fill=col2] (-14,0) rectangle ++(2,4);
    \node[color=white] at (-13,2) {$2$};

    \draw[fill=col6] (-18,4) rectangle ++(2,4);
    \node[color=white] at (-17,6) {$6$};
    \draw[fill=col4] (-16,4) rectangle ++(4,4);
    \node[color=white] at (-14,6) {$4$};

    \draw[fill=col3] (-16,10) |- (-12,12) |- (-14,8) |- cycle;
    \node[color=white] at (-13,11) {$3$};
    \draw[fill=col5] (-18,8) |- (-16,12) |- (-14,10) |- cycle;
    \node[color=white] at (-17,9) {$5$};

    \draw[fill=col4] (-18,12) rectangle ++(4,4);
    \node[color=white] at (-16,14) {$4$};
    \draw[fill=col1] (-14,12) rectangle ++(2,4);
    \node[color=white] at (-13,14) {$1$};

    \draw[fill=col3] (-16,16) |- (-14,20) |- (-12,18) |- cycle;
    \node[color=white] at (-15,17) {$3$};
    \draw[fill=col2] (-18,16) rectangle ++(2,4);
    \node[color=white] at (-17,18) {$2$};
    \draw[fill=col0] (-14,18) rectangle ++(2,2);
    \node[color=white] at (-13,19) {$0$};

    \draw[fill=col4] (-18,20) rectangle ++(4,4);
    \node[color=white] at (-16,22) {$4$};
    \draw[fill=col1] (-14,20) rectangle ++(2,4);
    \node[color=white] at (-13,22) {$1$};

    \draw[fill=col3] (-16,26) |- (-12,28) |- (-14,24) |- cycle;
    \node[color=white] at (-13,27) {$3$};
    \draw[fill=col5] (-18,24) |- (-16,28) |- (-14,26) |- cycle;
    \node[color=white] at (-17,25) {$5$};

    \draw[fill=col6] (-18,28) rectangle ++(2,4);
    \node[color=white] at (-17,30) {$6$};
    \draw[fill=col4] (-16,28) rectangle ++(4,4);
    \node[color=white] at (-14,30) {$4$};

    \draw[line width = 1.5] (-18,0) rectangle ++(6,4);
    \draw[line width = 1.5] (-18,4) rectangle ++(6,4);
    \draw[line width = 1.5] (-18,8) rectangle ++(6,4);
    \draw[line width = 1.5] (-18,12) rectangle ++(6,4);
    \draw[line width = 1.5] (-18,16) rectangle ++(6,4);
    \draw[line width = 1.5] (-18,20) rectangle ++(6,4);
    \draw[line width = 1.5] (-18,24) rectangle ++(6,4);
    \draw[line width = 1.5] (-18,28) rectangle ++(6,4);

    \node at (-15,35) {\LARGE $\vdots$};

    \draw[line width = 1.5] (6,-2) -- (6,34);
    \draw[xshift=-6cm, line width = 1.5] (6,-2) -- (6,34);
    \draw[xshift=-12cm, line width = 1.5] (6,-2) -- (6,34);
    \draw[xshift=-18cm, line width = 1.5] (6,-2) -- (6,34);
    \draw[xshift=-24cm, line width = 1.5] (6,-2) -- (6,34);

    \draw[line width = 1.5] (4,0) -- (-20,0);
    \draw[yshift=4cm, line width = 1.5] (4,0) -- (-20,0);
    \draw[yshift=8cm, line width = 1.5] (4,0) -- (-20,0);
    \draw[yshift=12cm, line width = 1.5] (4,0) -- (-20,0);
    \draw[yshift=16cm, line width = 1.5] (4,0) -- (-20,0);
    \draw[yshift=20cm, line width = 1.5] (4,0) -- (-20,0);
    \draw[yshift=24cm, line width = 1.5] (4,0) -- (-20,0);
    \draw[yshift=28cm, line width = 1.5] (4,0) -- (-20,0);
    \draw[yshift=32cm, line width = 1.5] (4,0) -- (-20,0);

    \node at (-20,2) {\Large $\dots$};
    \node at (-20,6) {\Large $\dots$};
    \node at (-20,10) {\Large $\dots$};
    \node at (-20,14) {\Large $\dots$};
    \node at (-20,18) {\Large $\dots$};
    \node at (-20,22) {\Large $\dots$};
    \node at (-20,26) {\Large $\dots$};
    \node at (-20,30) {\Large $\dots$};
    \end{tikzpicture}
    }
    \end{tabular}}
    &
    \adjustbox{valign=m}{
    \subfloat[Type $\Eaff{8}$]{
    \input{E8_Young_wall_pattern}
    }
    }
    \end{tabular}
    \caption{\hspace{.5em}Young wall patterns for types $\Eaff{6}$, $\Eaff{7}$ and $\Eaff{8}$}\label{Young wall patterns}
\end{figure}

\begin{figure}[H]
    \centering
    \begin{tabular}{cc}
    \adjustbox{valign=m}{
    \begin{tabular}{@{}c@{}}
    \subfloat[Type $\Eaff{6}$]{
    \begin{tikzpicture}[scale=0.15, line width = 0.5]

    \node at (2,-6) {\LARGE $\vdots$};

    \draw[fill=col2] (0,-4) rectangle ++(2,4);
    \node[color=white] at (1,-2) {$2$};
    \draw[fill=col3] (2,-4) rectangle ++(2,4);
    \node[color=white] at (3,-2) {$3$};

    \node[color=black] at (1,1) {$0$};
    \draw[fill=col4] (0,2) |- (4,4) |- (2,0) |- cycle;
    \node[color=white] at (3,3) {$4$};

    \node[color=black] at (1,6) {$2$};
    \draw[fill=col5] (2,4) rectangle ++(2,4);
    \node[color=white] at (3,6) {$5$};

    \draw[fill=col6] (2,10) rectangle ++(2,2);
    \node[color=white] at (3,11) {$6$};
    \node[color=black] at (1,9) {$4$};

    \draw[line width = 1.5] (0,-4) rectangle ++(4,4);
    \draw[line width = 1.5] (0,0) rectangle ++(4,4);
    \draw[line width = 1.5] (0,4) rectangle ++(4,4);
    \draw[line width = 1.5] (0,8) rectangle ++(4,4);

    \draw[line width = 1.5] (4,-6) -- (4,12);
    \draw[line width = 1.5] (0,-6) -- (0,12);

    \begin{scope}[shift={(-4,0)}]
    \node at (2,-6) {\LARGE $\vdots$};

    \draw[fill=col5] (0,-4) rectangle ++(2,4);
    \node[color=white] at (1,-2) {$5$};
    \draw[fill=col2] (2,-4) rectangle ++(2,4);
    \node[color=white] at (3,-2) {$2$};

    \node[color=black] at (1,1) {$6$};
    \draw[fill=col4] (0,2) |- (4,4) |- (2,0) |- cycle;
    \node[color=white] at (3,3) {$4$};

    \node[color=black] at (1,6) {$5$};
    \draw[fill=col3] (2,4) rectangle ++(2,4);
    \node[color=white] at (3,6) {$3$};

    \draw[fill=col1] (2,10) rectangle ++(2,2);
    \node[color=white] at (3,11) {$1$};
    \node[color=black] at (1,9) {$4$};

    \draw[line width = 1.5] (0,-4) rectangle ++(4,4);
    \draw[line width = 1.5] (0,0) rectangle ++(4,4);
    \draw[line width = 1.5] (0,4) rectangle ++(4,4);
    \draw[line width = 1.5] (0,8) rectangle ++(4,4);

    \draw[line width = 1.5] (4,-6) -- (4,12);
    \draw[line width = 1.5] (0,-6) -- (0,12);
    \end{scope}

    \begin{scope}[shift={(-8,0)}]
    \node at (2,-6) {\LARGE $\vdots$};

    \draw[fill=col3] (0,-4) rectangle ++(2,4);
    \node[color=white] at (1,-2) {$3$};
    \draw[fill=col5] (2,-4) rectangle ++(2,4);
    \node[color=white] at (3,-2) {$5$};

    \node[color=black] at (1,1) {$1$};
    \draw[fill=col4] (0,2) |- (4,4) |- (2,0) |- cycle;
    \node[color=white] at (3,3) {$4$};

    \node[color=black] at (1,6) {$3$};
    \draw[fill=col2] (2,4) rectangle ++(2,4);
    \node[color=white] at (3,6) {$2$};

    \draw[fill=col0] (2,10) rectangle ++(2,2);
    \node[color=white] at (3,11) {$0$};
    \node[color=black] at (1,9) {$4$};

    \draw[line width = 1.5] (0,-4) rectangle ++(4,4);
    \draw[line width = 1.5] (0,0) rectangle ++(4,4);
    \draw[line width = 1.5] (0,4) rectangle ++(4,4);
    \draw[line width = 1.5] (0,8) rectangle ++(4,4);

    \draw[line width = 1.5] (4,-6) -- (4,12);
    \draw[line width = 1.5] (0,-6) -- (0,12);
    \end{scope}

    \begin{scope}[shift={(-12,0)}]
    \node at (2,-6) {\LARGE $\vdots$};

    \draw[fill=col2] (0,-4) rectangle ++(2,4);
    \node[color=white] at (1,-2) {$2$};
    \draw[fill=col3] (2,-4) rectangle ++(2,4);
    \node[color=white] at (3,-2) {$3$};

    \node[color=black] at (1,1) {$0$};
    \draw[fill=col4] (0,2) |- (4,4) |- (2,0) |- cycle;
    \node[color=white] at (3,3) {$4$};

    \node[color=black] at (1,6) {$2$};
    \draw[fill=col5] (2,4) rectangle ++(2,4);
    \node[color=white] at (3,6) {$5$};

    \draw[fill=col6] (2,10) rectangle ++(2,2);
    \node[color=white] at (3,11) {$6$};
    \node[color=black] at (1,9) {$4$};

    \draw[line width = 1.5] (0,-4) rectangle ++(4,4);
    \draw[line width = 1.5] (0,0) rectangle ++(4,4);
    \draw[line width = 1.5] (0,4) rectangle ++(4,4);
    \draw[line width = 1.5] (0,8) rectangle ++(4,4);

    \draw[line width = 1.5] (4,-6) -- (4,12);
    \draw[line width = 1.5] (0,-6) -- (0,12);
    \end{scope}

    \begin{scope}[shift={(-16,0)}]
    \node at (2,-6) {\LARGE $\vdots$};

    \draw[fill=col5] (0,-4) rectangle ++(2,4);
    \node[color=white] at (1,-2) {$5$};
    \draw[fill=col2] (2,-4) rectangle ++(2,4);
    \node[color=white] at (3,-2) {$2$};

    \node[color=black] at (1,1) {$6$};
    \draw[fill=col4] (0,2) |- (4,4) |- (2,0) |- cycle;
    \node[color=white] at (3,3) {$4$};

    \node[color=black] at (1,6) {$5$};
    \draw[fill=col3] (2,4) rectangle ++(2,4);
    \node[color=white] at (3,6) {$3$};

    \draw[fill=col1] (2,10) rectangle ++(2,2);
    \node[color=white] at (3,11) {$1$};
    \node[color=black] at (1,9) {$4$};

    \draw[line width = 1.5] (0,-4) rectangle ++(4,4);
    \draw[line width = 1.5] (0,0) rectangle ++(4,4);
    \draw[line width = 1.5] (0,4) rectangle ++(4,4);
    \draw[line width = 1.5] (0,8) rectangle ++(4,4);

    \draw[line width = 1.5] (4,-6) -- (4,12);
    \draw[line width = 1.5] (0,-6) -- (0,12);
    \end{scope}

    \begin{scope}[shift={(-20,0)}]
    \node at (2,-6) {\LARGE $\vdots$};

    \draw[fill=col3] (0,-4) rectangle ++(2,4);
    \node[color=white] at (1,-2) {$3$};
    \draw[fill=col5] (2,-4) rectangle ++(2,4);
    \node[color=white] at (3,-2) {$5$};

    \node[color=black] at (1,1) {$1$};
    \draw[fill=col4] (0,2) |- (4,4) |- (2,0) |- cycle;
    \node[color=white] at (3,3) {$4$};

    \node[color=black] at (1,6) {$3$};
    \draw[fill=col2] (2,4) rectangle ++(2,4);
    \node[color=white] at (3,6) {$2$};

    \draw[fill=col0] (2,10) rectangle ++(2,2);
    \node[color=white] at (3,11) {$0$};
    \node[color=black] at (1,9) {$4$};

    \draw[line width = 1.5] (0,-4) rectangle ++(4,4);
    \draw[line width = 1.5] (0,0) rectangle ++(4,4);
    \draw[line width = 1.5] (0,4) rectangle ++(4,4);
    \draw[line width = 1.5] (0,8) rectangle ++(4,4);

    \draw[line width = 1.5] (4,-6) -- (4,12);
    \draw[line width = 1.5] (0,-6) -- (0,12);
    \end{scope}
    
    \draw[yshift=-4cm, line width = 1.5] (4,0) -- (-22,0);
    \draw[line width = 1.5] (4,0) -- (-22,0);
    \draw[yshift=4cm, line width = 1.5] (4,0) -- (-22,0);
    \draw[yshift=8cm, line width = 1.5] (4,0) -- (-22,0);
    \draw[yshift=12cm, line width = 1.5] (4,0) -- (-22,0);

    \node at (-22,-2) {\Large $\dots$};
    \node at (-22,2) {\Large $\dots$};
    \node at (-22,6) {\Large $\dots$};
    \node at (-22,10) {\Large $\dots$};
    \end{tikzpicture}
    } \\
    \subfloat[Type $\Eaff{7}$]{
    \begin{tikzpicture}[scale=0.15, line width = 0.5]

    \node at (3,-6) {\LARGE $\vdots$};

    \draw[fill=col1] (0,-4) rectangle ++(2,4);
    \node[color=white] at (1,-2) {$1$};
    \draw[fill=col4] (2,-4) rectangle ++(4,4);
    \node[color=white] at (4,-2) {$4$};

    \node[color=black] at (1,1) {$0$};
    \draw[fill=col3] (0,2) |- (4,4) |- (2,0) |- cycle;
    \node[color=white] at (3,3) {$3$};
    \draw[fill=col2] (4,0) rectangle ++(2,4);
    \node[color=white] at (5,2) {$2$};

    \node[color=black] at (1,6) {$1$};
    \draw[fill=col4] (2,4) rectangle ++(4,4);
    \node[color=white] at (4,6) {$4$};

    \draw[fill=col5] (2,10) |- (6,12) |- (4,8) |- cycle;
    \node[color=white] at (5,11) {$5$};
    \node[color=black] at (1,9) {$3$};

    \node[color=black] at (2,14) {$4$};
    \draw[fill=col6] (4,12) rectangle ++(2,4);
    \node[color=white] at (5,14) {$6$};

    \draw[fill=none] (2,16) |- (4,20) |- (6,18) |- cycle;
    \draw[fill=col7] (4,18) rectangle ++(2,2);
    \node[color=black] at (3,17) {$5$};
    \node[color=black] at (1,18) {$2$};
    \node[color=white] at (5,19) {$7$};

    \draw[line width = 1.5] (0,-4) rectangle ++(6,4);
    \draw[line width = 1.5] (0,0) rectangle ++(6,4);
    \draw[line width = 1.5] (0,4) rectangle ++(6,4);
    \draw[line width = 1.5] (0,8) rectangle ++(6,4);
    \draw[line width = 1.5] (0,12) rectangle ++(6,4);
    \draw[line width = 1.5] (0,16) rectangle ++(6,4);

    \draw[line width = 1.5] (6,-6) -- (6,20);
    \draw[line width = 1.5] (0,-6) -- (0,20);

    \begin{scope}[shift={(-6,0)}]
    \node at (3,-6) {\LARGE $\vdots$};

    \draw[fill=col6] (0,-4) rectangle ++(2,4);
    \node[color=white] at (1,-2) {$6$};
    \draw[fill=col4] (2,-4) rectangle ++(4,4);
    \node[color=white] at (4,-2) {$4$};

    \node[color=black] at (1,1) {$7$};
    \draw[fill=col5] (0,2) |- (4,4) |- (2,0) |- cycle;
    \node[color=white] at (3,3) {$5$};
    \draw[fill=col2] (4,0) rectangle ++(2,4);
    \node[color=white] at (5,2) {$2$};

    \node[color=black] at (1,6) {$6$};
    \draw[fill=col4] (2,4) rectangle ++(4,4);
    \node[color=white] at (4,6) {$4$};

    \draw[fill=col3] (2,10) |- (6,12) |- (4,8) |- cycle;
    \node[color=white] at (5,11) {$3$};
    \node[color=black] at (1,9) {$5$};

    \node[color=black] at (2,14) {$4$};
    \draw[fill=col1] (4,12) rectangle ++(2,4);
    \node[color=white] at (5,14) {$1$};

    \draw[fill=none] (2,16) |- (4,20) |- (6,18) |- cycle;
    \draw[fill=col0] (4,18) rectangle ++(2,2);
    \node[color=black] at (3,17) {$3$};
    \node[color=black] at (1,18) {$2$};
    \node[color=white] at (5,19) {$0$};

    \draw[line width = 1.5] (0,-4) rectangle ++(6,4);
    \draw[line width = 1.5] (0,0) rectangle ++(6,4);
    \draw[line width = 1.5] (0,4) rectangle ++(6,4);
    \draw[line width = 1.5] (0,8) rectangle ++(6,4);
    \draw[line width = 1.5] (0,12) rectangle ++(6,4);
    \draw[line width = 1.5] (0,16) rectangle ++(6,4);

    \draw[line width = 1.5] (6,-6) -- (6,20);
    \draw[line width = 1.5] (0,-6) -- (0,20);
    \end{scope}

    \begin{scope}[shift={(-12,0)}]
    \node at (3,-6) {\LARGE $\vdots$};

    \draw[fill=col1] (0,-4) rectangle ++(2,4);
    \node[color=white] at (1,-2) {$1$};
    \draw[fill=col4] (2,-4) rectangle ++(4,4);
    \node[color=white] at (4,-2) {$4$};

    \node[color=black] at (1,1) {$0$};
    \draw[fill=col3] (0,2) |- (4,4) |- (2,0) |- cycle;
    \node[color=white] at (3,3) {$3$};
    \draw[fill=col2] (4,0) rectangle ++(2,4);
    \node[color=white] at (5,2) {$2$};

    \node[color=black] at (1,6) {$1$};
    \draw[fill=col4] (2,4) rectangle ++(4,4);
    \node[color=white] at (4,6) {$4$};

    \draw[fill=col5] (2,10) |- (6,12) |- (4,8) |- cycle;
    \node[color=white] at (5,11) {$5$};
    \node[color=black] at (1,9) {$3$};

    \node[color=black] at (2,14) {$4$};
    \draw[fill=col6] (4,12) rectangle ++(2,4);
    \node[color=white] at (5,14) {$6$};

    \draw[fill=none] (2,16) |- (4,20) |- (6,18) |- cycle;
    \draw[fill=col7] (4,18) rectangle ++(2,2);
    \node[color=black] at (3,17) {$5$};
    \node[color=black] at (1,18) {$2$};
    \node[color=white] at (5,19) {$7$};

    \draw[line width = 1.5] (0,-4) rectangle ++(6,4);
    \draw[line width = 1.5] (0,0) rectangle ++(6,4);
    \draw[line width = 1.5] (0,4) rectangle ++(6,4);
    \draw[line width = 1.5] (0,8) rectangle ++(6,4);
    \draw[line width = 1.5] (0,12) rectangle ++(6,4);
    \draw[line width = 1.5] (0,16) rectangle ++(6,4);

    \draw[line width = 1.5] (6,-6) -- (6,20);
    \draw[line width = 1.5] (0,-6) -- (0,20);
    \end{scope}

    \begin{scope}[shift={(-18,0)}]
    \node at (3,-6) {\LARGE $\vdots$};

    \draw[fill=col6] (0,-4) rectangle ++(2,4);
    \node[color=white] at (1,-2) {$6$};
    \draw[fill=col4] (2,-4) rectangle ++(4,4);
    \node[color=white] at (4,-2) {$4$};

    \node[color=black] at (1,1) {$7$};
    \draw[fill=col5] (0,2) |- (4,4) |- (2,0) |- cycle;
    \node[color=white] at (3,3) {$5$};
    \draw[fill=col2] (4,0) rectangle ++(2,4);
    \node[color=white] at (5,2) {$2$};

    \node[color=black] at (1,6) {$6$};
    \draw[fill=col4] (2,4) rectangle ++(4,4);
    \node[color=white] at (4,6) {$4$};

    \draw[fill=col3] (2,10) |- (6,12) |- (4,8) |- cycle;
    \node[color=white] at (5,11) {$3$};
    \node[color=black] at (1,9) {$5$};

    \node[color=black] at (2,14) {$4$};
    \draw[fill=col1] (4,12) rectangle ++(2,4);
    \node[color=white] at (5,14) {$1$};

    \draw[fill=none] (2,16) |- (4,20) |- (6,18) |- cycle;
    \draw[fill=col0] (4,18) rectangle ++(2,2);
    \node[color=black] at (3,17) {$3$};
    \node[color=black] at (1,18) {$2$};
    \node[color=white] at (5,19) {$0$};

    \draw[line width = 1.5] (0,-4) rectangle ++(6,4);
    \draw[line width = 1.5] (0,0) rectangle ++(6,4);
    \draw[line width = 1.5] (0,4) rectangle ++(6,4);
    \draw[line width = 1.5] (0,8) rectangle ++(6,4);
    \draw[line width = 1.5] (0,12) rectangle ++(6,4);
    \draw[line width = 1.5] (0,16) rectangle ++(6,4);

    \draw[line width = 1.5] (6,-6) -- (6,20);
    \draw[line width = 1.5] (0,-6) -- (0,20);
    \end{scope}

    \draw[yshift=-4cm, line width = 1.5] (4,0) -- (-20,0);
    \draw[line width = 1.5] (4,0) -- (-20,0);
    \draw[yshift=4cm, line width = 1.5] (4,0) -- (-20,0);
    \draw[yshift=8cm, line width = 1.5] (4,0) -- (-20,0);
    \draw[yshift=12cm, line width = 1.5] (4,0) -- (-20,0);
    \draw[yshift=16cm, line width = 1.5] (4,0) -- (-20,0);
    \draw[yshift=20cm, line width = 1.5] (4,0) -- (-20,0);

    \node at (-20,-2) {\Large $\dots$};
    \node at (-20,2) {\Large $\dots$};
    \node at (-20,6) {\Large $\dots$};
    \node at (-20,10) {\Large $\dots$};
    \node at (-20,14) {\Large $\dots$};
    \node at (-20,18) {\Large $\dots$};
    \end{tikzpicture}
    }
    \end{tabular}}
    &
    \adjustbox{valign=m}{
    \subfloat[Type $\Eaff{8}$]{
    \input{E8_ground_state_wall}
    }
    }
    \end{tabular}
    \caption{\hspace{.5em}Ground state walls for types $\Eaff{6}$, $\Eaff{7}$ and $\Eaff{8}$}\label{Ground state walls}
\end{figure}

\begin{defn} \label{Young wall definition}
    In each type, a Young wall is a collection of blocks stacked inside the Young wall pattern such that
    \begin{itemize}
        \item it differs from the ground state wall in finitely many blocks,
        \item each column of the wall is a Young column.
    \end{itemize}
\end{defn}

Many authors assume two further conditions on their Young walls, the first of which we shall call the \emph{right block property}:
\begin{itemize}
    \item if a Young wall contains a block, then it must contain the block occupying the same position in the column to the right,
    $\hfill \refstepcounter{equation}(\theequation)\label{right block property}$
    \item a Young wall must be built on top of the ground state wall.
    $\hfill \refstepcounter{equation}(\theequation)\label{built on ground state wall property}$
\end{itemize}
We have chosen to omit these conditions from our definition since it does not seem inherently obvious that they should hold for the walls in our models for $B(\Lambda_{i})$.
However, with Proposition \ref{highest weight right block property proposition} and Corollary \ref{highest weight built on ground state wall corollary} we confirm that the reduced Young walls do indeed satisfy (\ref{right block property}) and (\ref{built on ground state wall property}) respectively.
\\

In this section we write each Young wall $Y$ as a sequence $(y_{r})_{r=0}^{\infty}$ of Young columns, considered only up to equivalence class as elements of $C$.
We also define $|y_{r}|_{0}$ (resp. $|y_{r}|$) to be the difference in the number of $0$-blocks (resp. blocks) between $Y$ and the ground state wall in column $r$.

\begin{defn} \label{reduced definition}
A pair of adjacent columns $(y_{r+1},y_{r})$ in a Young wall $Y = (y_{r})_{r=0}^{\infty}$ is reduced if
\begin{equation} \label{reduced definition equation}
    H(y_{r+1}\otimes y_{r}) - |y_{r+1}|_{0} + |y_{r}|_{0}
        = H(b_{r+1}\otimes b_{r}).
\end{equation}
The Young wall $Y$ is reduced if $(y_{r+1},y_{r})$ is reduced for every $r\in\Nbb$.
\end{defn}

In particular, if a pair of adjacent columns $(y_{r+1},y_{r})$ is reduced then there is exactly one option for the integer $|y_{r}|_{0} - |y_{r+1}|_{0}$.
Denote the set of reduced Young walls by $\Ycal(\Lambda_{0})$.

\begin{prop} \label{highest weight right block property proposition}
    If a Young wall is reduced then it must satisfy the right block property (\ref{right block property}).
\end{prop}
\begin{proof}
    This follows from Propositions \ref{E6 and E7 Fock right block proposition} and \ref{E8 Fock right block proposition} since $\Ycal(\Lambda_{0})$ can be viewed as the set of $(z^{n_{r}} y_{r})_{r=0}^{\infty} \in \Zcal(\Lambda_{0})$ with every $\Haff(z^{n_{r+1}} y_{r+1}\otimes z^{n_{r}} y_{r}) = 1$.
    In particular, a Young wall $(y_{r})_{r=0}^{\infty}$ in $\Ycal(\Lambda_{0})$ is precisely the same as the Young wall $(z^{m_{r}-|y_{r}|_{0}} y_{r})_{r=0}^{\infty}$ in $\Zcal(\Lambda_{0})$, where the $m_{r}$ are given by the ground state sequence $\sbold_{\Lambda_{0}} = (z^{m_{r}}b_{r})_{r=0}^{\infty}$ for the Fock space.
\end{proof}

\begin{cor} \label{highest weight built on ground state wall corollary}
    Every reduced Young wall is built on top of the ground state wall.
\end{cor}
\begin{proof}
    Since a Young wall differs from the ground state wall in finitely many blocks and thus matches it in all columns sufficiently far to the left, this condition follows from the right block property.
\end{proof}

\begin{prop} \label{difference in added blocks for reduced adjacent columns}
    If a pair of adjacent columns $(y_{r+1},y_{r})$ in a Young wall is reduced then $|y_{r}| - |y_{r+1}|$ is a fixed non-negative integer.
\end{prop}
\begin{proof}
    It is clear from Definition \ref{reduced definition} that if we specify the columns $y_{r+1}$ and $y_{r}$ (up to equivalence), then there is precisely one value of $|y_{r}|_{0} - |y_{r+1}|_{0}$ which makes $(y_{r+1},y_{r})$ reduced.
    By looking at the Young column patterns from Figures \ref{E6 and E7 Young column patterns} and \ref{E8 Young column pattern} we see that this in turn fixes $|y_{r}| - |y_{r+1}|$.
    The right block property of Proposition \ref{highest weight right block property proposition} implies that this number must be non-negative.
\end{proof}

It therefore follows that up to vertical shift, there are precisely $|B|^{2}$ pairs of reduced adjacent columns $(y_{r+1},y_{r})$ for each $r\in\Nbb$, one for each choice of equivalence class for both $y_{r+1}$ and $y_{r}$.
However, it is important to note that in types $\Eaff{6}$ and $\Eaff{7}$ the options change with $r$ since from Table \ref{highest weight ground state sequences table} we see that the sequence $(H(b_{r+1}\otimes b_{r}))_{r=0}^{\infty}$ is not constant.
Figure \ref{Young wall examples} demonstrates these phenomena with some examples in type $\Eaff{6}$.
The first wall is reduced, but the second is not since it does not have the required value of $|y_{r}|_{0} - |y_{r+1}|_{0}$ (or equivalently $|y_{r}| - |y_{r+1}|$).
The third wall shows how the reduced adjacent pairs $(y_{r+1},y_{r})$ depend on $r$.

\begin{figure}[H]
    \centering
    \subfloat[reduced]{
    \begin{tikzpicture}[scale=0.15, line width = 0.5]

    \node at (2,-2) {\LARGE $\vdots$};

    \draw[fill=col0] (0,0) rectangle ++(2,2);
    \node[color=white] at (1,1) {$0$};
    \draw[fill=col4] (0,2) |- (4,4) |- (2,0) |- cycle;
    \node[color=white] at (3,3) {$4$};

    \draw[fill=col2] (0,4) rectangle ++(2,4);
    \node[color=white] at (1,6) {$2$};
    \draw[fill=col5] (2,4) rectangle ++(2,4);
    \node[color=white] at (3,6) {$5$};

    \draw[fill=col6] (2,10) rectangle ++(2,2);
    \node[color=white] at (3,11) {$6$};
    \draw[fill=col4] (0,8) |- (2,12) |- (4,10) |- cycle;
    \node[color=white] at (1,9) {$4$};

    \draw[fill=col3] (0,12) rectangle ++(2,4);
    \node[color=white] at (1,14) {$3$};
    \draw[fill=col5] (2,12) rectangle ++(2,4);
    \node[color=white] at (3,14) {$5$};

    \draw[fill=col1] (0,16) rectangle ++(2,2);
    \node[color=white] at (1,17) {$1$};
    \draw[fill=col4] (0,18) |- (4,20) |- (2,16) |- cycle;
    \node[color=white] at (3,19) {$4$};

    \draw[fill=col3] (0,20) rectangle ++(2,4);
    \node[color=white] at (1,22) {$3$};
    \draw[fill=col2] (2,20) rectangle ++(2,4);
    \node[color=white] at (3,22) {$2$};

    \draw[fill=col0] (2,26) rectangle ++(2,2);
    \node[color=white] at (3,27) {$0$};
    \draw[fill=col4] (0,24) |- (2,28) |- (4,26) |- cycle;
    \node[color=white] at (1,25) {$4$};

    \draw[fill=col5] (0,28) rectangle ++(2,4);
    \node[color=white] at (1,30) {$5$};
    \draw[fill=col2] (2,28) rectangle ++(2,4);
    \node[color=white] at (3,30) {$2$};

    \draw[fill=col6] (0,32) rectangle ++(2,2);
    \node[color=white] at (1,33) {$6$};
    \draw[fill=col4] (0,34) |- (4,36) |- (2,32) |- cycle;
    \node[color=white] at (3,35) {$4$};

    \draw[line width = 1.5] (0,0) rectangle ++(4,4);
    \draw[line width = 1.5] (0,4) rectangle ++(4,4);
    \draw[line width = 1.5] (0,8) rectangle ++(4,4);
    \draw[line width = 1.5] (0,12) rectangle ++(4,4);
    \draw[line width = 1.5] (0,16) rectangle ++(4,4);
    \draw[line width = 1.5] (0,20) rectangle ++(4,4);
    \draw[line width = 1.5] (0,24) rectangle ++(4,4);
    \draw[line width = 1.5] (0,28) rectangle ++(4,4);
    \draw[line width = 1.5] (0,32) rectangle ++(4,4);

    \node at (-2,-2) {\LARGE $\vdots$};

    \draw[fill=col6] (-4,0) rectangle ++(2,2);
    \node[color=white] at (-3,1) {$6$};
    \draw[fill=col4] (-4,2) |- (0,4) |- (-2,0) |- cycle;
    \node[color=white] at (-1,3) {$4$};

    \draw[fill=col5] (-4,4) rectangle ++(2,4);
    \node[color=white] at (-3,6) {$5$};
    \draw[fill=col3] (-2,4) rectangle ++(2,4);
    \node[color=white] at (-1,6) {$3$};

    \draw[fill=col1] (-2,10) rectangle ++(2,2);
    \node[color=white] at (-1,11) {$1$};
    \draw[fill=none] (-4,8) |- (-2,12) |- (0,10) |- cycle;
    \node[color=black] at (-3,9) {$4$};

    \draw[line width = 1.5] (-4,0) rectangle ++(4,4);
    \draw[line width = 1.5] (-4,4) rectangle ++(4,4);
    \draw[line width = 1.5] (-4,8) rectangle ++(4,4);

    \draw[line width = 1.5] (4,-2) -- (4,0);
    \draw[xshift=-4cm, line width = 1.5] (4,-2) -- (4,0);
    \draw[xshift=-8cm, line width = 1.5] (4,-2) -- (4,0);

    \draw[line width = 1.5] (4,0) -- (-6,0);
    \draw[yshift=4cm, line width = 1.5] (4,0) -- (-6,0);
    \draw[yshift=8cm, line width = 1.5] (4,0) -- (-6,0);
    \draw[yshift=12cm, line width = 1.5] (4,0) -- (-6,0);

    \node at (-6,2) {\Large $\dots$};
    \node at (-6,6) {\Large $\dots$};
    \node at (-6,10) {\Large $\dots$};
\end{tikzpicture}
    }
    \qquad
    \subfloat[non-reduced]{ \label{Young wall example 2}
    \begin{tikzpicture}[scale=0.15, line width = 0.5]

    \node at (2,-2) {\LARGE $\vdots$};

    \draw[fill=col0] (0,0) rectangle ++(2,2);
    \node[color=white] at (1,1) {$0$};
    \draw[fill=col4] (0,2) |- (4,4) |- (2,0) |- cycle;
    \node[color=white] at (3,3) {$4$};

    \draw[fill=col2] (0,4) rectangle ++(2,4);
    \node[color=white] at (1,6) {$2$};
    \draw[fill=col5] (2,4) rectangle ++(2,4);
    \node[color=white] at (3,6) {$5$};

    \draw[fill=col6] (2,10) rectangle ++(2,2);
    \node[color=white] at (3,11) {$6$};
    \draw[fill=col4] (0,8) |- (2,12) |- (4,10) |- cycle;
    \node[color=white] at (1,9) {$4$};

    \draw[line width = 1.5] (0,0) rectangle ++(4,4);
    \draw[line width = 1.5] (0,4) rectangle ++(4,4);
    \draw[line width = 1.5] (0,8) rectangle ++(4,4);

    \node at (-2,-2) {\LARGE $\vdots$};

    \draw[fill=col6] (-4,0) rectangle ++(2,2);
    \node[color=white] at (-3,1) {$6$};
    \draw[fill=col4] (-4,2) |- (0,4) |- (-2,0) |- cycle;
    \node[color=white] at (-1,3) {$4$};

    \draw[fill=col5] (-4,4) rectangle ++(2,4);
    \node[color=white] at (-3,6) {$5$};
    \draw[fill=col3] (-2,4) rectangle ++(2,4);
    \node[color=white] at (-1,6) {$3$};

    \draw[fill=col1] (-2,10) rectangle ++(2,2);
    \node[color=white] at (-1,11) {$1$};
    \draw[fill=none] (-4,8) |- (-2,12) |- (0,10) |- cycle;
    \node[color=black] at (-3,9) {$4$};

    \draw[line width = 1.5] (-4,0) rectangle ++(4,4);
    \draw[line width = 1.5] (-4,4) rectangle ++(4,4);
    \draw[line width = 1.5] (-4,8) rectangle ++(4,4);

    \draw[line width = 1.5] (4,-2) -- (4,0);
    \draw[xshift=-4cm, line width = 1.5] (4,-2) -- (4,0);
    \draw[xshift=-8cm, line width = 1.5] (4,-2) -- (4,0);

    \draw[line width = 1.5] (4,0) -- (-6,0);
    \draw[yshift=4cm, line width = 1.5] (4,0) -- (-6,0);
    \draw[yshift=8cm, line width = 1.5] (4,0) -- (-6,0);
    \draw[yshift=12cm, line width = 1.5] (4,0) -- (-6,0);

    \node at (-6,2) {\Large $\dots$};
    \node at (-6,6) {\Large $\dots$};
    \node at (-6,10) {\Large $\dots$};
\end{tikzpicture}
    }
    \qquad
    \subfloat[non-reduced]{
    \input{E6_Young_wall_example_3}
    }
    \caption{\hspace{.5em}Examples of Young walls in type $\Eaff{6}$}\label{Young wall examples}
\end{figure}

Next we define the structure of an affine crystal on the set $\Ycal(\Lambda_{0})$ of reduced Young walls.
Recall that $\varphi_{i}(y)$ (resp. $\varepsilon_{i}(y)$) is by definition the maximum number of $i$-blocks which can be added to (resp. removed from) a Young column $y$ sequentially, while still remaining a Young column.

\begin{defn}
    The \textit{$i$-signature} of $y$ is the sequence $\sign_{i}(y) = \underbrace{-\dots-}_{\varepsilon_{i}(y)}\underbrace{+\dots+}_{\varphi_{i}(y)}$.
\end{defn}

For a Young wall $Y = (y_{r})_{r=0}^{\infty}$ we define the pre-$i$-signature to be the (possibly infinite) sequence
\begin{equation*}
    \presign_{i}(Y) = \dots\sign_{i}(y_{2})\sign_{i}(y_{1})\sign_{i}(y_{0})
\end{equation*}
of $+$'s and $-$'s.
Cancelling every $+-$ pair leaves a finite number of $-$'s followed by a finite number of $+$'s, which we call the $i$-signature $\sign_{i}(Y)$ of $Y$.
\\

Let $\et_{i}(Y)$ be the Young wall obtained from $Y$ by applying $\et_{i}$ to the column containing the rightmost $-$ in $\sign_{i}(Y)$ if it exists, and $0$ otherwise.
Conversely, $\ft_{i}(Y)$ is defined by applying $\ft_{i}$ to the column containing the leftmost $+$ in $\sign_{i}(Y)$ if it exists, and is $0$ otherwise.

\begin{prop} \label{reduced walls closed under e and f proposition}
    For any $Y\in\Ycal(\Lambda_{0})$ we have $\et_{i}(Y),\ft_{i}(Y)\in\Ycal(\Lambda_{0})\sqcup\lbrace 0\rbrace$.
\end{prop}
\begin{proof}
    Let $Y = (y_{r})_{r=0}^{\infty}$ be a reduced Young wall.
    If $\ft_{i}(Y)=0$ then we are done, so instead suppose that $\ft_{i}(Y) = (\dots,y_{k+1},z_{k},y_{k-1},\dots,y_{0})$ where $z_{k}$ is obtained by adding an $i$-block to $y_{k}$.
    Both $\presign_{i}(Y)$ and $\sign_{i}(Y)$ have at least one $+$ in column $k$, and no $+$ in column $k+1$.
    In $\presign_{i}(Y)$ there are strictly fewer $-$'s in column $k-1$ than $+$'s in column $k$, and hence there is no $-$ in column $k-1$ of $\sign_{i}(Y)$.
    From the formulae (\ref{tensor product of crystals}) for the tensor product of crystals, we therefore have
    \begin{align*}
        \ft_{i}(z^{-|y_{k+1}|_{0}}y_{k+1}\otimes z^{-|y_{k}|_{0}}y_{k}) &= z^{-|y_{k+1}|_{0}}y_{k+1}\otimes z^{-|z_{k}|_{0}}z_{k}, \\
        \ft_{i}(z^{-|y_{k}|_{0}}y_{k}\otimes z^{-|y_{k-1}|_{0}}y_{k-1}) &= z^{-|z_{k}|_{0}}z_{k}\otimes z^{-|y_{k-1}|_{0}}y_{k-1}.
    \end{align*}
    By Lemma \ref{Haff constant on components} it follows that
    \begin{align*}
        \Haff(z^{-|y_{k+1}|_{0}}y_{k+1}\otimes z^{-|z_{k}|_{0}}z_{k}) &= \Haff(z^{-|y_{k+1}|_{0}}y_{k+1}\otimes z^{-|y_{k}|_{0}}y_{k}), \\
        \Haff(z^{-|y_{k}|_{0}}y_{k}\otimes z^{-|y_{k-1}|_{0}}y_{k-1}) &= \Haff(z^{-|z_{k}|_{0}}z_{k}\otimes z^{-|y_{k-1}|_{0}}y_{k-1}).
    \end{align*}
    From equation (\ref{Haff definition}) which defines $\Haff$ we can deduce that since $(y_{k+1},y_{k})$ and $(y_{k},y_{k-1})$ are reduced, so are $(y_{k+1},z_{k})$ and $(z_{k},y_{k-1})$.
    All other pairs of adjacent columns in $\ft_{i}(Y)$ are the same as those in $Y$ and thus satisfy (\ref{reduced definition equation}), hence $\ft_{i}(Y)$ is a reduced Young wall.
    The case of $\et_{i}(Y)$ is similar.
\end{proof}

If we further define
\begin{itemize}
    \item $\varepsilon_{i}(Y) =$ number of $-$'s in $\sign_{i}(Y)$
    \item $\varphi_{i}(Y) =$ number of $+$'s in $\sign_{i}(Y)$
    \item $\wt(Y) = \Lambda_{0} - \sum_{i\in I} k_{i} \alpha_{i}$
\end{itemize}
where $k_{i}$ is the number of $i$-blocks in $Y$ added to the ground state wall, then a routine check proves the following.

\begin{thm}
    The maps $\et_{i},\ft_{i} : \Ycal(\Lambda_{0}) \rightarrow \Ycal(\Lambda_{0})\sqcup\lbrace 0\rbrace$, $\varepsilon_{i},\varphi_{i} : \Ycal(\Lambda_{0}) \rightarrow \Zbb\cup\lbrace -\infty\rbrace$ and $\wt : \Ycal(\Lambda_{0}) \rightarrow P$ above endow $\Ycal(\Lambda_{0})$ with the structure of an affine crystal.
\end{thm}

Note that it follows from Corollary \ref{highest weight built on ground state wall corollary} and the definitions of $\et_{i}$ and $\ft_{i}$ that the ground state wall is the highest weight element of $\Ycal(\Lambda_{0})$.
As an example, we have included the top part of the crystal $\Ycal(\Lambda_{0})$ in type $\Eaff{6}$ in Appendix \ref{E6 top part of Young wall crystal appendix}.

\begin{thm}
    There is an isomorphism of affine crystals $\Ycal(\Lambda_{0}) \xrightarrow{\sim} B(\Lambda_{0})$ sending the ground state wall to the highest weight element $u_{\Lambda_{0}}$.
\end{thm}
\begin{proof}
    From the path realization, we only need to prove that $\Ycal(\Lambda_{0})$ and $\Pcal(\Lambda_{0})$ are isomorphic.
    Define a map $\Psi : \Ycal(\Lambda_{0}) \rightarrow \Pcal(\Lambda_{0})$ by
    \begin{align*}
        (y_{r})_{r=0}^{\infty} \mapsto (\psi^{-1}(y_{r}))_{r=0}^{\infty}
    \end{align*}
    where $\psi : B \xrightarrow{\sim} C$ is as in Propositions \ref{E6 and E7 Young column isomorphisms} and \ref{E8 Young column isomorphism}.
    The ground state wall is clearly sent to the ground state sequence $\pb_{\lambda}$, and by comparing the crystal structure of $\Ycal(\Lambda_{0})$ with (\ref{tensor product of crystals}) it is easy to see that $\Psi$ commutes with the actions of all $\et_{i}$, $\ft_{i}$, $\varepsilon_{i}$, $\varphi_{i}$ and $\wt$.
    \\

    If reduced Young walls $(y_{r})_{r=0}^{\infty}$ and $(z_{r})_{r=0}^{\infty}$ in $\Ycal(\Lambda_{0})$ are mapped by $\Psi$ to the same sequence in $\Pcal(\Lambda_{0})$, then $y_{r}$ and $z_{r}$ must represent the same Young column equivalence class for all $r\in\Nbb$.
    But since both Young walls are reduced, all $|y_{r}| - |y_{r+1}| = |z_{r}| - |z_{r+1}|$ by Proposition \ref{difference in added blocks for reduced adjacent columns}.
    So as $|y_{r}| = |z_{r}| = 0$ for $r\gg 0$ it follows that $|y_{r}| = |z_{r}|$ for all $r\in\Nbb$ and hence $\Psi$ is injective.
    \\

    Given a path $\pb = (p_{r})_{r=0}^{\infty}$ in $\Pcal(\Lambda_{0})$ let $k$ be maximal with $p_{k}\not= b_{k}$.
    We can form a reduced Young wall $(y_{r})_{r=0}^{\infty}$ by first specifying that it matches the ground state wall after column $k$, and then recursively from $r=k$ down to $r=0$ choosing $y_{r}$ so that $(y_{r+1},y_{r})$ is reduced and $\psi^{-1}(y_{r}) = p_{r}$.
    Such a choice exists and is unique by Proposition \ref{difference in added blocks for reduced adjacent columns}.
    Moreover, $\Psi$ maps this Young wall to $\pb$ by construction.
    So $\Psi$ is surjective and our proof is complete.
\end{proof}

\begin{rmk}
    As mentioned in Remark \ref{removing first columns of reduced Young walls remark}, we have analogues of all of these results for the other level $1$ dominant weights $\Lambda_{i}$ simply by removing one or two columns from the Young wall pattern and ground state wall.
    The only other edit is to replace $\Lambda_{0}$ with $\Lambda_{i}$ in the definition of the weight function.
\end{rmk}

\section{Fock space crystals in type E} \label{Fock space section}

Recall the following preliminaries from Section \ref{Fock space preliminaries}.
\begin{itemize}
    \item Starting with a good $\Udash$-module with level $1$ perfect crystal basis $B$, we can construct a Fock space representation $\Fcal(\Lambda_{i})$ of $\Uaff$ for each level $1$ dominant integral weight $\Lambda_{i}\in\Pbar^{+}$.
    \item There is an associated ground state sequence $\sbold_{\Lambda_{i}} = (g_{r})_{r=0}^{\infty} = (z^{m_{r}}b_{r})_{r=0}^{\infty}$ in $\Baff$, where $\pb_{\Lambda_{i}} = (b_{r})_{r=0}^{\infty}$ is the ground state sequence for $B(\Lambda_{i})$ and the $m_{r}\in\Zbb$ are chosen so that all $\Haff(z^{m_{r+1}}b_{r+1}\otimes z^{m_{r}}b_{r}) = 1$.
    \item A sequence $\sbold = (s_{r})_{r=0}^{\infty}$ in $\Baff$ is normally ordered if all $H(s_{r+1}\otimes s_{r})>0$.
    \item The crystal basis $B(\Fcal(\Lambda_{i}))$ consists of the normally ordered sequences $\sbold = (s_{r})_{r=0}^{\infty}$ in $\Baff$ with $s_{r} = g_{r}$ for $r\gg 0$, and has an affine crystal structure given by (\ref{crystal structure on paths}) and (\ref{affine weights on paths}).
\end{itemize}

As in Section \ref{Irreducible highest weight section} we will often denote the crystals $B_{6}$, $B_{7}$ and $B_{8}$ by $B$ and the crystals $C_{6}$, $C_{7}$ and $C_{8}$ by $C$, provided it is clear which type(s) we are referring to.
From the ground state sequences $\pb_{\lambda}$ and energy function values in Table \ref{highest weight ground state sequences table}, together with equation (\ref{Haff definition}), we derive ground state sequences $\sbold_{\lambda}$ in each case.

\begin{table}[H]
\begin{center}
\begin{tabular}{|c|c|c|}
\hline
Type &
\begin{tabular}{@{}c@{}}
$\lambda\in\Pbar^{+}$ \\[1pt] of level $1$
\end{tabular}
&
Ground state sequence $\sbold_{\lambda} = (g_{r})_{r=0}^{\infty}$
\\
\hline
& & \\[-11pt]
\hline
& & \\[-11pt]
$\Eaff{6}$
&
\begin{tabular}{@{}c@{}}
$\Lambda_{0}$ \\[1pt] $\Lambda_{1}$ \\[1pt] $\Lambda_{6}$
\end{tabular}
&
\begin{tabular}{@{}c@{}}
$\dots \otimes z^{-3}\hspace{.1em}\overline{0}1 \otimes z^{-2}\hspace{.1em}\overline{1}6 \otimes z^{-1}\hspace{.1em}\overline{6}0 \otimes z^{-2}\hspace{.1em}\overline{0}1 \otimes z^{-1}\hspace{.1em}\overline{1}6 \otimes z^{0}\hspace{.1em}\overline{6}0 \hspace{.6em}$ \\[1pt]
$\dots \otimes z^{-3}\hspace{.1em}\overline{1}6 \otimes z^{-2}\hspace{.1em}\overline{6}0 \otimes z^{-3}\hspace{.1em}\overline{0}1 \otimes z^{-2}\hspace{.1em}\overline{1}6 \otimes z^{-1}\hspace{.1em}\overline{6}0 \otimes z^{-2}\hspace{.1em}\overline{0}1$ \\[1pt]
$\dots \otimes z^{-2}\hspace{.1em}\overline{6}0 \otimes z^{-3}\hspace{.1em}\overline{0}1 \otimes z^{-2}\hspace{.1em}\overline{1}6 \otimes z^{-1}\hspace{.1em}\overline{6}0 \otimes z^{-2}\hspace{.1em}\overline{0}1 \otimes z^{-1}\hspace{.1em}\overline{1}6$
\end{tabular} \\
\hline
& & \\[-11pt]
$\Eaff{7}$
&
\begin{tabular}{@{}c@{}}
$\Lambda_{0}$ \\[1pt] $\Lambda_{7}$
\end{tabular}
&
\begin{tabular}{@{}c@{}}
$\dots \otimes z^{-4}\hspace{.1em}\overline{0}7 \otimes z^{-2}\hspace{.1em}\overline{7}0 \otimes z^{-3}\hspace{.1em}\overline{0}7 \otimes z^{-1}\hspace{.1em}\overline{7}0 \otimes z^{-2}\hspace{.1em}\overline{0}7 \otimes z^{0}\hspace{.1em}\overline{7}0 \hspace{.6em}$ \\[1pt]
$\dots \otimes z^{-3}\hspace{.1em}\overline{7}0 \otimes z^{-4}\hspace{.1em}\overline{0}7 \otimes z^{-2}\hspace{.1em}\overline{7}0 \otimes z^{-3}\hspace{.1em}\overline{0}7 \otimes z^{-1}\hspace{.1em}\overline{7}0 \otimes z^{-2}\hspace{.1em}\overline{0}7$
\end{tabular} \\
\hline
& & \\[-11pt]
$\Eaff{8}$
&
$\Lambda_{0}$
&
$\dots \otimes \hspace{.25em} z^{-5}\hspace{.1em}\emptyset \hspace{.25em} \otimes \hspace{.25em} z^{-4}\hspace{.1em}\emptyset \hspace{.25em} \otimes \hspace{.25em} z^{-3}\hspace{.1em}\emptyset \hspace{.25em} \otimes \hspace{.25em} z^{-2}\hspace{.1em}\emptyset \hspace{.25em} \otimes \hspace{.25em} z^{-1}\hspace{.1em}\emptyset \hspace{.25em} \otimes \hspace{.25em} z^{0}\hspace{.1em}\emptyset \hspace{.25em} \hspace{.6em}$
\\[1pt]
\hline
\end{tabular}
\caption{\hspace{.5em}Ground state sequences for Fock space crystals}\label{Fock space ground state sequences table}
\end{center}
\end{table}

\begin{rmk}
    As in Remark \ref{removing first columns of reduced Young walls remark}, Young wall models for all other Fock space crystals $B(\Fcal(\Lambda_{i}))$ can be obtained by removing the first one or two columns from the model for $B(\Fcal(\Lambda_{0}))$, and so we only need to consider $B(\Fcal(\Lambda_{0}))$ for the remainder of this section.
\end{rmk}

Recall that Young columns stacked inside the patterns of Figures \ref{E6 and E7 Young column patterns} and \ref{E8 Young column pattern} provide combinatorial models for $\Baff$.
Lining up each Young column $\psi_{\mathrm{aff}}(g_{r})$ at the same height and orientation (so they occupy the same spaces within each vertical strip) we obtain the \emph{Young column pattern} and \emph{ground state wall} for $\Fcal(\Lambda_{0})$ in each type, which are precisely the same as those in Figures \ref{Young wall patterns} and \ref{Ground state walls}.
\\

Young walls in the Fock space situation are defined exactly as in Definition \ref{Young wall definition}.
But here we write them as sequences $(z^{n_{r}}y_{r})_{r=0}^{\infty}$ where all $y_{r}\in C$ and $n_{r}\in\Zbb$, and hence each $z^{n_{r}}y_{r}$ is a Young column (not up to equivalence) in $\Caff$.
Let $\Zcal(\Lambda_{0})$ be the set of Young walls $(z^{n_{r}}y_{r})_{r=0}^{\infty}$ with all $\Haff(z^{n_{r+1}}y_{r+1} \otimes z^{n_{r}}y_{r}) > 0$.
It is clear that every $\sbold = (s_{r})_{r=0}^{\infty}$ in $B(\Fcal(\Lambda_{0}))$ can be represented by an element of $\Zcal(\Lambda_{0})$ since each $s_{r}\in\Baff$ and only finitely many $s_{r}\not= g_{r}$.
\\

Moreover $\Zcal(\Lambda_{0})$ has an affine crystal structure defined exactly as for $\Ycal(\Lambda_{0})$ in Section \ref{Irreducible highest weight section}, using the same notions of pre-$i$-signatures and $i$-signatures:
\begin{itemize}
    \item $\et_{i}$ acts on the column corresponding to the rightmost $-$ in $\sign_{i}(Y)$
    \item $\ft_{i}$ acts on the column corresponding to the leftmost $+$ in $\sign_{i}(Y)$
    \item $\varepsilon_{i}(Y) =$ number of $-$'s in $\sign_{i}(Y)$
    \item $\varphi_{i}(Y) =$ number of $+$'s in $\sign_{i}(Y)$
    \item $\wt(Y) = \Lambda_{0} - \sum_{i\in I} k_{i} \alpha_{i}$
\end{itemize}
where $k_{i}$ is the difference in the number of $i$-blocks between $Y$ and the ground state wall.
A very similar proof to Proposition \ref{reduced walls closed under e and f proposition} shows that $\et_{i}$ and $\ft_{i}$ either map to $0$ or preserve the values of $\Haff(z^{n_{r+1}}y_{r+1} \otimes z^{n_{r}}y_{r})$, and thus $\et_{i},\ft_{i} : \Zcal(\Lambda_{0}) \rightarrow \Zcal(\Lambda_{0})\sqcup\lbrace 0\rbrace$.
Furthermore, the same routine check as for $\Ycal(\Lambda_{0})$ verifies the following.

\begin{thm}
    The maps $\et_{i},\ft_{i} : \Zcal(\Lambda_{0}) \rightarrow \Zcal(\Lambda_{0})\sqcup\lbrace 0\rbrace$, $\varepsilon_{i},\varphi_{i} : \Zcal(\Lambda_{0}) \rightarrow \Zbb\cup\lbrace -\infty\rbrace$ and $\wt : \Zcal(\Lambda_{0}) \rightarrow P$ above endow $\Zcal(\Lambda_{0})$ with the structure of an affine crystal.
\end{thm}

And indeed, this provides a combinatorial Young wall model for the Fock space crystal $B(\Fcal(\Lambda_{0}))$.

\begin{thm}
    There is an isomorphism of affine crystals $\Zcal(\Lambda_{0}) \xrightarrow{\sim} B(\Fcal(\Lambda_{0}))$ sending the ground state wall to the ground state sequence $\sbold_{\Lambda_{0}}$.
\end{thm}
\begin{proof}
    Let $\psi_{\aff} : \Baff \xrightarrow{\sim} \Caff$ be the isomorphism of affine crystals coming from the isomorphism $\psi : B \xrightarrow{\sim} C$ in Proposition \ref{E6 and E7 Young column isomorphisms} or \ref{E8 Young column isomorphism}.
    From the definition of $\Zcal(\Lambda_{0})$ it is clear that
    \begin{align*}
        (s_{r})_{r=0}^{\infty} \mapsto (\psi_{\aff}(s_{r}))_{r=0}^{\infty}
    \end{align*}
    is a bijection from $B(\Fcal(\Lambda_{0}))$ to $\Zcal(\Lambda_{0})$ which sends the ground state sequence $\sbold_{\lambda}$ to the ground state wall.
    Furthermore, by comparing the crystal structure of $\Zcal(\Lambda_{0})$ with (\ref{crystal structure on paths}) and (\ref{affine weights on paths}) it is easy to see that this commutes with all $\et_{i}$, $\ft_{i}$, $\varepsilon_{i}$, $\varphi_{i}$ and $\wt$.
\end{proof}

The next proposition rephrases the normally ordered condition for elements of $B(\Fcal(\Lambda_{0}))$ as a combinatorial condition on the set of Young walls.

\begin{prop}
    A Young wall $Y = (z^{n_{r}}y_{r})_{r=0}^{\infty}$ lies in $\Zcal(\Lambda_{0})$ if and only if
    \begin{align} \label{0-block normally ordered condition}
        |z^{n_{r}}y_{r}|_{0} > |z^{n_{r+1}}y_{r+1}|_{0} - H(y_{r+1} \otimes y_{r}) + m_{r} - m_{r+1}
    \end{align}
    for all $r\in\Nbb$, where $|z^{n_{r}}y_{r}|_{0}$ is the difference in the number of $0$-blocks between $z^{n_{r}}y_{r}$ and column $r$ of the ground state wall.
\end{prop}
\begin{proof}
    Since adding a $0$-block to a Young column corresponds to applying $\ft_{0}$ we see that $|z^{n_{r}}y_{r}|_{0} = m_{r} - n_{r}$.
    Putting this into the normally ordered condition and applying (\ref{Haff definition}) exactly gives condition (\ref{0-block normally ordered condition}).
\end{proof}

Recall that Proposition \ref{component of ground state sequence proposition} embeds $B(\Lambda_{0})$ into $B(\Fcal(\Lambda_{0}))$ as the set of $(s_{r})_{r=0}^{\infty}$ with all $\Haff(s_{r+1} \otimes s_{r}) = 1$.
In the combinatorial language above, since $m_{r} - m_{r+1} = H(b_{r+1} \otimes b_{r}) - 1$ by definition of $\sbold_{\Lambda_{0}}$ this precisely becomes the condition (\ref{reduced definition equation}) for a Young wall to be reduced.
So in the Young wall setting this embedding is simply the identity.
\\

The other interpretation of $B(\Lambda_{0})$ as those $(s_{r})_{r=0}^{\infty} \in B(\Fcal(\Lambda_{0}))$ such that applying $z$ to any entry gives a sequence which no longer lies in $B(\Fcal(\Lambda_{0}))$ can also be seen combinatorially.

\begin{defn}
    A $\delta$-column is a continuous piece of the Young column pattern consisting of $a_{i}$ many $i$-blocks for each $i\in I$.
\end{defn}

By continuous we mean that if any part of some block lies vertically between two blocks inside the $\delta$-column, then it must also be part of the $\delta$-column.
It is clear that removing a $\delta$-column from a Young column precisely corresponds to applying $z$.

\begin{defn}
    A $\delta$-column lying inside a Young wall $Y \in \Zcal(\Lambda_{0})$ is removable if removing it from $Y$ produces another Young wall in $\Zcal(\Lambda_{0})$.
\end{defn}

Therefore $B(\Lambda_{0})$ can further be viewed as those elements of $\Zcal(\Lambda_{0})$ without a removable $\delta$-column.

\subsection{Analysing the structure of Young walls in \texorpdfstring{$\Zcal(\Lambda_{0})$}{Z(Lambda0)}}

The purpose of this subsection is to investigate in more detail the structure of Young walls lying inside the crystal $\Zcal(\Lambda_{0})$.
In particular, we prove the following propositions.

\begin{prop} \label{E6 and E7 Fock right block proposition}
    In types $\Eaff{6}$ and $\Eaff{7}$ each Young wall in $\Zcal(\Lambda_{0})$ must satisfy the right block property (\ref{right block property}).
\end{prop}

A slightly weaker result holds in type $\Eaff{8}$.
Consider the right block property for a pair of adjacent columns in a Young wall:
\begin{itemize}
    \item if the Young wall contains a block in column $r+1$ then it contains the block occupying the same position in column $r$.
    $\hfill \refstepcounter{equation}(\theequation)\label{adjacent right block property}$
\end{itemize}

\begin{prop} \label{E8 Fock right block proposition}
    In type $\Eaff{8}$ a Young wall $(z^{n_{r}}y_{r})_{r=0}^{\infty} \in \Zcal(\Lambda_{0})$ satisfies condition (\ref{adjacent right block property}) on columns $r+1$ and $r$ whenever $\Haff(z^{n_{r+1}}y_{r+1}\otimes z^{n_{r}}y_{r}) \not= 2$.
\end{prop}

In either case, just as for Corollary \ref{highest weight built on ground state wall corollary} we can easily deduce that Young walls in $\Zcal(\Lambda_{0})$ are built on top of the ground state wall, since they differ in only finitely many blocks.
\\

In order to prove Propositions \ref{E6 and E7 Fock right block proposition} and \ref{E8 Fock right block proposition} we fix $r\in\Nbb$ and introduce the following automorphism of $\Caff$ as an \emph{unlabelled digraph}.
Each Young column is viewed inside column $r+1$ of the Young wall pattern, and sent to the column with blocks in the same positions but in column $r$.
For example, the ground state column $\psi_{\aff}(g_{r+1})$ is sent to $\psi_{\aff}(g_{r})$.
By inspecting the Young wall patterns in Figure \ref{Young wall patterns} it is clear that this map is independent of $r$.
\\

Conjugating by $\psi_{\aff}$ gives an automorphism $\sigma$ of $\Baff$ which acts on edge labels by the elements $\pi_{1}$, $\pi_{7}$ and $\pi_{0} = \mathrm{id}$ of $\Omega$ in types $\Eaff{6}$, $\Eaff{7}$ and $\Eaff{8}$ respectively.
\\

While $\sigma$ simply acts by $z$ in type $\Eaff{8}$, to describe $\sigma$ for $\Eaff{6}$ and $\Eaff{7}$ we let $\sigma(z^{n}b) := z^{n+p}c$ for all $z^{n}b\in \Baff$.
It is clear that if $b = \overline{i_{1}}\dots \overline{i_{m}}j_{1}\dots j_{n}$ then $c = \overline{\pi(i_{1})}\dots \overline{\pi(i_{m})}\pi(j_{1})\dots \pi(j_{n})$ where $\pi$ is $\pi_{1}$ and $\pi_{7}$ respectively.
For completeness, we list all values of $p\in\Zbb$ and $c\in B$ in Appendix \ref{Right block appendix}.

\begin{proof}[Proof of Proposition {\upshape\ref{E6 and E7 Fock right block proposition}}]
    It suffices to show that if $\Haff(z^{n}b\otimes z^{m}a) > 0$ then there is a directed path $\sigma(z^{n}b) = z^{n+p}c \rightarrow \cdots \rightarrow z^{m}a$ in $\Baff$, as $\psi_{\aff}(z^{m}a)$ can then be obtained by adding blocks to $\psi_{\aff}(\sigma(z^{n}b))$.
    \\

    Since $\sigma$ commutes with the action of $z$ on $\Baff$ we may without loss of generality take $n=0$, so $m < H(b\otimes a)$ by equation (\ref{Haff definition}).
    And as there is always a path $z^{k+1}a\rightarrow \cdots \rightarrow z^{k}a$ we can further restrict to the case $m = H(b\otimes a) - 1$.
    \\
    
    Lemma \ref{B6 and B7 energy function lemma} tells us that $H(a\otimes c)$ is the minimum number of $0$-arrows in a path $c\rightarrow \cdots \rightarrow a$ in $B$.
    Combining this with the existence of paths $z^{k+1}a\rightarrow \cdots \rightarrow z^{k}a$ it remains to verify that $p + 1 - H(b\otimes a) \geq H(a\otimes c)$, which is a finite check using Table \ref{Sigma tables} in Appendix \ref{Right block appendix}.
\end{proof}

\begin{proof}[Proof of Proposition {\upshape\ref{E8 Fock right block proposition}}]
    Similarly to the proof above, it suffices to show that if $\Haff(z^{n}b\otimes z^{m}a) > 0$ and $\Haff(z^{n}b\otimes z^{m}a) \not= 2$ then there is a directed path $\sigma(z^{n}b) = z^{n+1}b \rightarrow \cdots \rightarrow z^{m}a$ in $\Baff$.
    We may without loss of generality take $n = -1$, whereby $m \leq H(b\otimes a) - 2$ and $m \not= H(b\otimes a) - 3$ by equation (\ref{Haff definition}).
    Furthermore it is clear that there always exist paths $z^{m}a\rightarrow\dots\rightarrow z^{m-2}a$ and $z^{m}a\rightarrow\dots\rightarrow z^{m-3}a$ in $\Baff$, so it is enough to consider $m = H(b\otimes a) - 2$.
    The existence of a path $z^{0}b \rightarrow \cdots \rightarrow z^{H(b\otimes a) - 2}a$ in $\Baff$ is clear for all $b\otimes a$ inside any of the connected components
    \begin{align*}
        \Ccal(\emptyset \otimes \emptyset),~
        \Ccal(\emptyset \otimes \xt),~
        \Ccal(\xt \otimes \emptyset),~
        \Ccal(\xt \otimes \xt),~
        \Ccal(\xt \otimes x_{-\theta}),
    \end{align*}
    and $\Ccal(\xt \otimes y_{8})$ is a quick check using Proposition \ref{component of xt otimes yi}.
    For example, for elements of the form $x_{\theta - \alpha} \otimes x_{-\beta}$ and $x_{\beta} \otimes x_{-\theta + \alpha}$ there exist paths
    \begin{align*}
        &z^{0}x_{\theta - \alpha} \rightarrow z^{0}x_{\alpha_{8}} \rightarrow z^{0}y_{8} \rightarrow z^{0}x_{-\alpha_{8}} \rightarrow z^{0}x_{-\beta} \\
        &z^{0}x_{\beta} \rightarrow z^{0}x_{\alpha_{8}} \rightarrow z^{0}y_{8} \rightarrow z^{0}x_{-\alpha_{8}} \rightarrow z^{0}x_{-\theta + \alpha}
    \end{align*}
    since $\theta - \alpha,\beta\in\Phi_{1}^{+}$.
    So we are left with $\Ccal(\xt \otimes x_{\theta - \alpha_{8}})$ and $\Ccal(\xt \otimes \xb)$ where $\beta$ is the maximal element of $\Phi_{0}^{+}$, which can be checked computationally with the help of SageMath \cite{SageMath} as outlined in Appendix \ref{Right block appendix}.
\end{proof}

In fact, we can say more.
Consider a pair of adjacent columns in some Young wall in $\Zcal(\Lambda_{0})$ in type $\Eaff{8}$ which fails condition (\ref{adjacent right block property}).
From the proof of Proposition \ref{E8 Fock right block proposition} this must be of the form $(z^{n}\psi(b),z^{n + H(b\otimes a) - 1}\psi(a))$ where there does not exist a path $z^{0}b \rightarrow \cdots \rightarrow z^{H(b\otimes a) - 1}a$ in $\Baff$.
Using Lemma \ref{path to itself lemma} and the fact that $x_{-\theta} \otimes x_{\pm \theta} \in \Ccal(\xt \otimes \xt)$ we see that the options for $b\otimes a$ are precisely $\emptyset \otimes \emptyset$, $\emptyset \otimes \xt$, $x_{-\theta} \otimes \emptyset$ and $x_{-\theta} \otimes \xt$.

\begin{lem} \label{path to itself lemma}
    In type $\Eaff{8}$ there is a path $z^{n}a \rightarrow \dots \rightarrow z^{n-1}a$ in $\Baff$ if and only if $a\not= \emptyset,x_{\pm \theta}$.
\end{lem}
\begin{proof}
    For $a = \xa$ with $\alpha\in\Phi_{1}^{+}$ we can take
    \begin{align*}
        z^{n}\xa \rightarrow \dots \rightarrow z^{n}x_{\alpha_{8}} \rightarrow z^{n}y_{8} \rightarrow z^{n}x_{-\alpha_{8}} \xrightarrow{0} z^{n-1}x_{\theta - \alpha_{8}} \rightarrow \dots \rightarrow z^{n-1}\xa
    \end{align*}
    by Lemma \ref{0-arrows lemma}, and similarly for $\alpha\in\Phi_{1}^{-}$.
    If $\alpha\in\Phi_{0}^{\pm}$ this is just a quick check using the crystal graph of $B_{8}$ in Appendix \ref{B8 crystal graph appendix}, while each $a = y_{i}$ follows from the case $a = x_{\alpha_{i}}$.
\end{proof}

It is also important to note that in each type, not every Young wall satisfying the relevant right block property is an element of $\Zcal(\Lambda_{0})$.
Indeed, Figure \ref{Young wall example 2} provides an example of such a wall.
In types $\Eaff{6}$ and $\Eaff{7}$ a pair of adjacent columns $(z^{n}\psi(b),z^{m}\psi(a))$ satisfies condition (\ref{adjacent right block property}) precisely when $m \leq n + p - H(a \otimes c)$.
On the other hand, $\Haff(z^{n}\psi(b)\otimes z^{m}\psi(a)) > 0$ when $m \leq n - 1 + H(b \otimes a)$.
Hence $p + 1 - H(b \otimes a) - H(a \otimes c)$ is the difference in the number of $0$-blocks between the lowest column $z^{m}\psi(a)$ such that $(z^{n}\psi(b),z^{m}\psi(a))$ satisfies condition (\ref{adjacent right block property}), and the lowest such that $(z^{n}\psi(b),z^{m}\psi(a))$ could be adjacent columns in an element of $\Zcal(\Lambda_{0})$.
In type $\Eaff{8}$ this difference is given by $2 - H(b \otimes a)$ plus the minimum number of $0$-arrows in a path $b \rightarrow \dots \rightarrow a$ in $B_{8}$.

\pagebreak

\appendix

\section{The crystal graph of \texorpdfstring{$B_{8}$}{B8}} \label{B8 crystal graph appendix}

\input{E8_level_1_perfect_crystal_page_1}
\pagebreak
\input{E8_level_1_perfect_crystal_page_2}

\pagebreak

\section{Young wall crystals in type \texorpdfstring{$\Eaff{6}$}{E6(1)}} \label{E6 top part of Young wall crystal appendix}

\begin{figure}[H]
\centering

\begin{tikzpicture}[scale=0.11, line width = 0.5]

\input{wall_1}

\begin{scope}[shift={(30,-15)}]
\input{wall_2}
\end{scope}

\begin{scope}[shift={(60,-30)}]
\input{wall_3}
\end{scope}

\begin{scope}[shift={(30,-45)}]
\input{wall_4}
\end{scope}

\begin{scope}[shift={(10,-70)}]
\input{wall_5}
\end{scope}

\begin{scope}[shift={(50,-70)}]
\input{wall_6}
\end{scope}

\begin{scope}[shift={(-10,-95)}]
\input{wall_7}
\end{scope}

\begin{scope}[shift={(30,-95)}]
\input{wall_8}
\end{scope}

\begin{scope}[shift={(70,-95)}]
\input{wall_9}
\end{scope}

\begin{scope}[shift={(-10,-135)}]
\input{wall_10}
\end{scope}

\begin{scope}[shift={(30,-135)}]
\input{wall_11}
\end{scope}

\begin{scope}[shift={(70,-135)}]
\input{wall_12}
\end{scope}

\tikzstyle{every node}=[midway, fill=white]
\draw[col0, line width = 1, -{Stealth[scale=0.9]}] (6,2) -- (20,-5) node{$\scriptstyle 0$};

\draw[col2, line width = 1, -{Stealth[scale=0.9]}] (36,-13) -- (50,-20) node{$\scriptstyle 2$};

\draw[col4, line width = 1, -{Stealth[scale=0.9]}] (50,-28) -- (36,-35) node{$\scriptstyle 4$};

\draw[col5, line width = 1, -{Stealth[scale=0.9]}] (20,-43) -- (12,-52) node{$\scriptstyle 5$};
\draw[col3, line width = 1, -{Stealth[scale=0.9]}] (36,-43) -- (48,-52) node{$\scriptstyle 3$};

\draw[col6, line width = 1, -{Stealth[scale=0.9]}] (0,-68) -- (-8,-77) node{$\scriptstyle 6$};
\draw[col3, line width = 1, -{Stealth[scale=0.9]}] (16,-68) -- (28,-77) node{$\scriptstyle 3$};
\draw[col5, line width = 1, -{Stealth[scale=0.9]}] (40,-68) -- (32,-77) node{$\scriptstyle 5$};
\draw[col1, line width = 1, -{Stealth[scale=0.9]}] (56,-68) -- (68,-77) node{$\scriptstyle 1$};

\draw[col3, line width = 1, -{Stealth[scale=0.9]}] (-10,-101) -- (-10,-117) node{$\scriptstyle 3$};
\draw[col6, line width = 1, -{Stealth[scale=0.9]}] (20,-93) -- (-8,-117) node{$\scriptstyle 6$};
\draw[col4, line width = 1, -{Stealth[scale=0.9]}] (30,-101) -- (30,-113) node{$\scriptstyle 4$};
\draw[col1, line width = 1, -{Stealth[scale=0.9]}] (36,-93) -- (68,-117) node{$\scriptstyle 1$};
\draw[col5, line width = 1, -{Stealth[scale=0.9]}] (70,-101) -- (70,-113) node{$\scriptstyle 5$};

\end{tikzpicture}
    
\caption{\hspace{.5em}The top part of the crystals $\Ycal(\Lambda_{0})$ and $\Zcal(\Lambda_{0})$ in type $\Eaff{6}$}\label{E6 top part of Young wall crystal}
\end{figure}

\pagebreak

\section{The right block property} \label{Right block appendix}

\begin{table}[H]
\begin{center}
\subfloat[Type $E_{6}$]{
\begin{tabular}{ |c|c|c| }
    \hline
    $b$ & $c$ & $p$ \\
    \hline
    & & \\[-11pt]
    \hline & & \\[-11pt]
    $\overline{0}1$ & $\overline{1}6$ & $1$ \\ \hline & & \\[-11pt]
    $\overline{01}3$ & $\overline{16}5$ & $1$ \\ \hline & & \\[-11pt]
    $\overline{03}4$ & $\overline{15}4$ & $1$ \\ \hline & & \\[-11pt]
    $\overline{04}25$ & $\overline{14}23$ & $1$ \\ \hline & & \\[-11pt]
    $\overline{05}26$ & $\overline{12}03$ & $1$ \\ \hline & & \\[-11pt]
    $\overline{2}5$ & $\overline{3}2$ & $1$ \\ \hline & & \\[-11pt]
    $\overline{06}2$ & $\overline{01}3$  & $0$ \\ \hline & & \\[-11pt]
    $\overline{25}46$ & $\overline{23}04$ & $1$ \\ \hline & & \\[-11pt]
    $\overline{26}4$ & $\overline{03}4$  & $0$ \\ \hline & & \\[-11pt]
    $\overline{4}36$ & $\overline{4}05$ & $1$ \\ \hline & & \\[-11pt]
    $\overline{46}35$ & $\overline{04}25$  & $0$ \\ \hline & & \\[-11pt]
    $\overline{3}16$ & $\overline{5}06$ & $1$ \\ \hline & & \\[-11pt]
    $\overline{5}3$ & $\overline{2}5$  & $0$ \\ \hline & & \\[-11pt]
    $\overline{36}15$ & $\overline{05}26$  & $0$ \\ \hline & & \\[-11pt]
    $\overline{1}6$ & $\overline{6}0$ & $1$ \\ \hline & & \\[-11pt]
    $\overline{35}14$ & $\overline{25}46$  & $0$ \\ \hline & & \\[-11pt]
    $\overline{16}5$ & $\overline{06}2$  & $0$ \\ \hline & & \\[-11pt]
    $\overline{4}12$ & $\overline{4}36$ & $0$ \\ \hline & & \\[-11pt]
    $\overline{15}4$ & $\overline{26}4$  & $0$ \\ \hline & & \\[-11pt]
    $\overline{2}01$ & $\overline{3}16$ & $0$ \\ \hline & & \\[-11pt]
    $\overline{14}23$ & $\overline{46}35$ & $0$ \\ \hline & & \\[-11pt]
    $\overline{12}03$ & $\overline{36}15$ & $0$ \\ \hline & & \\[-11pt]
    $\overline{3}2$ & $\overline{5}3$ & $0$ \\ \hline & & \\[-11pt]
    $\overline{23}04$ & $\overline{35}14$ & $0$ \\ \hline & & \\[-11pt]
    $\overline{4}05$ & $\overline{4}12$ & $0$ \\ \hline & & \\[-11pt]
    $\overline{5}06$ & $\overline{2}01$ & $0$ \\ \hline & & \\[-11pt]
    $\overline{6}0$ & $\overline{0}1$  & $-1$ \\ \hline
\end{tabular}
}
\qquad
\subfloat[Type $E_{7}$]{
\begin{tabular}{ |c|c|c| }
\hline
$b$ & $c$ & $p$ \\
\hline
& & \\[-11pt]
\hline & & \\[-11pt]
$\overline{0}7$ & $\overline{7}0$ & $1$ \\ \hline & & \\[-11pt]
$\overline{07}6$ & $\overline{07}1$ & $0$ \\ \hline & & \\[-11pt]
$\overline{06}5$ & $\overline{17}3$ & $0$ \\ \hline & & \\[-11pt]
$\overline{05}4$ & $\overline{37}4$ & $0$ \\ \hline & & \\[-11pt]
$\overline{04}23$ & $\overline{47}25$ & $0$ \\ \hline & & \\[-11pt]
$\overline{03}12$ & $\overline{57}26$ & $0$ \\ \hline & & \\[-11pt]
$\overline{02}3$ & $\overline{27}5$ & $0$ \\ \hline & & \\[-11pt]
$\overline{1}2$ & $\overline{6}2$ & $0$ \\ \hline & & \\[-11pt]
$\overline{023}14$ & $\overline{257}46$ & $0$ \\ \hline & & \\[-11pt]
$\overline{12}4$ & $\overline{26}4$ & $0$ \\ \hline & & \\[-11pt]
$\overline{04}15$ & $\overline{47}36$ & $0$ \\ \hline & & \\[-11pt]
$\overline{14}35$ & $\overline{46}35$ & $0$ \\ \hline & & \\[-11pt]
$\overline{05}16$ & $\overline{37}16$ & $0$ \\ \hline & & \\[-11pt]
$\overline{3}5$ & $\overline{5}3$ & $0$ \\ \hline & & \\[-11pt]
$\overline{15}36$ & $\overline{36}15$ & $0$ \\ \hline & & \\[-11pt]
$\overline{06}17$ & $\overline{17}06$ & $0$ \\ \hline & & \\[-11pt]
$\overline{35}46$ & $\overline{35}14$ & $0$ \\ \hline & & \\[-11pt]
$\overline{16}37$ & $\overline{16}05$ & $0$ \\ \hline & & \\[-11pt]
$\overline{07}1$ & $\overline{07}6$ & $-1$ \\ \hline & & \\[-11pt]
$\overline{4}26$ & $\overline{4}12$ & $0$ \\ \hline & & \\[-11pt]
$\overline{36}47$ & $\overline{15}04$ & $0$ \\ \hline & & \\[-11pt]
$\overline{17}3$ & $\overline{06}5$ & $-1$ \\ \hline & & \\[-11pt]
$\overline{2}6$ & $\overline{2}1$ & $0$ \\ \hline & & \\[-11pt]
$\overline{46}257$ & $\overline{14}023$ & $0$ \\ \hline & & \\[-11pt]
$\overline{37}4$ & $\overline{05}4$ & $-1$ \\ \hline & & \\[-11pt]
$\overline{26}57$ & $\overline{12}03$ & $0$ \\ \hline & & \\[-11pt]
$\overline{5}27$ & $\overline{3}02$ & $0$ \\ \hline & & \\[-11pt]
$\overline{47}25$ & $\overline{04}23$ & $-1$ \\ \hline
\end{tabular}
\quad
\begin{tabular}{ |c|c|c| }
\hline
$b$ & $c$ & $p$ \\
\hline
& & \\[-11pt]
\hline & & \\[-11pt]
$\overline{25}47$ & $\overline{23}04$ & $0$ \\ \hline & & \\[-11pt]
$\overline{27}5$ & $\overline{02}3$ & $-1$ \\ \hline & & \\[-11pt]
$\overline{57}26$ & $\overline{03}12$ & $-1$ \\ \hline & & \\[-11pt]
$\overline{4}37$ & $\overline{4}05$ & $0$ \\ \hline & & \\[-11pt]
$\overline{257}46$ & $\overline{023}14$ & $-1$ \\ \hline & & \\[-11pt]
$\overline{6}2$ & $\overline{1}2$ & $-1$ \\ \hline & & \\[-11pt]
$\overline{3}17$ & $\overline{5}06$ & $0$ \\ \hline & & \\[-11pt]
$\overline{47}36$ & $\overline{04}15$ & $-1$ \\ \hline & & \\[-11pt]
$\overline{26}4$ & $\overline{12}4$ & $-1$ \\ \hline & & \\[-11pt]
$\overline{1}07$ & $\overline{6}07$ & $0$ \\ \hline & & \\[-11pt]
$\overline{37}16$ & $\overline{05}16$ & $-1$ \\ \hline & & \\[-11pt]
$\overline{46}35$ & $\overline{14}35$ & $-1$ \\ \hline & & \\[-11pt]
$\overline{17}06$ & $\overline{06}17$ & $-1$ \\ \hline & & \\[-11pt]
$\overline{36}15$ & $\overline{15}36$ & $-1$ \\ \hline & & \\[-11pt]
$\overline{5}3$ & $\overline{3}5$ & $-1$ \\ \hline & & \\[-11pt]
$\overline{16}05$ & $\overline{16}37$ & $-1$ \\ \hline & & \\[-11pt]
$\overline{35}14$ & $\overline{35}46$ & $-1$ \\ \hline & & \\[-11pt]
$\overline{15}04$ & $\overline{36}47$ & $-1$ \\ \hline & & \\[-11pt]
$\overline{4}12$ & $\overline{4}26$ & $-1$ \\ \hline & & \\[-11pt]
$\overline{14}023$ & $\overline{46}257$ & $-1$ \\ \hline & & \\[-11pt]
$\overline{2}1$ & $\overline{2}6$ & $-1$ \\ \hline & & \\[-11pt]
$\overline{3}02$ & $\overline{5}27$ & $-1$ \\ \hline & & \\[-11pt]
$\overline{12}03$ & $\overline{26}57$ & $-1$ \\ \hline & & \\[-11pt]
$\overline{23}04$ & $\overline{25}47$ & $-1$ \\ \hline & & \\[-11pt]
$\overline{4}05$ & $\overline{4}37$ & $-1$ \\ \hline & & \\[-11pt]
$\overline{5}06$ & $\overline{3}17$ & $-1$ \\ \hline & & \\[-11pt]
$\overline{6}07$ & $\overline{1}07$ & $-1$ \\ \hline & & \\[-11pt]
$\overline{7}0$ & $\overline{0}7$ & $-2$ \\ \hline
\end{tabular}
}
\caption{\hspace{.5em}Describing the function $\sigma : z^{n}b \mapsto z^{n+p}c$ in types $\Eaff{6}$ and $\Eaff{7}$}\label{Sigma tables}
\end{center}
\end{table}

\begin{proof}[Finishing the proof of Proposition {\upshape\ref{E8 Fock right block proposition}}]
As mentioned before, SageMath can be used to show that the Young walls in our crystal $\Zcal(\Lambda_{0})$ in type $\Eaff{8}$ satisfy the weakened right block property of Proposition \ref{E8 Fock right block proposition}.
In particular, we can check that for any $b\otimes a$ in $\Ccal(\xt \otimes x_{\theta - \alpha_{8}})$ or $\Ccal(\xt \otimes \xb)$ there exists a path $z^{0}b \rightarrow \cdots \rightarrow z^{H(b\otimes a) - 2}a$ in $\Baff$.
\\

Indeed, since our level $1$ perfect crystal $B_{8}$ is equal to the crystal basis of the level $0$ fundamental representation $W(\varpi_{8})$ of $\Udash$, the following produces a list of maximal vectors in $B_{8} \otimes B_{8}$.

\begin{lstlisting}
sage: K=crystals.kirillov_reshetikhin.LSPaths(['E',8,1],8)
sage: K2=crystals.TensorProduct(K,K)
sage: hw=K2.classically_highest_weight_vectors()
\end{lstlisting}

The second and sixth entries are $\xt \otimes x_{\theta - \alpha_{8}}$ and $\xt \otimes \xb$ respectively, where we note that SageMath uses an alternative tensor crystal structure in which tensor factors are reversed compared to (\ref{tensor product of crystals}).
Substituting ${\tt n}$ below with ${\tt 1}$ and ${\tt 5}$ therefore gives the components $\Ccal(\xt \otimes x_{\theta - \alpha_{8}})$ and $\Ccal(\xt \otimes \xb)$ respectively.

\begin{lstlisting}
sage: C=K2.subcrystal(generators=[hw[n]],
      index_set=[1,2,3,4,5,6,7,8])
sage: C.digraph().vertices()
\end{lstlisting}

The next step is to weight the edges of $B_{8}$ so that $0$-arrows have weight $1$ and all other arrows have weight $1000$.
To this end, we start by obtaining the list of edges in $B_{8}$.

\begin{lstlisting}
sage: K.digraph().edges()
\end{lstlisting}

With a simple `find and replace' procedure we can turn this into a list {\tt E} of \textit{weighted} edges, where for technical reasons we must also replace any {\tt Lambda[j]} with {\tt Lj}.
The following then defines the desired weighted digraph, and computes the minimal weight of a path between any two vertices.

\begin{lstlisting}
sage: var('L0 L1 L2 L3 L4 L5 L6 L7 L8')
sage: from sage.graphs.base.boost_graph
      import floyd_warshall_shortest_paths
sage: D=DiGraph(E, weighted=True)
sage: floyd_warshall_shortest_paths(D)
\end{lstlisting}

The final digit of the minimal weight of a path from $b$ to $a$ in our weighted digraph is equal to the minimum number of $0$-arrows in a path from $b$ to $a$ in $B$.
Alternatively, this is the minimal $k$ for which there is a path $z^{0}b \rightarrow \cdots \rightarrow z^{-k}a$ in $\Baff$.
\\

For every $b \otimes a \in \Ccal(\xt \otimes \xb)$ this is $0$ as desired.
For $b \otimes a \in \Ccal(\xt \otimes x_{\theta - \alpha_{8}})$ it is either $1$ -- in which case we are done -- or it is $0$.
Since $b\not= \xt,\emptyset$ or $a\not= x_{-\theta},\emptyset$ our proof is complete by Lemma \ref{path to itself lemma}.
\end{proof}

\pagebreak


\addcontentsline{toc}{section}{References}

\begin{bibsection}
\begin{biblist}

\bib{AK97}{article}{
    title={Finite-dimensional representations of quantum affine algebras},
    author={T. Akasaka},
    author={M. Kashiwara},
    journal={Publ. Res. Inst. Math. Sci.},
    volume={33},
    number={5},
    pages={839–-867},
    year={1997},
    note={\url{https://doi.org/10.2977/prims/1195145020}},
}


\bib{BFKL06}{article}{
    title={Level 1 perfect crystals and path realizations of basic representations at q=0},
    author={G. Benkart},
    author={I. Frenkel},
    author={S-J. Kang},
    author={H. Lee},
    journal={Int. Math. Res. Not.},
    volume={2006},
    number={10312},
    pages={1--28},
    year={2006},
    note={\url{https://doi.org/10.1155/IMRN/2006/10312}},
}

\bib{FHKS23}{article}{
    title={Young wall construction of level-1 highest weight crystals over $U_{q}(D_{4}^{(3)})$ and $U_{q}(G_{2}^{(1)})$},
    author={Z. Fan},
    author={S. Han},
    author={S-J. Kang},
    author={Y-S. Shin},
    journal={J. Algebra},
    year={2023},
    note={\url{https://doi.org/10.1016/j.jalgebra.2023.08.001}},
}

\bib{Hiroshima21}{article}{
    title={Perfectness of Kirillov-Reshetikhin Crystals $B^{r, s}$ for types $E_{6}^{(1)}$ and $E_{7}^{(1)}$ with a minuscule node $r$},
    author={T. Hiroshima},
    journal={arXiv preprint},
    year={2021},
    eprint={arXiv:2107.08614},
}

\bib{HKKOT00}{article}{
    title={Finite crystals and paths},
    author={G. Hatayama},
    author={Y. Koga},
    author={A. Kuniba},
    author={M. Okado},
    author={T. Takagi},
    conference={
    title={Combinatorial methods in representation theory},
    address={Kyoto},
    date={1998},
    },
    book={
    series={Adv. Stud. in Pure Math.},
    volume={28},
    publisher={Math. Soc. Jpn.},
    date={2000},
    },
    pages={113--133},
    note={\url{https://doi.org/10.2969/ASPM/02810113}},
}

\bib{HKL04}{article}{
    title={Young wall realization of crystal graphs for $U_{q}(C_{n}^{(1)})$},
    author={J. Hong},
    author={S-J. Kang},
    author={H. Lee},
    journal={Comm. Math. Phys.},
    volume={244},
    number={1},
    pages={111--131},
    year={2004},
    note={\url{https://doi.org/10.1007/s00220-003-0966-6}},
}

\bib{JS10}{article}{
    title={Affine structures and a tableau model for $E_{6}$ crystals},
    author={B. Jones},
    author={A. Schilling},
    journal={J. Algebra},
    volume={324},
    number={9},
    pages={2512--2542},
    year={2010},
    note={\url{https://doi.org/10.1016/j.jalgebra.2010.07.041}},
}

\bib{Kac90}{book}{
	title={Infinite-dimensional Lie algebras},
	author={V. Kac},
	year={1990},
	publisher={Cambridge University Press},
        note={\url{https://doi.org/10.1017/CBO9780511626234}},
}

\bib{Kang03}{article}{
    title={Crystal bases for quantum affine algebras and combinatorics of Young walls},
    author={S-J. Kang},
    journal={Proc. London Math. Soc.},
    volume={86},
    number={1},
    pages={29--69},
    year={2003},
    note={\url{https://doi.org/10.1112/S0024611502013734}},
}

\bib{KKMMNN92}{article}{
    title={Affine crystals and vertex models},
    author={S-J. Kang},
    author={M. Kashiwara},
    author={K. C. Misra},
    author={T. Miwa},
    author={T. Nakashima},
    author={A. Nakayashiki},
    journal={Int. J. Mod. Phys. A},
    volume={7},
    number={supp01a},
    pages={449--484},
    year={1992},
    note={\url{https://doi.org/10.1142/S0217751X92003896}},
}

\bib{KKMMNN92(2)}{article}{
    title={Perfect crystals of quantum affine Lie algebras},
    author={S-J. Kang},
    author={M. Kashiwara},
    author={K. C. Misra},
    author={T. Miwa},
    author={T. Nakashima},
    author={A. Nakayashiki},
    journal={Duke Math. J.},
    volume={68},
    number={3},
    pages={499--607},
    year={1992},
    note={\url{https://doi.org/10.1215/S0012-7094-92-06821-9}},
}

\bib{KK04}{article}{
    title={Crystal bases of the Fock space representations and string functions},
    author={S-J. Kang},
    author={J-H. Kwon},
    journal={J. Algebra},
    volume={280},
    number={1},
    pages={313--349},
    year={2004},
    note={\url{https://doi.org/10.1016/J.JALGEBRA.2004.04.013}},
}

\bib{KK08}{article}{
    title={Fock space representations of quantum affine algebras and generalized Lascoux-Leclerc-Thibon algorithm},
    author={S-J. Kang},
    author={J-H. Kwon},
    journal={J. Korean Math. Soc.},
    volume={45},
    number={4},
    pages={1135--1202},
    year={2008},
    note={\url{https://doi.org/10.4134/JKMS.2008.45.4.1135}},
}

\bib{Kashiwara90}{article}{
    title={Crystalizing the $q$-analogue of universal enveloping algebras},
    author={M. Kashiwara},
    journal={Commun. Math. Phys.},
    volume={133},
    number={2},
    pages={249--260},
    year={1990},
    note={\url{https://doi.org/10.1007/BF02097367}},
}

\bib{Kashiwara91}{article}{
    title={On crystal bases of the q-analogue of universal enveloping algebras},
    author={M. Kashiwara},
    journal={Duke Math. J.},
    volume={63},
    number={2},
    pages={465--516},
    year={1991},
    note={\url{https://doi.org/10.1215/S0012-7094-91-06321-0}},
}

\bib{Kashiwara93}{article}{
    title={Global crystal bases of quantum groups},
    author={M. Kashiwara},
    journal={Duke Math. J.},
    volume={69},
    number={2},
    pages={455--485},
    year={1993},
    note={\url{https://doi.org/10.1215/S0012-7094-93-06920-7}},
}

\bib{Kashiwara94}{article}{
    title={Crystal bases of modified quantized enveloping algebra},
    author={M. Kashiwara},
    journal={Duke Math. J.},
    volume={73},
    number={2},
    pages={383--413},
    year={1994},
    note={\url{https://doi.org/10.1215/S0012-7094-94-07317-1}},
}

\bib{Kashiwara95}{article}{
    title={On crystal bases},
    author={M. Kashiwara},
    conference={
    title={Representations of Groups, Proc. of the 1994 Annual Seminar of the Canadian Math. Soc.},
    address={Banff Center, Banff, Alberta},
    date={1995},
    },
    book={
    address={Providence, RI},
    series={CMS Conf. Proc.},
    volume={16},
    publisher={Amer. Math. Soc.},
    date={1995},
    },
    pages={155--197},
    note={\url{https://doi.org/10.1515/9783110801897.195}},
}

\bib{Kashiwara02}{article}{
    title={On level zero representations of quantized affine algebras},
    author={M. Kashiwara},
    journal={Duke Math. J.},
    volume={112},
    number={1},
    pages={117--175},
    year={2002},
    note={\url{https://doi.org/10.1215/S0012-9074-02-11214-9}},
}

\bib{KMPY96}{article}{
    title={Perfect crystals and q-deformed Fock spaces},
    author={M. Kashiwara},
    author={T. Miwa},
    author={J.-U. H. Petersen},
    author={C. M. Yung},
    journal={Sel. Math. New Ser.},
    volume={2},
    number={3},
    pages={415--499},
    year={1996},
    note={\url{https://doi.org/10.1007/BF01587950}},
}

\bib{KMS95}{article}{
    title={Decomposition of $q$-deformed Fock spaces},
    author={M. Kashiwara},
    author={T. Miwa},
    author={E. Stern},
    journal={Sel. Math. New Ser.},
    volume={1},
    pages={787--805},
    year={1995},
    note={\url{https://doi.org/10.1007/BF01587910}},
}

\bib{KN94}{article}{
    title={Crystal Graphs for Representations of the $q$-Analogue of Classical Lie Algebras},
    author={M. Kashiwara},
    author={T. Nakashima},
    journal={J. Algebra},
    volume={165},
    number={2},
    pages={295--345},
    year={1994},
    note={\url{https://doi.org/10.1006/jabr.1994.1114}},
}

\bib{Lusztig90}{article}{
    title={Canonical bases arising from quantized enveloping algebras},
    author={G. Lusztig},
    journal={J. Amer. Math. Soc.},
    volume={3},
    number={2},
    pages={447--498},
    year={1990},
    note={\url{https://doi.org/10.2307/1990961}},
}

\bib{SageMath}{manual}{
    author={Developers, The~Sage},
    title={{S}agemath, the {S}age {M}athematics {S}oftware {S}ystem ({V}ersion 9.8)},
    date={2023},
    note={\tt{https://www.sagemath.org}. \url{https://doi.org/10.5281/zenodo.593563}}
}

\bib{Stern95}{article}{
    title={Semi-infinite wedges and vertex operators},
    author={E. Stern},
    journal={Int. Math. Res. Not.},
    volume={1995},
    number={4},
    pages={201--220},
    year={1995},
    note={\url{https://doi.org/10.1155/S107379289500016X}},
}



\end{biblist}
\end{bibsection}
\end{document}